%% file: G_infty_ring_spectra_and_Beta_rings.tex
\documentclass[a4paper]{article}

\usepackage[english]{babel}
\usepackage[utf8x]{inputenc}
\usepackage[T1]{fontenc}
\usepackage{lmodern}

\usepackage{geometry}

\usepackage{enumerate}

\usepackage{amsmath,amssymb,amsfonts,amsthm,mathtools,faktor}
\usepackage{tikz}
\usepackage{tikz-cd}
\usetikzlibrary{calc, graphs, patterns, decorations, intersections}
\usetikzlibrary{babel}
\usepackage{graphicx,color}
\usepackage{relsize}
\usepackage[hidelinks]{hyperref}

\usepackage[nottoc]{tocbibind}

\usepackage{fancyhdr}

\newtheorem{thm}{Theorem}[section]
\newtheorem*{thm*}{Theorem}
\theoremstyle{plain}
\newtheorem{prop}[thm]{Proposition}
\newtheorem{lemma}[thm]{Lemma}
\newtheorem{corollary}[thm]{Corollary}

\theoremstyle{definition}
\newtheorem{definition}[thm]{Definition}
\newtheorem*{definition*}{Definition}
\newtheorem{example}[thm]{Example}
\newtheorem{construction}[thm]{Construction}

\theoremstyle{remark}
\newtheorem{remark}[thm]{Remark}

\makeatletter
\let\c@equation=\c@thm

\makeatother

\allowdisplaybreaks

\DeclareMathOperator{\const}{const}
\DeclareMathOperator{\Hom}{Hom}
\DeclareMathOperator{\Ext}{Ext}

\DeclareMathOperator{\pr}{pr}
\DeclareMathOperator{\ev}{ev}
\DeclareMathOperator{\map}{map}
\DeclareMathOperator{\Id}{Id}

\DeclareMathOperator{\Ab}{Ab}
\DeclareMathOperator{\ab}{ab}

\DeclareMathOperator{\ucom}{ucom}
\DeclareMathOperator{\tr}{tr}

\DeclareMathOperator{\res}{res}
\DeclareMathOperator{\incl}{incl}
\DeclareMathOperator{\Ho}{Ho}
\DeclareMathOperator{\gl}{gl}
\DeclareMathOperator{\st}{st}
\DeclareMathOperator{\sh}{sh}
\DeclareMathOperator{\Com}{Com}

\DeclareMathOperator{\Map}{Map}
\DeclareMathOperator{\Cov}{Cov}

\DeclareMathOperator*{\colim}{colim}

\def\eps{\varepsilon}

\def\vcong{\rotatebox[origin=c]{-90}{$\cong$}}
\def\Swarrow{\rotatebox[origin=c]{-135}{$\Rightarrow$}}

\def\vertical{\boxminus}
\def\horizontal{\mathbin{\mathpalette\Verticalversion\relax }}
\newcommand{\Verticalversion}[1]{\rotatebox[origin=c]{-90}{$#1\vertical$}}

\def\inverse{^{-1}}

\def\GF{\mathcal{GF}}
\def\GH{\mathcal{GH}}
\def\SH{\mathcal{SH}}
\def\Green{\mathcal Gl\mathcal Green}

\newcommand{\nameto}[1]{\xrightarrow{#1}}

\newcommand\myatop[2]{\genfrac{}{}{0pt}{}{#1}{#2}}
\let\Oldsquare\square 
\renewcommand{\square}{\mathbin{\Oldsquare}}

\newcommand{\adjunction}[4]{\begin{tikzcd}[cramped, ampersand replacement=\&] #1 \arrow[r, shift left,"{#3}"] \& #2 \arrow[l, shift left, "{#4}"]  \end{tikzcd}}
\newcommand{\outlineadjunction}[5][large]{\begin{tikzcd}[column sep=#1, ampersand replacement=\&] #2 \arrow[r, shift left,"{#4}"] \& #3 \arrow[l, shift left, "{#5}"]  \end{tikzcd}}

\author{Michael Stahlhauer}
\date{\today}
\title{$G_\infty$-ring spectra and Moore spectra for $\beta$-rings}

\begin{document}	

\maketitle
\thispagestyle{fancy}
\lfoot{2020 Mathematics Subject Classification 55P43, 55P91 (primary), 55S91, 18C15, 19A22 (secondary). Please refer to {\tt http://www.ams.org/msc/} for a list of codes.\\[0.5\baselineskip]	
The research in this paper was supported by the DFG Schwerpunktprogramm 1786 Homotopy Theory and Algebraic Geometry (GZ SCHW 860/1-1) and the Max-Planck-Institut für Mathematik Bonn. The author is an Associate Member of the Hausdorff Center for Mathematics at the University of Bonn (DFG GZ 2047/1, project ID 390685813).}
\cfoot{}
\renewcommand{\headrulewidth}{0pt}
\renewcommand{\footrulewidth}{0.4pt}

\begin{abstract}
In this paper, we introduce the notion of $G_\infty$-ring spectra. These are globally equivariant homotopy types with a structured multiplication, giving rise to power operations on their equivariant homotopy and cohomology groups. We illustrate this structure by analysing when a Moore spectrum can be endowed with a $G_\infty$-ring structure. Such $G_\infty$-structures correspond to power operations on the underlying ring, indexed by the Burnside ring. We exhibit a close relation between these globally equivariant power operations and the structure of a $\beta$-ring, thus providing a new perspective on the theory of $\beta$-rings.
\end{abstract} 

\tableofcontents

\input{Introduction}

\input{Chapter_External_Power_Operations}

\input{Chapter_Definition_G_infty}

\input{Chapter_Moore_spectra}

\appendix
\input{Appendix_Transferring_Monads}

\clearpage
\bibliographystyle{hplain}
\bibliography{Literatur}

\end{document}

%% file: Introduction.tex
\section*{Introduction}
\addcontentsline{toc}{section}{Introduction}

The aim of this article is the introduction of the new notion of $G_\infty$-ring spectra. These support power operations on their equivariant homotopy groups and cohomology with coefficients in such spectra. We moreover provide an algebraic description of this notion on Moore spectra, linking $G_\infty$-ring structures on a Moore spectrum to $\beta$-ring structures on the represented ring.

Algebraic invariants are more useful the more structure they are endowed with. One example of this slogan are power operations on cohomology. The Steenrod operations on mod-$p$ cohomology, the Adams operations on $K$-theory and power operations on complex bordism and stable cohomotopy all carry a lot of additional information and have seen extensive use in classical homotopy theory. More recently, Hill, Hopkins and Ravenel used equivariant power operations in the guise of norm maps to prove the non-existence of elements of Kervaire invariant one in \cite{HHR_2016}. This work renewed interest in both multiplicative aspects of homotopy theory and equivariant techniques.

Classical power operations in cohomology arise from an $H_\infty$-ring structure on the representing spectrum, as defined by Bruner, May, McClure and Steinberger in \cite{BMMS_1986}. In the present work, we generalize the notion of an $H_\infty$-ring spectrum to a globally equivariant context, in order to represent equivariant power operations. Here, globally equivariant means that we encode compatible actions by all compact Lie groups, using the framework provided by Schwede in \cite{Schwede_2018}. Thus, a $G_\infty$-ring spectrum encodes power operations on equivariant cohomology groups for all compact Lie groups. Hence, the notion of a $G_\infty$-ring spectrum relates to the stricter notion of an ultra-commutative ring spectrum as an $H_\infty$-ring structure relates to an $E_\infty$-ring spectrum. This is visualized in the following diagram, which exhibits forgetful functors between the corresponding homotopy categories:
\[\begin{tikzcd}
\textrm{Ultra-commutative ring spectra} \rar \dar & G_\infty \textrm{-ring spectra} \dar \\
E_\infty \textrm{-ring spectra} \rar & H_\infty \textrm{-ring spectra}
\end{tikzcd}\]

Algebraically, power operations can be packaged in different ways. The Adams operations on $K$-theory endow it with the structure of a $\lambda$-ring, and the power operations on stable cohomotopy give it the structure of a $\beta$-ring. Among these, $\lambda$-rings are better-behaved and are widely studied in algebraic topology and representation theory, e.g. \cite{Hoffman_1979, Atiyah_Tall_1969, Knutson_1973}. On the other hand, the theory of $\beta$-rings is still largely mysterious, with different definitions and many subtleties not present in the study of $\lambda$-rings, see e.g. \cite{Rymer_1977, Ochoa_1988, Vallejo_1993}. In this paper, we present a different approach to the notion of $\beta$-rings, coming from a well-structured theory of global power operations, where the question of scalar extensions of the Burnside ring global power functor naturally leads to considering $\beta$-rings. In this way, $G_\infty$-ring structures on Moore spectra, which yield scalar extensions of the global power operations on the sphere spectrum, are intimately tied to $\beta$-ring structures. Moreover, we can also obtain the $\beta$-ring structure on stable cohomotopy $\pi^0(X)$ in this way. We believe that this new approach to $\beta$-rings clarifies the structure of power operations indexed by the Burnside ring, and underlines that the notion of a global power functor is more fundamental than the notion of $\beta$-rings.

\paragraph{Results}
In the first part of this work, we introduce the notion of a $G_\infty$-ring spectrum. This is a derived version of a structured ring spectrum. In contrast to e.g. an $E_\infty$-ring spectrum, the definition is at the level of the homotopy category. Concretely, we take the free commutative algebra monad $\mathbb P\colon \mathcal Sp\to \mathcal Sp$ at the level of spectra. This functor is left derivable for the positive global model structure and thus induces a monad on the global homotopy category $\GH$.
\begin{definition*}[Definition \ref{def:G_infty_ring_spectra}]
A $G_\infty$-ring spectrum is an algebra over the monad $\mathbb G= L\mathbb P$.
\end{definition*}
We then study properties of $G_\infty$-ring spectra. As mentioned above, the main property of a $G_\infty$-ring spectrum is that it supports power operations on its equivariant homotopy groups:
\begin{thm*}[Proposition \ref{prop:power_operations_on_G_infty_spectra}]
Let $E$ be a $G_\infty$-ring spectrum. The structure map $\mathbb GE\to E$ defines the structure of a global power functor on $\underline{\pi}_0(E)$.
\end{thm*}
To prove this proposition, we construct external power operations $\pi_0^G(E)\to \pi_0^{\Sigma_m\wr G}(\mathbb GE)$, which exist for any global spectrum $E$. In presence of a $G_\infty$-structure, the power operations on the homotopy groups are then obtained by postcomposition with the multiplication.\\
Any ultra-commutative ring spectrum induces a $G_\infty$-ring structure on its homotopy type. However, there are also examples of $G_\infty$-ring spectra which are not induced from any strictly commutative multiplication. These are constructed by means of an adjunction
\[ \outlineadjunction{G_\infty\textup{-Rings}}{H_\infty\textup{-Rings},}{U}{R} \]
where $U$ is the forgetful functor from $G_\infty$-rings to $H_\infty$-rings and $R$ does not change the underlying $H_\infty$-ring spectrum. Applying the right adjoint $R$ to an $H_\infty$-ring spectrum $E$ which cannot be rigidified to a strict commutative ring spectrum, we see that $RE$ is not the homotopy type of an ultra-commutative ring spectrum.

In the second part of this paper, we analyse the structure of a $G_\infty$-multiplication on Moore spectra. We obtain the following result, which characterizes the $G_\infty$-ring structures on global Moore spectra via purely algebraic data:
\begin{thm*}[Theorem \ref{thm:equivalence_G_Moore_spectra_A_power_algebras}]
The functor
\[\underline{\pi}_0\colon G_\infty\textit{-Moore}^{\textup{torsion-free}} \to {\mathcal GlPow}_{\mathrm{left}}^{\textup{torsion-free}}\]
is an equivalence of categories between the homotopy category of $G_\infty$-Moore spectra for countable torsion-free rings and the category of countable torsion-free left-induced global power functors.
\end{thm*}
Here, a left induced global power functor is one where the multiplication map $\mathbb A\otimes R(e)\to R$ is an isomorphism of global Green functors, where $\mathbb A$ denotes the Burnside ring global power functor. The restriction to torsion-free rings is necessary, since already the existence of multiplications on Moore spectra in the presence of torsion is a subtle question. Countability is a mild technical assumption we use in order to construct such Moore spectra. The above theorem shows that whenever multiplications on Moore spectra are tractable, then also the power operations are completely determined by algebraic power operations on the represented ring. This relation can be used to prove that neither the Moore spectra $\mathbb S(\mathbb Z/p)$ nor $\mathbb S(\mathbb Z[i])$ can be endowed with a $G_\infty$-ring structure. This is a shadow of the classical results that these spectra do not support an $A_\infty$- or $E_\infty$-structure respectively. However, the $G_\infty$-result can be obtained by elementary calculations of the power operations in the Burnside ring.

Moreover, the theory of left-induced global power functors is closely linked to the theory of $\beta$-rings. In fact, we prove the following theorem, which allows to induce $\beta$-ring structures from global power operations.
\begin{thm*}[Theorem \ref{thm:powered_algebras_yield_tau_rings}]
The assignment $(G, R)\mapsto R(G)$ extends to a functor
\[ \ev\colon \textup{Rep}^{\operatorname{op}} \times {\mathcal GlPow}_{\mathbb A\textup{-defl}} \to \beta\textup{-Rings}, \]
which sends a conjugacy class of a morphism of compact Lie groups to the corresponding restriction.
\end{thm*}
Here, $\textup{Rep}$ is the category of compact Lie groups and conjugacy classes of continuous homomorphisms between these, and ${\mathcal GlPow}_{\mathbb A\textup{-defl}}$ denotes the category of global power functors equipped with deflation maps $R(K\times G)\otimes \mathbb A(K)\to R(G)$. This includes all left-induced global power functors. In particular, this shows that for a Moore spectrum $\mathbb SB$ that supports a $G_\infty$-ring structure, all equivariant homotopy groups $\pi_0^G(\mathbb SB)\cong \mathbb A(G)\otimes B$ come endowed with the structure of a $\beta$-ring. Moreover, this condition also includes (equivariant) stable cohomotopy $\pi^0_G(X)$ for a based space $X$, hence we reprove the fact that stable cohomotopy can be equipped with the structure of a $\beta$-ring from \cite{Guillot_2006}.

The above theorem shows that the notion of a global power functor with $\mathbb A$-deflations captures the structure of a ring supporting $\beta$-ring structures at all compact Lie groups at once. Hence, it provides a unifying point of view to the theory of $\beta$-rings, which has proven rather hard to understand.

\paragraph{Structure}

In Section 1, we recall orthogonal spectra as a model for global homotopy theory, and the multiplicative aspects leading to power operations. Moreover, we give a construction of the external power operations on the equivariant homotopy groups of any spectrum.\\
In Section 2, we define $G_\infty$-ring spectra and see that they support power operations on their homotopy groups. We also compare this new notion to classical $H_\infty$-ring spectra by means of an adjunction featuring the forgetful functor, and provide a homotopical comparison of the derived symmetric algebra monad $\mathbb G=L\mathbb P$ and extended symmetric powers $\Sigma^\infty_+(E_{\gl}\Sigma_m)\wedge_{\Sigma_m} X^{\wedge m}$.\\
In the last section, we analyse $G_\infty$-ring Moore spectra. On the topological side, we prove that power operations on $\mathbb A\otimes B$ are equivalent to a $G_\infty$-ring structure on $\mathbb SB$ for torsion-free $B$. On the algebraic side, we prove that such power operations give rise to $\beta$-ring structures on $\mathbb A(G)\otimes B$ for all compact Lie groups G. This comparison also yields $\beta$-operations on stable cohomotopy groups $\pi^0(X)$.\\
In the appendix,  we collect results from the theory of monads used throughout our work. In particular, we study under which 2- and double categorical functors monads and monad functors are preserved. We utilize these results when constructing the adjunction between $G_\infty$- and $H_\infty$-ring spectra.

\paragraph{Acknowledgements}

This paper is a revised version of a Master's thesis written at the University of Bonn. I would like to thank my supervisor Stefan Schwede for suggesting the topic of $G_\infty$-ring spectra and $\beta$-rings to me and for his constant support and encouragement.\\
I would also like to thank Markus Hausmann for helpful remarks, Irakli Patchkoria for answering questions on Moore spectra and Jack Davies for interesting conversations on global power functors. I also thank the anonymous referee for helpful comments, leading to several improvements of the exposition.

%% file: Chapter_External_Power_Operations.tex
\section{Power operations on ultra-commutative ring spectra}

In this chapter, we give an introduction to power operations on the global homotopy groups of an ultra-commutative ring spectrum, and construct external power operations for any orthogonal spectrum. We work throughout this article in the context of global homotopy theory, where the adjective `global' indicates that we study equivariant spectra for all compact Lie groups at once. As a model, we use the model category of global orthogonal spectra provided by Stefan Schwede in \cite{Schwede_2018}, and we use this work as a general referencing point for the foundations of global homotopy theory.

The category of orthogonal spectra is endowed with a symmetric monoidal structure, and commutative monoids are called ultra-commutative ring spectra. These spectra support power operations on their global homotopy groups. Such power operations are important additional structure, as emphasized most prominently in the work of Hill, Hopkins and Ravenel on the non-existence of elements of Kervaire invariant one in \cite{HHR_2016}. We recall the multiplicative aspects of global homotopy theory and the formalism of power operations.

Moreover, we introduce external power operations in Construction \ref{constr:external_power_operations}. These are defined on the homotopy groups of any orthogonal spectrum $X$, but only take values in the homotopy groups of the symmetric powers $\mathbb P^mX$. For an ultra-commutative ring spectrum, the structure morphism $\mathbb PX\to X$ then recovers the usual power operations. These external power operations are used in Section \ref{section:G_infty_rings} to define power operations on the homotopy groups of any $G_\infty$-ring spectrum.

\subsubsection*{Orthogonal spectra and homotopy groups}

We assume familiarity with the context of global orthogonal spectra and ultra-commutative ring spectra, as established in \cite[Chapters 3-5]{Schwede_2018}. We quickly collect the relevant notions.

We denote by $\mathcal Sp$ the category of orthogonal spectra, see also \cite{MMSS_2001}. For any compact Lie group $G$ and integer $k$, we can associate to an orthogonal spectrum $X$ its $k$-th $G$-equivariant homotopy group $\pi_k^G(X)$ \cite[3.1.11]{Schwede_2018}. A morphism $f\colon X\to Y$ of orthogonal spectra is called a {\itshape global equivalence} if it induces isomorphisms $f_\ast \colon \pi_k^G(X)\to \pi_k^G(Y)$ for all compact Lie groups $G$ and all $k$. The resulting homotopy category obtained by inverting global equivalences is called the {\itshape global homotopy category} $\GH$.

The collection of homotopy groups $\underline{\pi}_0(X)=\{\pi_0^G(X)\}_G$ for any orthogonal spectrum $X$ comes equipped with {\itshape restriction maps} $\alpha^\ast \colon \pi_0^K(X) \to \pi_0^G(X)$ for any continuous homomorphism $\alpha\colon G\to K$ of compact Lie groups \cite[Construction 3.1.15]{Schwede_2018}, and with {\itshape transfer maps} $\tr_H^G\colon \pi_0^H(X) \to \pi_0^G(X)$ for any closed subgroup $H\subset G$ \cite[Construction 3.2.22]{Schwede_2018}, which are trivial if the Weyl group $W_GH$ is infinite. These morphisms endow $\underline{\pi}_0(X)$ with the structure of a {\itshape global functor} \cite[Definition 4.2.2]{Schwede_2018}, and the category of global functors is denoted $\GF$. It is defined as the category of additive functors from the global Burnside category $\mathbf A$ \cite[Construction 4.2.1]{Schwede_2018} to abelian groups. Here the Burnside category has as objects the compact Lie groups and the morphisms are generated by restrictions and transfers \cite[Proposition 4.2.5]{Schwede_2018}. Composition inside $\mathbf A$ contains the information about the compositions of transfers and restrictions, such as the {\itshape double coset formula} \cite[Theorem 3.4.9]{Schwede_2018}.

Note that the notion of a global functor is a global version of a Mackey functor for a fixed compact Lie group $G$. The adjective `global' refers to the fact that we allow restrictions along arbitrary group homomorphisms, not just inclusions. There are various different global versions of Mackey functors, and we refer to \cite[Remark 4.2.16]{Schwede_2018} for a discussion of the different definitions.

\begin{remark}\label{remark:global_classifying_spaces}
The global homotopy groups for a fixed compact Lie group $G$ can be considered as a functor
\[ \GH \to \mathcal Sets,\, X \mapsto \pi_0^G(X).\]
This functor is representable by the suspension spectrum of the global classifying space $B_{\gl}G$ by \cite[Theorem 4.4.3 i)]{Schwede_2018}. The orthogonal space $B_{\gl}G$ is a generalization of the classical space $BG$, and we recall its construction now:

For two inner product spaces $V$ and $W$, we denote with $\mathbf L(V,W)$ the space of linear isometric embeddings from $V$ into $W$. Let $G$ be a compact Lie group. We define, for any $G$-representation $V$, the orthogonal $G$-space $\mathbf L_V= \mathbf L(V,\_)$ with the right $G$-action by precomposition with the $G$-action on $V$, and the orthogonal space $\mathbf L_{G,V}= \mathbf L(V,\_)/G$. By \cite[Proposition 1.1.26]{Schwede_2018}, the global homotopy types of $\mathbf L_V$ and $\mathbf L_{G,V}$ are independent of the choice of the $G$-representation $V$ as long as $V$ is faithful. We then denote for any faithful $V$ the orthogonal $G$-spaces $E_{\gl} G= \mathbf L_V$ and $B_{\gl}G= \mathbf L_{G,V}$.

We have a stable tautological class $e_G\in \pi_0^G(\Sigma^\infty_+ B_{\gl}G)$ as defined in \cite[4.1.12]{Schwede_2018}, and the pair $(\Sigma^\infty_+ B_{\gl}G, e_G)$ represents $\pi_0^G$ in the sense that
\begin{align}
\begin{split}\label{eq:representability_of_pi_0}
\GH( \Sigma^\infty_+ B_{\gl}G, X) & \to \pi_0^G(X)\\
[f]& \mapsto f_{\ast}(e_G)
\end{split}
\end{align}
is a bijection for every orthogonal spectrum $X$.

We can moreover describe the suspension spectrum of $B_{\gl}G$ as follows (see \cite[3.1.2]{Schwede_2018}): Choose a faithful $G$-representation $V$, such that $L_{G,V}$ represents $B_{\gl}G$. Then, for any inner product space $W$, we define a homeomorphism 
\begin{equation}\label{eq:untwisting_iso}
S^W\wedge \mathbf L(V,W)_+ \cong \mathbf O(V,W) \wedge S^V,
\end{equation}
called the untwisting isomorphism, using that we can trivialize the orthogonal complement bundle over $\mathbf L(V,W)$ with an additional copy of $V$. These homeomorphisms descend to $G$-orbits and assemble for varying $W$ into an isomorphism
\begin{equation}\label{eq:suspension_spectrum_of_BG}
\Sigma^\infty_+ B_{\gl}G \to \mathbf O(V,\_) \wedge_G S^V =: F_{G,V}S^V.
\end{equation}
\end{remark}

\subsubsection*{Global Green and power functors}

We now give more details about the multiplicative structure on orthogonal spectra and homotopy groups.\\
The category of orthogonal spectra has a symmetric monoidal structure using the smash product as defined in \cite[Definition 3.5.1]{Schwede_2018}. Using this symmetric monoidal structure, we can define a symmetric algebra monad $\mathbb PX= \bigvee_{m\geq 0}\mathbb P^mX$ with $\mathbb P^mX = X^{\wedge m}/\Sigma_m$ for orthogonal spectra $X$ and a notion of commutative monoids in this category.

\begin{definition}
	An ultra-commutative ring spectrum is a commutative monoid in the category $\mathcal Sp$ of orthogonal spectra. We write $\ucom$ for the category of ultra-commutative ring spectra and multiplicative maps.
\end{definition}

Note that the categories of $\mathbb P$-algebras and of ultra-commutative ring spectra are isomorphic.\\
This strict multiplication induces a rich structure on the homotopy groups of ultra-commutative ring spectra. Besides endowing the homotopy groups with the structure of a commutative monoid in the category of global functors, it also induces power operations. We now give a short recollection on how these multiplications and power operations are constructed.

The category of global functors has a symmetric monoidal structure, called the box product, arising as a Day convolution product, compare \cite[Construction 4.2.27]{Schwede_2018}.

\begin{definition}\label{def:green_functor}
	A global Green functor is a commutative monoid in the category $\GF$ endowed with the symmetric monoidal structure provided by the box product. A morphism of global Green functors is a morphism of global functors compatible with the multiplication.
\end{definition}

Explicitly, the multiplication map $R\square R\to R$ of a global Green functor is equivalent both to a family of multiplication maps 
\[ \times\colon R(G)\times R(K) \to R(G\times K)\]
for all compact Lie groups $G$ and $K$, and to a family of diagonal products 
\begin{equation}\label{eq:diagonal_product_Green_functor}
\cdot\colon R(G) \times R(G) \to R(G)
\end{equation}
for all compact Lie groups $G$. These multiplications have to satisfy the properties explained after \cite[Definition 5.1.3]{Schwede_2018}. The relationship between these formulations via restrictions along diagonal and projections is elaborated upon in \cite[Remark 4.2.20]{Schwede_2018}.

Classically, it is well known that the equivariant homotopy groups of a homotopy commutative $G$-ring spectrum support the structure of a Green functor. The same is true here.

\begin{prop}\label{prop:homotopy_for_ring_spectrum_Green_functor}
	There is an external multiplication map 
	\begin{equation}\label{eq:external_pairing_homotopy_groups}
	\boxtimes\colon \pi_0^G(X)\times \pi_0^K(Y)\to \pi_0^{G\times K}(X\wedge Y),
	\end{equation}
	defined for any two orthogonal spectra $X$ and $Y$. Upon composition with the induced map of the multiplication $E\wedge E\to E$ for a homotopy commutative ring spectrum $E$, this defines the structure of a global Green functor on $\underline{\pi}_0(E)$.
\end{prop}
\begin{proof}
	This statement is already contained in the discussion of multiplications on global functors from \cite{Schwede_2018}. The definition of the external multiplication map is given in \cite[Construction 4.1.20]{Schwede_2018}, and its properties are listed in \cite[Theorem 4.1.22]{Schwede_2018}. These, together with the properties of the multiplication on $E$, imply that $\underline{\pi}_0(E)$ with this multiplication is indeed a global Green functor.
\end{proof}

\begin{remark}\label{remark:recall_wreath_products}
To define the power operations induced by an ultra-commutative ring spectrum, we recall the wreath product $\Sigma_m\wr G$ of the symmetric group $\Sigma_m$ on $m$ letters with a group $G$. We refer to \cite[Construction 2.2.3]{Schwede_2018} for details. The wreath product is defined as the semidirect product $\Sigma_m\wr G= \Sigma_m \ltimes G^m$ with respect to the permutation action of the symmetric group on the factors of $G^m$. This has a natural action on the $m$-th power $A^m$ of a $G$-set $A$, given by the $G$-action on each factor and the permutation action of $\Sigma_m$ on the factors. In particular, if $V$ is a $G$-representation, then $V^m$ is a $\Sigma_m\wr G$-representation, which is faithful if $V$ is a non-zero faithful $G$-representation.
\end{remark}

Now we can define the power operations as follows: Let $E$ be an ultra-commutative ring spectrum, then we define for $f\colon S^V\to E(V)$ in $\pi_0^G(E)$, where $V$ is some $G$-representation, the $(\Sigma_m\wr G)$-map 
\begin{equation}\label{eq:power_operations_for_ucom}
P^m(f)\colon S^{V^m}\cong (S^V)^{\wedge m} \nameto{f^m} E(V)^{\wedge m} \nameto{\mu_{V,\ldots, V}} E(V^m),
\end{equation}
which represents an element in $\pi_0^{\Sigma_m\wr G} (E)$. The map $\mu_{V,\ldots ,V}$ is the value of the $m$-fold multiplication map $E^{\wedge m}\to E$ on the inner product space $V^m$. This assignment defines a morphism
\[P^m\colon \pi_0^G(E) \to \pi_0^{\Sigma_m\wr G} (E)\]
for all $m\geq 1$.\\
These power operations satisfy equivariant versions of the properties of powers in a ring, and compatibility results with the restriction and transfer maps present on the global homotopy groups. These properties are condensed into the notion of a global power functor, defined in \cite[Definition 5.1.6]{Schwede_2018}.

\begin{thm}\label{thm:power_operations_for_ucom}
For an ultra-commutative ring spectrum $E$, the global homotopy groups $\underline{\pi}_0(E)$ together with the operations 
\[P^m\colon \pi_0^G(E) \to \pi_0^{\Sigma_m\wr G} (E)\]
as defined in (\ref{eq:power_operations_for_ucom}) form a global power functor.
\end{thm}
For the proof we refer to \cite[Theorem 5.1.11]{Schwede_2018}.

\subsubsection*{External power operations}

We see that a strict multiplication on an orthogonal spectrum gives rise to power operations on the global homotopy groups. We now ask the question whether we need the full structure of an ultra-commutative ring spectrum to obtain these power operations. In classical homotopy theory, we have the notion of $H_\infty$-ring spectra from \cite{BMMS_1986}, which is only defined in the stable homotopy category. Such $H_\infty$-ring spectra also define power operations on their homotopy and cohomology groups. We define a global analogue in this paper. To facilitate the proof that the global homotopy groups of a $G_\infty$-ring spectrum support power operations, we introduce external power operations.

The multiplication on the homotopy groups of a homotopy commutative ring spectrum can be constructed by defining an external multiplication as described in \eqref{eq:external_pairing_homotopy_groups}, which exists for any two orthogonal spectra. Upon composition with the homotopy multiplication, this gives the structure of a global Green functor on the homotopy groups. We now define a similar external power operation for any orthogonal spectrum.

\begin{construction}\label{constr:external_power_operations}
	Let $X$ be an orthogonal spectrum. We set $\mathbb P^mX= X^{\wedge m}/ \Sigma_m$ and define maps 
	\[ \hat{P}^m\colon \pi_0^G(X)\to \pi_0^{\Sigma_m\wr G} ( \mathbb P^m X)\]
	for every compact Lie group $G$ and $m\geq 1$ as follows:
	For a $G$-representation $V$ and a $G$-equivariant map $f\colon S^V\to X(V)$ representing an element in $\pi_0^G(X)$, we set
	\[\hat{P}^m (f) \colon S^{V^m}\cong (S^V)^{\wedge m} \nameto{f^m} X(V)^{\wedge m} \nameto{i_{V,\ldots ,V}} X^{\wedge m} (V^m) \nameto{\pr(V^m)} (\mathbb P^m X)(V^m),\]
	where the morphism $i_{V,\ldots,V}$ is the iteration of the universal bimorphism from the definition of the smash product (see \cite[Definition 3.5.1]{Schwede_2018}), and $\pr\colon X^{\wedge m}\to X^{\wedge m}/\Sigma_m$ is the projection.\\
	We claim that this map is $\Sigma_m\wr G$-equivariant: To see this, let \[(\sigma;g_\ast)=(\sigma;g_1,\ldots, g_m)\in \Sigma_m\wr G\]
	be an element of the wreath product. Then we consider the following diagram, where $\sigma\cdot (\_\,)$ signifies the $\Sigma_m$-action by permutation:
	\[\begin{tikzcd}
	(S^V)^{\wedge m} \arrow[r, "f^{\wedge m}"] \arrow[d, "g_\ast"] \arrow[dd, bend right=50, "{(\sigma; g_\ast)}", swap]
	& X(V)^{\wedge m} \arrow[r, "i_{V, \ldots, V}"] \arrow[d, "g_\ast"]
	& X^{\wedge m} (V^m) \arrow[r] \arrow[d, "g_\ast"]
	& (\mathbb P^m X) (V^m) \arrow[d, "g_\ast"] \arrow[dd, bend left=50, shift left = 4, "{(\sigma; g_\ast)}"]\\
	(S^V)^{\wedge m} \arrow[r, "f^{\wedge m}"] \arrow[d, "\sigma"]
	& X(V)^{\wedge m} \arrow[r, "i_{V, \ldots, V}"] \arrow[d, "\sigma"]
	& X^{\wedge m} (V^m) \arrow[r] \arrow[d, "\sigma"]
	& (\mathbb P^m X) (V^m) \arrow[d, "\sigma"]	\\
	S^{\sigma\cdot V^m} =\sigma\cdot (S^V)^{\wedge m} \arrow[r, "f^{\wedge m}"]
	& \sigma\cdot X(V)^{\wedge m} \arrow[r, "i_{V, \ldots, V}"]
	& (\sigma \cdot X^{\wedge m}) (\sigma\cdot V^m) \arrow[r]
	& (\mathbb P^m X) (\sigma\cdot V^m)
	\end{tikzcd}\]
	In this diagram, the upper left square commutes by equivariance of $f$, the other squares on the top commute as the horizontal map $i_{V, \ldots, V}$ is an $m$-morphism of spectra, and $\pr$ also is a morphism of spectra. The squares on the bottom row commute by the symmetry property of both the direct sum of inner product spaces and the smash product of spectra, and as we exactly quotient out the permutation action on $X^{\wedge m}$ in the passage to $\mathbb P^m X$. Thus this diagram is commutative and proves that the morphism $\hat{P}^m (f)$ is $\Sigma_m\wr G$-equivariant, hence defines an element in $\pi_0^{\Sigma_m\wr G}(\mathbb P^m X)$.\\
	These maps $\hat{P}^m$ are called external power operations. Note that the quotient $X^{\wedge m}\to \mathbb P^m(X)$ is necessary for this definition, since on $X^{\wedge m}$ we have to consider the $\Sigma_m$-action by permuting the factors.
	
	These external operations fit into a commutative diagram
	\begin{equation}\label{diagram:external_power_operations_as_retract}
	\begin{tikzcd}
	\pi_0^G(X) \arrow[r, "(\incl_1)_\ast"] \arrow[d, "\hat P^m", swap] 
		& \pi_0^G(\mathbb PX) \arrow[r, "(\pr_1)_\ast"] \arrow[d, "P^m", swap]
		& \pi_0^G(X) \arrow[d, "\hat{P}^m"]\\
	\pi_0^{\Sigma_m\wr G} (\mathbb P^mX) \arrow[r, "(\incl_m)_\ast", swap]
		& \pi_0^{\Sigma_m\wr G} (\mathbb PX) \arrow[r, "(\pr_m)_\ast", swap]
		& \pi_0^{\Sigma_m\wr G} (\mathbb P^mX),
	\end{tikzcd}
	\end{equation}
	where $\incl_m\colon \mathbb P^mX\to \mathbb PX$ is the inclusion as the wedge summand indexed by $m$ and $\pr_m$ is the projection onto the wedge summand indexed by $m$. Moreover, we use that $\mathbb PX$ as an ultra-commutative ring spectrum has power operations on its homotopy groups. This diagram exhibits $\hat P^m$ as a retract of $P^m$.
\end{construction}

Note that by definition, we obtain the power operation of an ultra-commutative ring spectrum $E$ by composing with the map induced by the multiplication $\mathbb P^mE\to E$. However, we also can consider a weaker type of structure which also allows to internalize these external power operations. This leads to the notion of a $G_\infty$-ring spectrum.

%% file: Chapter_Definition_G_infty.tex
\section{\texorpdfstring{$G_\infty$}{G-infinity}-ring spectra and their properties}

In this chapter, we give the definition of $G_\infty$-ring spectra in  Definition \ref{def:G_infty_ring_spectra}. This notion is a homotopical version of structured ring spectra, with structure morphisms only defined in the global homotopy category. This structured multiplication allows us to construct power operations on the equivariant homotopy groups of a $G_\infty$-ring spectrum in Construction \ref{constr:power_op_for_G_infty_spectra}.

The notion of $G_\infty$-ring spectra is a global generalization of the non-equivariant notion of an $H_\infty$-ring spectrum from \cite[Definition I.3.1]{BMMS_1986}. In Section \ref{section:adjunctions_between_infinity_ring_spectra} we construct an adjunction between $G_\infty$- and $H_\infty$-ring spectra. The left adjoint is a forgetful functor from the globally equivariant $G_\infty$-ring spectrum to the non-equivariant $H_\infty$-ring spectrum. The right adjoint exhibits a way to obtain a $G_\infty$-ring spectrum from an $H_\infty$-ring spectrum, thought of as a ``global Borel construction''. This also gives a way to generate examples of $G_\infty$-ring spectra which do not come as the homotopy types of ultra-commutative ring spectra, see Remark \ref{remark:G_infty_is_not_ucom}. For this, we use the non-equivariant examples of Noel in \cite{Noel_2014} and Lawson in \cite{Lawson_2015} of $H_\infty$-ring spectra which do not rigidify to commutative ring spectra.

In Section \ref{section:homotopical_properties_P}, we compare the derived symmetric powers to a global version of the extended powers $D_mX= (E\Sigma_m)_+\wedge_{\Sigma_m}X^{\wedge m}$ in Theorem \ref{thm:comparison_symmetric_extended_powers}. This can be used to give an alternative description of $G_\infty$-ring spectra, which is closer to the original definition from \cite{BMMS_1986}.

\subsection{Definition of \texorpdfstring{$G_\infty$}{G-infinity}-ring spectra and their power operations}\label{section:G_infty_rings}

Recall that the multiplication maps $\mathbb P^mE\to E$ of an ultra-commutative ring spectrum can be used to define internal power operations on the homotopy groups of $E$ from the external power operations defined in Construction \ref{constr:external_power_operations}. But all that is really needed are such maps on the homotopy groups, hence we define the corresponding structure on the level of the global homotopy category $\GH$. To do so, we make use of the positive global model structures on $\mathcal Sp$ and $\ucom$, which are constructed in \cite[Proposition 4.3.33 and Theorem 5.4.3]{Schwede_2018}.

\begin{lemma}\label{lemma:P_defines_Quillen_adjunction}
The functors $\adjunction{\mathcal Sp}{\ucom}{\mathbb P}{U}$ form a Quillen adjoint functor pair, with respect to the positive global model structures on both sides.
\end{lemma}
\begin{proof}
It is clear that these functors are adjoint to one another. To prove that this is a Quillen adjunction, it suffices to show that the right adjoint $U$ preserves both fibrations and acyclic fibrations. This is directly evident from the characterization of global equivalences and positive global fibrations of ultra-commutative ring spectra by their underlying maps.
\end{proof}

In fact, the model structure on $\ucom$ is transferred from the positive model structure on $\mathcal Sp$ along this adjunction.

Now, every Quillen adjunction defines an adjunction on the homotopy categories, see \cite[Lemma 1.3.10]{Hovey_1999}, hence we get an adjunction
\[ \outlineadjunction{\GH}{\Ho(\ucom).}{L_{\gl}\mathbb P}{\Ho(U)} \]
Note that $U$ is already homotopical, so it can be derived without a fibrant replacement. From this adjunction, we obtain the monad
\begin{equation}\label{eq:monad_G_on_GH}
\mathbb G = \Ho(U)\circ L_{\gl} \mathbb P\colon \GH\to \GH.
\end{equation}

\begin{definition}\label{def:G_infty_ring_spectra}
A $G_\infty$-ring spectrum is an algebra over the monad $\mathbb G$.
\end{definition}

\begin{example}
As we obtain the notion of $G_\infty$-ring spectra as algebras over a derived monad $\mathbb G=L_{\gl} \mathbb P$, we see that algebras over the point-set monad $\mathbb P$, i.e. ultra-commutative ring spectra, also induce a $G_\infty$-ring structure on their global homotopy type. This already gives a broad class of examples, which encompasses the sphere spectrum $\mathbb S$, Eilenberg-Mac Lane spectra $HR$ for a global power functor $R$ as constructed in \cite[Theorem 5.4.14]{Schwede_2018}, and the global versions of Thom and $K$-theory spectra from \cite[Chapter 6]{Schwede_2018}.\\
These examples as homotopy types of ultra-commutative ring spectra however do not provide all $G_\infty$-ring spectra: In Theorems \ref{thm:Example_Noel} and \ref{thm:Example_Lawson}, we provide examples of $G_\infty$-ring spectra which are not the homotopy type of an ultra-commutative ring spectrum such that the $G_\infty$-ring structure is induced by the ultra-commutative multiplication.
\end{example}

\begin{remark}
We also note that the definition of $G_\infty$-ring spectra is not the same as that of a homotopy commutative ring spectrum in $\GH$. This can be seen from the fact that a homotopy commutative ring spectrum does not support power operations on its homotopy groups, whereas Proposition \ref{prop:power_operations_on_G_infty_spectra} proves the existence of equivariant power operations for $G_\infty$-ring spectra. For the same reason, $G_\infty$-ring spectra posses more structure than $H_\infty$-rings internal to the global homotopy category $\GH$. The difference lies in the fact that the derived symmetric power $\mathbb G^mX$ can be represented by the global extended power $\Sigma^\infty_+ E_{\gl}\Sigma_m \wedge_{\Sigma_m} X^{\wedge m}$ as shown in Theorem \ref{thm:comparison_symmetric_extended_powers}. In contrast, for an $H_\infty$-ring spectrum, the non-equivariant extended powers $\Sigma^\infty_+ E\Sigma_m \wedge_{\Sigma_m} X^{\wedge m}$ would be used.
\end{remark}

Our aim is to define power operations on the homotopy groups of a $G_\infty$-ring spectrum. Since the definition of $G_\infty$-ring spectra is internal to the global homotopy category, we also rephrase the external power operations in terms of the representability of the homotopy groups in $\GH$. To do this, we derive the levels of $\mathbb P$ separately with respect to the positive global model structures from \cite[4.3.33 and 5.4.3]{Schwede_2018}.

\begin{lemma}\label{lemma:P^m_is_derivable}
	Let $f\colon X\to Y$ be a global equivalence between positively cofibrant spectra, and let $m\geq 0$. Then $\mathbb P^mf\colon \mathbb P^mX\to \mathbb P^mY$ is a global equivalence in $\mathcal Sp$.
\end{lemma}
\begin{proof}
	We follow the argument given at the end of the proof of \cite[Theorem 5.4.12]{Schwede_2018}. By \ref{lemma:P_defines_Quillen_adjunction}, we know that the functor $\mathbb P\colon \mathcal Sp\to \ucom$ is left Quillen. By Ken Brown's lemma \cite[Lemma 1.1.12]{Hovey_1999}, $\mathbb Pf\colon \mathbb PX\to \mathbb PY$ is a global equivalence of ultra-commutative ring spectra, hence by definition a global equivalence of the underlying spectra. But the transformations $\mathbb P^m \nameto{\incl_m} \mathbb P= \bigvee_{m\geq0} \mathbb P^m \nameto{\pr_m} \mathbb P^m$ exhibit $\mathbb P^mf$ as a retract of $\mathbb Pf$, thus also this morphism is a global equivalence. 
\end{proof}
\begin{remark}\label{remark:P^m_also_preserves_cofibrancy}
	This lemma is enough to conclude that $\mathbb P^m\colon \mathcal Sp\to \mathcal Sp$ admits a left derived functor 
	\[\mathbb G^m\colon \GH\to \GH.\]
	However, one can indeed show more: for any $m>0$, the functor $\mathbb P^m\colon \mathcal Sp\to \mathcal Sp$ preserves positive cofibrations and acyclic positive cofibrations between positively cofibrant spectra. This uses \cite[Theorem 22]{Gorchinskiy_Guletski_2016}, that the positive cofibrations are symmetrizable \cite[Theorem 5.4.1]{Schwede_2018} and that smashing with positively cofibrant spectra preserves weak equivalences \cite[Theorem 4.3.27]{Schwede_2018}.
\end{remark}

We now calculate the value of the functor $\mathbb G$ on the sphere spectrum and on the representing spectra $\Sigma^\infty_+ B_{\gl}G$ for the global homotopy groups. We use this to define the external power operations intrinsically in $\GH$.

\begin{example}\label{example:G_on_classifying_spaces}
	We calculate the value of $\mathbb G$ on the sphere spectrum $\mathbb S$:\\
	Since $\mathbb S$ is not positively cofibrant, we need to positively replace it. For this, consider the map 
	\[\lambda_{\Sigma_1, \mathbb R, 0}\colon F_{\Sigma_1, \mathbb R}S^1 \to F_{\Sigma_1, 0}=\mathbb S\]
	from \cite[4.1.28]{Schwede_2018}. This map is a global equivalence by \cite[Theorem 4.1.29]{Schwede_2018}, as $0$ is a faithful representation of the trivial group $\Sigma_1$. Moreover, the spectrum $F_{\Sigma_1, \mathbb R}S^1= \mathbf O(\mathbb R,\_)\wedge S^1$ is positively cofibrant, hence this map $\lambda_{\Sigma_1, \mathbb R, 0}$ can be chosen as a positively cofibrant replacement. Then, we have that $\mathbb G^m \mathbb S$ is represented by
	\[\mathbb P^m (F_{\Sigma_1, \mathbb R}S^1) \cong \mathbf O(\mathbb R^m, \_)\wedge_{\Sigma_m} S^m = F_{\Sigma_m, \mathbb R^m}S^m.\]
	Hence, by the description of the global classifying spaces via semifree orthogonal spectra in \eqref{eq:suspension_spectrum_of_BG} and since $\mathbb R^m$ is a faithful $\Sigma_m$-representation, we see that
	\[\mathbb G^m\mathbb S\cong \Sigma^\infty_+ B_{\gl}\Sigma_m.\]
	
	More generally, we can calculate $\mathbb G^m (\Sigma^\infty_+ B_{\gl}G)$ for any compact Lie group, with the previous calculation a special case for $G=e$, using the identification $\Sigma^\infty_+ B_{\gl}e\cong \mathbb S$. For this calculation, choose a non-zero faithful $G$-representation $V$. Then we can write 
	\[\Sigma^\infty_+ B_{\gl}G \cong F_{G,V}S^V\]
	as in \eqref{eq:suspension_spectrum_of_BG}. Now, the spectrum $F_{G,V}S^V$ is positively cofibrant, and we calculate
	\[\mathbb P^m (F_{G,V}S^V) \cong F_{G^m, V^m}S^{V^m} /\Sigma_m \cong F_{\Sigma_m\wr G, V^m} S^{V^m},\]
	where for the last identification, we used that the permutation action of $\Sigma_m$ and the action of $G^m$ on $V^m$ assemble into the natural action of $\Sigma_m\wr G$. Then $V^m$ is a faithful $\Sigma_m\wr G$-representation, hence we see that 
	\begin{equation}\label{eq:G_on_classifying_spaces}
	\mathbb G^m (\Sigma^\infty_+ B_{\gl}G) \cong \Sigma^\infty_+ B_{\gl} (\Sigma_m \wr G).
	\end{equation}
\end{example}

\begin{construction}\label{constr:external_operations_using_G}
	We now give another description of the external power operation, using our calculation of $\mathbb G^m$ on the representing spectra for $\pi_0^G$.\\
	Let $f\in \pi_0^G(X)$ be an element of the homotopy groups of a global homotopy type $X$. By the representability result \eqref{eq:representability_of_pi_0}, we can represent $f$ by a map $f\colon \Sigma^\infty_+ B_{\gl}G \to X$ in the global homotopy category. Then, we define for $m\geq 1$
	\[ \mathbb G^m(f)\colon \Sigma^\infty_+ B_{\gl}(\Sigma_m\wr G) \cong \mathbb G^m (\Sigma^\infty_+ B_{\gl}G)\nameto{\mathbb G^m f} \mathbb G^m X,\]
	and this morphism represents an element in $\pi_0^{\Sigma_m\wr G}(\mathbb G^mX)$. Thus we define the external operations as the effect of the functor $\mathbb G^m$ on the homotopy groups $\pi_0^G(X) \cong \GH(\Sigma^\infty_+ B_{\gl}G, X)$, and obtain maps
	\[\mathbb G^m\colon \pi_0^G(X) \to \pi_0^{\Sigma_m\wr G}(\mathbb G^m X). \]
\end{construction}
\begin{lemma}\label{lemma:comparison_external_operations}
	For a positively cofibrant spectrum $X$, the two external operations $\hat{P}^m$ and $\mathbb G^m\colon \pi_0^G(X)\to \pi_0^{\Sigma_m\wr G}(\mathbb P^mX)$ agree.
\end{lemma}
\begin{proof}
	Let $f\in \GH(\Sigma^\infty_+ B_{\gl}G, X)$, then by \eqref{eq:representability_of_pi_0} the corresponding class in $\pi_0^G(X)$ is $f_\ast(e_G)$. Concretely, let $f\colon S^V \wedge (B_{\gl}G)_+(V) \to X(V)$ for a non-zero faithful $G$-representation $V$. We can always represent $f$ on a faithful $G$-representation, since we can embed any $G$-representation into a faithful one. Then, the tautological class in $\pi_0^G(\Sigma^\infty_+ B_{\gl}G)$ is
	\[e_G\colon S^V \nameto{\_\wedge \Id_V} S^V\wedge (B_{\gl}G)_+(V) = S^V \wedge \mathbf L(V,V)_+/G,\]
	and the tautological class for $\Sigma_m\wr G$ is 
	\[ e_{\Sigma_m\wr G} \colon S^{V^m} \nameto{\_\wedge \Id_{V^m}} S^{V^m}\wedge \mathbf L(V^m,V^m)_+/(\Sigma_m\wr G).\]
	Note that this element agrees with
	\[\hat{P}^m e_G\colon S^{V^m} \nameto{(\_\wedge \Id_{V})^m} S^{V^m}\wedge (\mathbf L(V,V)_+/G)^m \nameto{\pr} S^{V^m}\wedge \mathbf L(V^m,V^m)_+/(\Sigma_m\wr G).\]
	We now compare $\hat{P}^m( f_\ast (e_G))$ and $(\mathbb G^m f)_\ast(e_{\Sigma_m\wr G})$. Since $X$ is positively cofibrant, we can write $\mathbb P^m$ instead of $\mathbb G^m$. Then by naturality of the external power operations, we obtain 
	\[ \hat P^m(f_\ast (e_G)) = \mathbb P^m(f)_\ast (\hat P^m(e_G)) = \mathbb P^m (f)_\ast (e_{\Sigma_m\wr G}). \]
	Hence, the two operations agree.
\end{proof}

We now prove that a $G_\infty$-ring spectrum structure indeed gives rise to power operations on the homotopy groups.

\begin{construction}\label{constr:power_op_for_G_infty_spectra}
	Let $E$ be a $G_\infty$-ring spectrum. We consider the multiplication map $\mathbb G E\to E$ in the homotopy category. Since we can derive $\mathbb P$ levelwise, this decomposes into maps $\zeta_m\colon \mathbb G^mE\to E$. Thus, we define the power operations as
	\[P^m= \underline{\pi}_0(\zeta_m)\circ \mathbb G^m\colon \pi_0^G(E)\to \pi_0^{\Sigma_m\wr G} (\mathbb G^mE) \to \pi_0^{\Sigma_m\wr G}(E).\]
\end{construction}

\begin{prop}\label{prop:power_operations_on_G_infty_spectra}
Let $E$ be a $G_\infty$-ring spectrum with structure map $\zeta\colon \mathbb GE\to E$. Then the operations $P^m$ defined in Construction \ref{constr:power_op_for_G_infty_spectra}, together with the multiplication given by $\zeta_2$, define a structure of a global power functor on $\underline{\pi}_0(E)$.
\end{prop}
\begin{proof}
The map $\zeta_2\colon \mathbb G^2 E\to E$ together with the unit map makes $E$ into an homotopy commutative ring spectrum, hence $\underline{\pi}_0(E)$ is a global Green functor by Proposition \ref{prop:homotopy_for_ring_spectrum_Green_functor}. Moreover, we need to check the relations for the global power operations as listed in \cite[Definition 5.1.6]{Schwede_2018}. For this, we use the same naturality as in the proof of Lemma \ref{lemma:comparison_external_operations} and check on the representing spectra $\Sigma^\infty_+ B_{\gl} G$. Since the arguments are all similar, we focus on one property here:

Let $i, j\geq 0$, and $x\in \pi_0^G(E)$. We need to check that $\Phi_{i,j}^\ast (P^{i+j}(x))= P^i(x) \times P^j(x)$ holds in $\pi_0^{\Sigma_i\wr G\times \Sigma_j\wr G}(E)$. Here, the map $\Phi_{i,j}\colon \Sigma_i\wr G\times \Sigma_j\wr G\to \Sigma_{i+j}\wr G$ is given by juxtaposition of permutations. Let $x$ be represented by a map $f\colon \Sigma^\infty_+ B_{\gl}G \to E$ in $\GH$. Then the  power operations on $x$ are given as $P^m (x)=\underline{\pi}_0(\zeta_m) (\mathbb G^m f)_\ast (e_{\Sigma_m \wr G})$.

We now consider the maps $\varphi_{i,j}\colon \mathbb P^iX\wedge \mathbb P^jX \to \mathbb P^{i+j}X$, given for any $X$ by forming orbits along $\Phi_{i,j}$ in $X^{\wedge i+j}/(\Sigma_i\times \Sigma_j) \to X^{\wedge i+j}/(\Sigma_{i+j})$. For $X= \Sigma^\infty_+ B_{\gl} G$, this map represents the restriction $\Phi_{i,j}^\ast\colon \pi_0^{\Sigma_{i+j}\wr G} \to \pi_0^{\Sigma_i\wr G\times \Sigma_j\wr G}$, meaning that $(\varphi_{i,j})_\ast (e_{\Sigma_i\wr G} \times e_{\Sigma_j\wr G}) = \Phi_{i,j}^\ast(e_{\Sigma_{i+j}\wr G})$. Moreover, $\varphi_{i,j}\colon \mathbb P^i\wedge \mathbb P^j \to \mathbb P^{i+j}$ features in the monad structure of the functor $\mathbb P$, hence also in the derived monad structure for $\mathbb G$. The fact that $\zeta$ defines a $G_\infty$-structure hence shows that we have the commutative square 
\begin{equation}\label{eq:levelwise_algebra_structure}
\begin{tikzcd}
\mathbb G^iE\wedge^L \mathbb G^j E \arrow[r, "\zeta_i \wedge \zeta_j"] \arrow[d, "\varphi_{i,j}", swap] 
	& E\wedge^L E \arrow[d, "\mu"]\\
\mathbb G^{i+j} E \arrow[r, "\zeta_{i+j}", swap]
	& E.
\end{tikzcd}
\end{equation}
Here, $\mu$ is the homotopy multiplication induced by $\zeta_2$. This square is also used in the original definition of $H_\infty$-ring spectra in \cite[Definition I.3.1]{BMMS_1986}. In total, we calculate
\begin{align*}
\underline{\pi}_0(\zeta_{i+j}) (\mathbb G^{i+j} f) \Phi_{i,j}^\ast (e_{\Sigma_{i+j} \wr G}) = & 
	\underline{\pi}_0(\zeta_{i+j}) (\mathbb G^{i+j} f) (\varphi_{i,j})_\ast (e_{\Sigma_i \wr G} \times e_{\Sigma_j \wr G})\\
= & \underline{\pi}_0(\zeta_{i+j}) (\varphi_{i,j})_\ast ( \mathbb G^i f\wedge \mathbb G^jf)_\ast (e_{\Sigma_i \wr G} \times e_{\Sigma_j \wr G})\\
= & \underline{\pi}_0(\zeta_i) (\mathbb G^i f)_\ast (e_{\Sigma_i \wr G}) \times \underline{\pi}_0(\zeta_j) (\mathbb G^j f)_\ast (e_{\Sigma_j \wr G}).
\end{align*}
This proves one of the relations of a global power functor, the other ones follow similarly by considering $\mathbb G^k\mathbb G^m (\Sigma^\infty_+ B_{\gl} G)$, $\mathbb G^m(\Sigma^\infty_+ B_{\gl} G\wedge \Sigma^\infty_+ B_{\gl} K)$ and $\mathbb G^m(\Sigma^\infty_+ B_{\gl} G\vee \Sigma^\infty_+ B_{\gl} G)$.
\end{proof}

\begin{remark}\label{remark:power_operations_on_cohomology}
As an application of this description of the external power operations, we also define external cohomology operations, and show that a $G_\infty$-structure can be used to internalize these operations. These internal cohomology operations are also constructed for an ultra-commutative ring spectrum in \cite[Remark 5.1.14]{Schwede_2018}.\\
Let $X$ be an orthogonal spectrum and $A$ be a cofibrant based $G$-space. We define an orthogonal space $\mathbf L_{G,V}A = \mathbf L(V,\_)\wedge_G A$ for any $G$-representation $V$, similar to the construction in Remark \ref{remark:global_classifying_spaces}. Then we define the $G$-equivariant $X$-cohomology of $A$ as
\[ X^0_G(A) = [\Sigma^\infty_+ \mathbf L_{G,V}A, X], \]
where $[\_,\_]$ denotes the morphisms in $\GH$, and $V$ is any faithful $G$-representation. Then, external power operations on this $X$-cohomology are defined by 
\begin{align*}
\hat P^m\colon X^0_G(A) = [\Sigma^\infty_+ \mathbf L_{G,V}A, X] \nameto{\mathbb G^m} & [\mathbb G^m\Sigma^\infty_+ \mathbf L_{G,V}A, \mathbb G^m X]\\
= & [\Sigma^\infty_+ \mathbf L_{\Sigma_m\wr G,V^m}A^m, \mathbb G^m X]= (\mathbb G^mX)^0_{\Sigma_m\wr G}(A^m).
\end{align*}
Here, we used a relative version of the calculations in \ref{example:G_on_classifying_spaces} to calculate 
\[\mathbb G^m\Sigma^\infty_+ \mathbf L_{G,V}A\cong\Sigma^\infty_+ \mathbf L_{\Sigma_m\wr G,V^m}A^m.\]
Using a $G_\infty$-ring structure on $X$, given by morphisms $\zeta_m\colon \mathbb G^mX\to X$, we can internalize these operations to 
\[P^m\colon  X^0_G(A) \nameto{\hat P^m} (\mathbb G^mX)^0_{\Sigma_m\wr G} (A^m)\nameto{(\zeta_m)_\ast} X^0_{\Sigma_m\wr G} (A^m). \]
In \cite[Remark 5.1.14]{Schwede_2018}, it is shown that these power operations forget to the classical power operations $X^0(A) \to X^0(B\Sigma_m\times A)$ on the non-equivariant $X$-cohomology of $A$ upon postcomposition with the diagonal on $A$. These are the power operations induced by an $H_\infty$-structure on $X$ in \cite[Definition I.4.1]{BMMS_1986}
\end{remark}

\subsection{An adjunction between \texorpdfstring{$G_\infty$}{G-infinity}- and \texorpdfstring{$H_\infty$}{H-infinity}-ring spectra}\label{section:adjunctions_between_infinity_ring_spectra}

In this section, we compare the notion of $G_\infty$-ring spectra to the classical notion of $H_\infty$-ring spectra. This is accomplished by lifting the adjunction 
\[\outlineadjunction{\GH}{\SH}{U}{R}\]
to structured ring spectra, where $U$ is the forgetful functor and $R$ its right adjoint. This adjunction is exhibited in \cite[Theorem 4.5.1]{Schwede_2018}.

\begin{remark}\label{remark:properties_of_Stolz_model_structure}
In this chapter, we use for the stable homotopy category the positive stable model structure defined by Stolz in \cite[Chapter 1.3]{Stolz_2011}. There, it is called the $\mathbb S$-model structure. This model structure has the following two desirable properties, which we use in order to analyse the derivability of the symmetric product functor $\mathbb P$ in Lemma \ref{lemma:composites_of_P_left_derivable}:
\begin{enumerate}[\itshape i)]
	\item The cofibrations and positive acyclic cofibrations are symmetrizable \cite[Definition 3]{Gorchinskiy_Guletski_2016}: If $f\colon X\to Y$ is a cofibration, then for all $n\geq 1$ the iterated pushout product map
	\[f^{\square n}/\Sigma_n \colon Q^n(f)/\Sigma_n \to \mathbb P^n(Y) \]
	is a cofibration. Here, $Q^n(f)$ is the colimit over the punctured cube diagram
	\begin{align*} 
		\{0\to 1\}^n\setminus \{1,\ldots, 1\} & \to \mathcal Sp\\
		(i_1,\ldots,i_n) & \mapsto Z_{i_1}\wedge \ldots \wedge Z_{i_n}\\
		id\wedge \ldots \wedge (0\to 1) \wedge \ldots \wedge id & \mapsto  id\wedge \ldots \wedge f \wedge \ldots \wedge id.
	\end{align*}
	In this definition, we set
	\[Z_i= \begin{dcases*}
	X & if $i=0$\\
	Y & if $i=1$.
	\end{dcases*} \]
	The corresponding property also holds for the positive acyclic cofibrations. That these properties hold for the stable model structure constructed by Stolz follows from observing that the cofibrations agree with those of the global model structure constructed by Schwede in \cite[Chapter 4.3]{Schwede_2018}, where symmetrizability is verified in \cite[Theorem 5.4.1]{Schwede_2018}. For the acyclic cofibrations, a similar calculation can be carried out.
	\item Cofibrant objects are flat: If $X$ is cofibrant in the stable model structure, then $X\wedge \_ $ preserves stable equivalences \cite[Proposition 1.3.11]{Stolz_2011}.
\end{enumerate}
These properties are required in order to apply the results from \cite{Gorchinskiy_Guletski_2016}, in particular Theorem 25, which states that the functor $\mathbb P^n$ preserves weak equivalences between positively cofibrant objects.
\end{remark}

\subsubsection{Lifting the forgetful functor \texorpdfstring{$\GH\to \SH$}{GH to SH} to structured ring spectra}

We first recall the classical definition of $H_\infty$-ring spectra: As defined in \cite[I, Definition 3.1]{BMMS_1986}, an $H_\infty$-ring spectrum $X$ is defined by maps 
\[\xi_m\colon D_m X\to X,\]
where $D_mX = (E\Sigma_m)_+ \wedge_{\Sigma_m} X^{\wedge m}$. These maps are required to satisfy compatibility conditions such as the one used in \eqref{eq:levelwise_algebra_structure}. Note that this formulation uses the modern smash product, which was not yet available in the original definition. Unravelling the definitions in \cite{BMMS_1986} however gives this formulation. In contrast, our definition of $G_\infty$-ring spectra uses a modern point-set category of spectra to obtain the monad $\mathbb G$, and defines $G_\infty$-ring spectra as algebras over this monad. As this definition is more conceptual and allows us to use the results from Appendix \ref{section:monads_under_lax_functors}, we also formulate the notion of $H_\infty$-ring spectra in this way.\\
For this, note that the adjunction 
\[ \outlineadjunction{\mathcal Sp}{\Com}{\mathbb P}{U_{\Com}}\]
is also a Quillen adjunction with respect to the positive stable model structures defined by Stolz in \cite[Proposition 1.3.10 and Theorem 1.3.28]{Stolz_2011}. Thus we obtain a derived adjunction 
\[ \outlineadjunction[huge]{\SH}{\Ho^{\st}(\Com).}{L_{\st}\mathbb P}{\Ho(U_{\Com})}\]

\begin{definition}\label{def:H_infty_ring_spectra}
An $H_\infty$-ring spectrum is an algebra over the monad $\mathbb H=\Ho(U_{\Com})\circ L_{\st}\mathbb P$.
\end{definition}
By abuse of notation, we will also denote $\mathbb H$ as $L_{\st}\mathbb P$, since it is the left derived functor of $\mathbb P\colon \mathcal Sp\to \mathcal Sp$. In the same way, we denote $\mathbb G$ by $L_{\gl}\mathbb P$.

That our definition using $L_{\st}\mathbb P$ agrees with the original definition follows from the following statement, after the necessary translations regarding the different models for spectra:
\begin{lemma}\label{lemma:extended_versus_symmetric_powers_stable}
Let $X$ be a positive stably cofibrant orthogonal spectrum. Then the map
\[p\colon D_mX= (E\Sigma_m)_+ \wedge_{\Sigma_m} X^{\wedge m} \to X^{\wedge m} /\Sigma_m = \mathbb P^mX \]
that collapses $E\Sigma_m$ is a stable weak equivalence.
\end{lemma}
\begin{proof}
Since we work in the stable model structure constructed by Stolz, this is the statement of \cite[Lemma 1.3.17]{Stolz_2011}, where a cellular induction along the lines of \cite[p. 36-37]{BMMS_1986} is carried out. The analogous statement for the more commonly used projective model structure by Mandell-May-Schwede-Shipley is \cite[Lemma 15.5]{MMSS_2001}.
\end{proof}

Using this definition of $H_\infty$-ring spectra, we show that the underlying stable homotopy type of a $G_\infty$-ring spectrum is an $H_\infty$-spectrum. To do this, we show that the derived functor $U\colon \GH\to \SH$ is a monad functor in the sense of \ref{def:lax_monad_functor}. We deduce this formally from a variant of the fact that taking the homotopy category of a model category is a pseudo-2-functor \cite[1.4.2f]{Hovey_1999}:

We consider the 2-category $(\textup{Model}, \textup{left})$ of model categories and left Quillen functors. Then, \cite[1.4.3]{Hovey_1999} shows that taking homotopy categories and left derived functors is a pseudo 2-functor $L\colon (\textup{Model}, \textup{left}) \to \textup{Cat}$. Hence, by Corollary \ref{corollary:lax_functor_preserves_monads} the functor $L$ preserves monads and monad morphisms.

However, the functor $\mathbb P\colon \mathcal Sp\to \mathcal Sp$ is not left Quillen, but merely left derivable, i.e. it sends weak equivalences between cofibrant objects to weak equivalences, in both the stable and the global positive model structure. Moreover, all compositions $\mathbb P^{\circ k}\colon \mathcal Sp \to \mathcal Sp$ can be derived:

\begin{lemma}\label{lemma:composites_of_P_left_derivable}
Let $X$ be a positively cofibrant spectrum in either the stable or global model structure, and let $A= \bigvee_I \mathbb S$ be a wedge of sphere spectra. Then $\mathbb P(A\vee X)\cong B\vee Y$, where $B=\bigvee_J \mathbb S$ is a wedge of spheres which only depends on $A$, and where $Y$ is a positively cofibrant spectrum. Moreover, if $f\colon X\to X^\prime$ is a weak equivalence between positively cofibrant spectra, then also $\mathbb P(id\vee f)$ is a weak equivalence of the form $id\vee g\colon B\vee Y\to B\vee Y^\prime$.\\
In particular, for any $k\geq 1$, the functor $\mathbb P^{\circ k}\colon \mathcal Sp\to \mathcal Sp$ sends weak equivalences between positively cofibrant spectra to weak equivalences.
\end{lemma}
\begin{proof}
We write 
\[\mathbb P(A\vee X) \cong  \mathbb P(A) \wedge \mathbb P(X)
\cong \mathbb P(A)\wedge (\mathbb S \vee  \mathbb P_{>0} (X)) 
\cong \mathbb P(A) \vee (\mathbb P(A) \wedge \mathbb P_{>0} (X))\]
Now, we see that 
\[\mathbb P(A) =  \mathbb P\left(\bigvee_I \mathbb S \right) \cong \bigwedge_I(\mathbb P \mathbb S) \cong \bigwedge_I \left(\bigvee_{i\geq 0} \mathbb S\right) \cong \bigvee_{\mathbb N^I}\left( \bigwedge_I \mathbb S\right) \cong \bigvee_{\mathbb N^I} \mathbb S\]
is a wedge of spheres. Moreover, the spectrum $\mathbb P_{>0} X$ is positively cofibrant by applying \cite[Corollary 10]{Gorchinskiy_Guletski_2016} to the positive model structures, and hence also $\mathbb P(A)\wedge \mathbb P_{>0} X$ is positively cofibrant. This proves the first assertion, putting $B=\mathbb P(A)$ and $Y= \mathbb P(A)\wedge \mathbb P_{>0} X$. If $ f\colon X\to X^\prime$ is a weak equivalence between positively cofibrant spectra, so are $\mathbb P_{>0}(f)$ and $\mathbb P(A)\wedge \mathbb P_{>0} (f)$ by the observations in Remark \ref{remark:properties_of_Stolz_model_structure}. This proves the second part of the lemma, since $\mathbb P(id\vee f)= id_{\mathbb P(A)} \vee(\mathbb P(A)\wedge \mathbb P_{>0} (f))$.\\
In total, this proves the conclusion that $\mathbb P^{\circ k}$ preserves weak equivalences between positively cofibrant spectra by induction.
\end{proof}

Now we generalize the statement of \cite[1.4.3]{Hovey_1999} to encompass all left derivable functors. There are two problems: the class of left derivable functors is not closed under composition, and if $F$ and $G$ are composable left derivable functors such that $GF$ also is left derivable, the natural transformation $LG\circ LF \to L(GF)$ might not be invertible. However, we obtain the following result:

\begin{prop}\label{prop:left_deriving_is_partial_lax_functor}
Let $(\textup{Model},\textup{all})$ be the $2$-category of model categories and all functors and natural transformations, and let $\mathcal{LD}er_1$ denote the class of all left derivable functors and $\mathcal{LD}er_2$ the class of all natural transformations between left derivable functors. Then the assignment 
\begin{align*}
L \colon (\textup{Model}, \mathcal{LD}er_1, \mathcal{LD}er_2)\to \textup{Cat}\\
\mathcal C\mapsto \Ho(\mathcal C), F\mapsto LF, \eta\mapsto L\eta 
\end{align*}
comes equipped with the following structure:
\begin{enumerate}[i)]
\item A unitality isomorphism $\alpha_{\mathcal C}\colon id_{\Ho(\mathcal C)} \to L(id_{\mathcal C})$ for any model category $\mathcal C$.
\item A natural transformation $\mu_{G,F}\colon LG\circ LF \to L(GF)$ for any pair of left derivable functors $F\colon \mathcal C\to \mathcal D$, $G\colon \mathcal D\to \mathcal E$ such that $GF$ is also left derivable.
\end{enumerate}
These satisfy the properties of a lax $2$-functor from Definition \ref{def:lax_functors_2_cat} where they are defined.\\
Moreover, if $F\colon \mathcal C\to \mathcal D$ is left derivable and $U\colon \mathcal D\to \mathcal E$ is homotopical, then $UF$ is left derivable and $\mu_{U, F}$ is invertible.
\end{prop}
\begin{proof}
The proof is the same as that of \cite[1.4.3]{Hovey_1999}, where we weaken the requirement of being left Quillen to sending weak equivalences between cofibrant objects to weak equivalences.
\end{proof}

We consider the following commutative diagram
\begin{equation}\label{diagram:U_is_monad_functor}
\begin{tikzcd}
\mathcal Sp^{\gl} \arrow[r, "U"] \arrow[d, "\mathbb P_{\gl}", swap] & \mathcal Sp^{\st} \arrow[d, "\mathbb P_{\st}"]\\
\mathcal Sp^{\gl} \arrow[r, "U", swap] & \mathcal Sp^{\st},
\end{tikzcd}
\end{equation}
which exhibits the functor $U$ as a monad functor between $(\mathbb P_{\gl}, \mu_{\gl}, \eta_{\gl})$ and $(\mathbb P_{\st}, \mu_{\st}, \eta_{\st})$. Moreover, since $U$ is homotopical, it guarantees that all composites $\mathbb P^{\circ i}\circ U\circ \mathbb P^{\circ j}= U\circ \mathbb P^{\circ i+j}$ are left derivable. Proposition \ref{prop:left_deriving_is_partial_lax_functor} allows us to conclude that taking homotopy categories and left derived functors preserves the monads $\mathbb P$ on $\mathcal Sp^{\gl}$ and $\mathcal Sp^{\st}$ as well as the functor $U$ between them.

\begin{prop}
The left derived functors $L\mathbb P_{\st}$ and $L\mathbb P_{\gl}$ have the structure of monads via the natural transformations
\[ L\mu_{\gl}\circ \mu_{\mathbb P_{\gl}, \mathbb P_{\gl}}, \, L\eta_{\gl} \circ \alpha_{\mathcal Sp^{\gl}} \]
and the analogous transformations for $L\mathbb P_{\st}$.\\
Moreover, the derived functor $\Ho(U)\colon \Ho(\mathcal Sp^{\gl}) \to \Ho(\mathcal Sp^{\st})$ has the structure of a monad functor between $L\mathbb P_{\gl}$ and $L\mathbb P_{\st}$ via the transformation $\mu_{U, \mathbb P_{\gl}}\inverse \circ \mu_{\mathbb P_{\st}, U} $.
\end{prop}
\begin{proof}
This is the statement of Corollary \ref{corollary:lax_functor_preserves_monads}. To apply this corollary as stated, we would need to have a lax 2-functor $L\colon \textup{Model} \to\textup{Cat}$ encompassing all left derivable functors. However, it suffices that we have the required structure morphisms for all composites $\mathbb P_{\st}^{\circ i} \circ U\circ \mathbb P_{\gl}^{\circ j}$ with $i+j\leq 3$. This is the case by Lemma \ref{lemma:composites_of_P_left_derivable} and the commutative square \ref{diagram:U_is_monad_functor}.
\end{proof}

Since monad functors lift to functors on the categories of algebras, we have proven the following:
\begin{prop}\label{prop:U_lifts_to_G_infty}
The functor $U$ lifts to a functor from the category of $G_\infty$-ring spectra to the category of $H_\infty$-ring spectra.\\
Explicitly, let $X$ be a $G_\infty$-ring spectrum with structure map $h\colon L\mathbb P_{\gl}X\to X$. Then the map $h\circ \mu_{U, \mathbb P_{\gl}}\inverse \circ \mu_{\mathbb P_{\st}, U}\colon (L_{\st} \mathbb P)UX \to UX$ defines an $H_\infty$-ring structure on the stable homotopy type $UX$. Moreover, for a $G_\infty$-ring morphism $f\colon X\to Y$ in $\GH$ between two $G_\infty$-ring spectra, the map $U(f)\in \SH(UX, UY)$ is an $H_\infty$-ring map.
\end{prop}

Using this result, we get the commutative diagram
\begin{equation}\label{diagram:structured_ring_spectra}
\begin{tikzcd}
\Ho^{\gl} (\ucom) \rar \dar & (G_\infty \textrm{-ring spectra}) \dar \\
\Ho^{\st} (\Com) \rar & (H_\infty \textrm{-ring spectra})
\end{tikzcd}
\end{equation}
of homotopy categories of structured ring spectra, where all functors are forgetful ones.

\subsubsection{Lifting the right adjoint \texorpdfstring{$\SH\to \GH$}{SH to GH} to structured ring spectra}

In this section, we study whether the forgetful functor $U$ from the category of $G_\infty$-ring spectra to $H_\infty$-ring spectra from \eqref{diagram:structured_ring_spectra} has adjoints. The corresponding question for the homotopy categories $\SH$ and $\GH$ is investigated in \cite[Chapter 4.5]{Schwede_2018}, and we use these results to obtain a right adjoint to the forgetful functor. This gives us a way to define $G_\infty$-ring spectra from $H_\infty$-ring spectra, and we use this to give examples of $G_\infty$-ring spectra which do not come from ultra-commutative ring spectra. 

We first recall the right adjoint $R\colon \SH\to \GH$ to the forgetful functor from \cite[Construction 4.5.21]{Schwede_2018}.

\begin{construction}\label{constr:point_set_lift_of_right_adjoint}
We define the functor
\[b\colon \mathcal Sp\to \mathcal Sp\]
as a ``global Borel construction'': For an orthogonal spectrum $X$ and an inner product space $V$, we set 
\[(bX)(V)= \map(\mathbf L(V, \mathbb R^\infty), X(V)),\]
with structure morphisms defined as in \cite[4.5.21]{Schwede_2018}.\\
We also define a natural transformation $i\colon \Id\to b$ via the map
\[
i_X(V)\colon X(V) \to \map(\mathbf L(V, \mathbb R^\infty), X(V)), x \mapsto \const_x.
\]
The morphism $i_X(V)$ is a non-equivariant homotopy equivalence for all inner product spaces $V$, as the space $\mathbf L(V, \mathbb R^\infty)$ is contractible. Hence the induced morphism $i_X\colon X\to bX$ is invertible in the stable homotopy category.\\
Moreover, $b$ comes equipped with a lax symmetric monoidal structure for the smash product of orthogonal spectra, such that $i\colon \Id\to b$ is a monoidal transformation. Thus, we obtain the following:
\end{construction}

\begin{corollary}\label{corollary:b_preserves_P_algebras}
	The functor $b\colon \mathcal Sp\to \mathcal Sp$ defines a monad endofunctor in the sense of Definition \ref{def:lax_monad_functor} of the symmetric algebra monad $\mathbb P$ on $\mathcal Sp$.\\
	Moreover, the transformation $i\colon \Id\to b$ is a monadic transformation.
\end{corollary}

By \cite[Propositon 4.5.22]{Schwede_2018}, the functor $b$ represents the right adjoint to the forgetful functor $U\colon \GH\to \SH$ on stable $\Omega$-spectra. We modify this statement to hold on positive $\Omega$-spectra, since we need positive model structures for the study of commutative ring spectra. Recall that for obtaining the stable homotopy category, we use the stable $\mathbb S$-model structure constructed by Stolz in \cite[Proposition 1.3.10]{Stolz_2011}, in order to achieve derivability of the symmetric powers in Lemma \ref{lemma:composites_of_P_left_derivable}. Note that this model structure has fewer fibrant objects than the projective positive stable model structure from \cite[Theorem 14.2]{MMSS_2001}, so all fibrant objects are in particular positive $\Omega$-spectra.

\begin{prop}\label{prop:b_represents_right_adjoint_positively}
Let $X$ be a positive orthogonal $\Omega$-spectrum.
\begin{enumerate}[i)]
	\item Then $bX$ is a positive global $\Omega$-spectrum whose homotopy type lies in the image of the right adjoint $R$.
	\item For every orthogonal spectrum $A$, the two homomorphisms 
	\[ \GH(A, bX) \nameto{U} \SH(A, bX) \nameto{(i_X)_\ast\inverse} \SH(A, X)\]
	are isomorphisms. In particular, the counit of the adjunction between $U$ and $R$ is given by $i_X\inverse\colon bX\to X\in \SH(bX, X)$.
\end{enumerate}
\end{prop}
\begin{proof}
The proof is completely analogous to the proof of \cite[Proposition 4.5.22]{Schwede_2018}. The cited proof that $bX$ is a global $\Omega$-spectrum works level-wise, hence if $X$ is only a positive $\Omega$-spectrum, $bX$ is a positive global $\Omega$-spectrum. That $bX$ is right induced from the stable homotopy category can be seen by replacing $bX$ by the globally equivalent $\Omega \sh (bX)$ and using that the shift of a positive $\Omega$-spectrum is a $\Omega$-spectrum.\\
Moreover, $ii)$ is a formal consequence of $i)$, as indicated in the proof of \cite[4.5.22]{Schwede_2018}.
\end{proof}

Moreover, we check that $b$ is right derivable.

\begin{lemma}\label{lemma:b_preserves_equivalences}
Let $f\colon X\to Y$ be a stable equivalence between (positive) $\Omega$-spectra. Then $b(f)$ is a global equivalence.
\end{lemma}
\begin{proof}
Since $f$ is a stable equivalence between (positive) $\Omega$-spectra, it is a (positive) level equivalence. Since the $G$-space $\mathbf L(V, \mathbb R^\infty)$ is $G$-cofibrant and free for any faithful $G$-representation $V$, mapping out of it takes weak equivalences to $G$-weak equivalences. Thus, $b$ takes (positive) level equivalences to (positive) level equivalences and thus to global equivalences. Hence, $b(f)$ is a global equivalence.
\end{proof}

Using this, we describe the unit of the adjunction $\adjunction{\GH}{\SH}{U}{R}$ in a similar way to the results for the counit in Proposition \ref{prop:b_represents_right_adjoint_positively} $ii)$.

\begin{lemma}\label{lemma:unit_of_forgetful_cofree_adjunction}
Let $X$ and $Y$ be (positive) $\Omega$-spectra, and assume that $X$ is moreover (positively) cofibrant. Then, the composition 
\[ \SH(X,Y) \nameto{b} \GH(bX, bY)\nameto{(i_X)^\ast} \GH(X, bY) \]
is a bijection inverse to 
\[ \GH(X, bY) \nameto{U} \SH(X, bY) \nameto{(i_Y)_\ast\inverse} \SH(X,Y). \]
\end{lemma}
\begin{proof}
We consider the diagram
\[\begin{tikzcd}
\GH (X, bY) \arrow[d, "U", swap] & \GH(bX, bY) \arrow[l, "(i_X)^\ast", swap] \\
\SH (X, bY) & \SH(X, Y) \arrow[l, "(i_Y)_\ast"] \arrow[u, "b", swap].
\end{tikzcd}\]
As both $(i_Y)_\ast\inverse$ and $U$ are bijective, it suffices to show that this diagram is commutative, i.e. that $(i_Y)_\ast = U\circ (i_X)^\ast \circ b$. This is a consequence of the fact that $i\colon \Id \to b$ is natural.
\end{proof}

Now, we set up the double categorical context we use to prove that the above adjunction lifts to $G_\infty$- and $H_\infty$-ring spectra.\\
We first consider the double category $\mathbf{Model}$ of model categories, left Quillen functors as vertical morphisms, right Quillen functors as horizontal morphisms and all natural transformations as 2-cells. Then, \cite[Theorem 7.6]{Shulman_2011} shows that taking the homotopy category and derived functors defines a pseudo double functor into the double category $\mathbf{Sq}(\textup{Cat})$ of categories, functors as horizontal and vertical morphisms and natural transformations as 2-cells.\\
In our context, however, neither the symmetric algebra monads $\mathbb P_{\gl}$ and $\mathbb P_{\st}$ nor $b$ are Quillen functors, but merely derivable. Hence, as in Proposition \ref{prop:left_deriving_is_partial_lax_functor}, we restrict to the classes of left and right derivable functors respectively, and obtain the following result:

\begin{prop}\label{prop:deriving_is_lax_oplax_partial_functor}
Let $(\mathbf{Model}, \textup{all}, \textup{all})$ denote the double category of model categories and all functors as horizontal and vertical morphisms, and natural transformations as 2-cells. Let $\mathcal{LD}er$ be the class of left derivable functors, $\mathcal{RD}er$ denote the class of right derivable functors and $\mathcal Der_2$ denote the class of natural transformations of the form $FG \to KH$ with $G,\, K$ right derivable and $F,\, H$ left derivable. Then the assignment 
\begin{align*}
\Ho \colon (\mathbf{Model}, \mathcal{LD}er, \mathcal{RD}er, \mathcal Der_2)  \to \mathbf{Sq}(\textup{Cat})\\
\mathcal C\mapsto \Ho(\mathcal C), F\mapsto LF, G\mapsto RG, \eta\mapsto \Ho(\eta)
\end{align*}
comes equipped with the following structure:
\begin{enumerate}[i)]
	\item Unitality isomorphisms $\alpha_{\mathcal C}^v\colon id_{\Ho(\mathcal C)} \to L(id_{\mathcal C})$ and $\alpha_{\mathcal C}^h\colon R(id_{\mathcal C}) \to id_{\Ho(\mathcal C)} $ for any model category $\mathcal C$.
	\item A natural transformation $\mu_{G,F}^v\colon LG\circ LF \to L(GF)$ for any pair of left derivable functors $F\colon \mathcal C\to \mathcal D$, $G\colon \mathcal D\to \mathcal E$ such that $GF$ is also left derivable.
	\item A natural transformation $\mu_{G,F}^h\colon R(GF)\to RG\circ RF$ for any pair of right derivable functors $F\colon \mathcal C\to \mathcal D$, $G\colon \mathcal D\to \mathcal E$ such that $GF$ is also right derivable.
\end{enumerate}
These satisfy the properties of a lax-oplax double functor from Definition \ref{def:lax_oplax_double_functor} where they are defined. \\
Moreover, if $F\colon \mathcal C\to \mathcal D$ and $G\colon \mathcal D\to \mathcal E$ are right derivable and either $F$ is right Quillen or $G$ is homotopical, then $GF$ is right derivable and $\mu_{G,F}^h$ is invertible.
\end{prop}
\begin{proof}
The proof is the same as for \cite[Theorem 7.6]{Shulman_2011}, where we weaken the requirements from being Quillen to being derivable. The last statement about invertibility of $\mu_{G,F}^h$ follows from the description of this transformation as $(GF)P \nameto{Gp_{FP}} GPFP$, where $p\colon id\to P$ denotes a functorial fibrant replacement. If $G$ is homotopical, it sends the weak equivalence $p_{FP}$ to a weak equivalence. If $F$ is right Quillen, $FP$ is fibrant and thus $p_{FP}$ is a weak equivalence between fibrant objects. Since $G$ is right derivable, it then sends $p_{FP}$ to a weak equivalence.
\end{proof}

Now, we have all the ingredients to prove that we have an adjunction between $G_\infty$-ring spectra and $H_\infty$-ring spectra.

\begin{thm}\label{thm:right_adjoint_for_G_infty}
Let $R\colon \SH\to \GH$ denote the right adjoint to $U\colon \GH\to \SH$.
	\begin{enumerate}[i)]
		\item The functor $R\colon \SH\to \GH $ induces a functor $\hat R$ from the category of $H_\infty$-ring spectra to the category of $G_\infty$-ring spectra.
		\item The functor $\hat R$ is right adjoint to the forgetful functor $U$, with adjunction unit lifted from $I\colon id_{\GH} \to RU$ and adjunction counit lifted from $J\inverse\colon UR \to id_{\SH}$, where both $I$ and $J$ are obtained from deriving $i\colon id\to b$.
	\end{enumerate}
\end{thm}
\begin{proof}
\begin{enumerate}[\itshape i)]
\item We have seen that $\mathbb P_{\gl}^{\circ i}$ and $\mathbb P_{\st}^{\circ j}$ are left derivable for all $i,j\geq 0$ and that $b$ is right derivable, and moreover that $b$ has the structure of a monad morphism between $\mathbb P_{\st}$ and $\mathbb P_{\gl}$ by Corollary \ref{corollary:b_preserves_P_algebras}. Hence the above Proposition \ref{prop:deriving_is_lax_oplax_partial_functor} suffices to invoke Proposition \ref{prop:lax_double_functor_preserves_monads} to conclude that the right derived functor $R=Rb\colon \SH\to \GH$ is a monad functor between $L_{\st}\mathbb P$ and $L_{\gl}\mathbb P$. Hence, it lifts to a functor $\hat{R}\colon (H_\infty\textup{-Rings})\to (G_\infty\textup{-Rings})$.
\item We know that both $R= R(b)$ and $U= R(u)$ for the forgetful functor $u\colon \mathcal Sp_{\gl} \to \mathcal Sp_{\st}$ are monad functors. Hence both compositions $RU$ and $UR$ are monad functors.\\
Moreover, we note that both compositions $ub$ and $bu$ are right derivable, since $u$ is homotopical and sends global $\Omega$-spectra to non-equivariant $\Omega$-spectra, on which $b$ is homotopical by Lemma \ref{lemma:b_preserves_equivalences}. Moreover, they are monad functors as composites of monad functors, and hence so are the derived functors $R(ub)$ and $R(bu)$. We start by constructing $I$.\\
We define
\[I= \mu_{b,u}^h\circ \Ho(i) \circ (\alpha_{\mathcal Sp_{\gl}}^h)\inverse\colon id_{\GH} \to R(id_{\mathcal Sp_{\gl}}) \to R(bu) \to Rb\circ Ru= RU.\]
Since $i$ is a monadic transformation by Corollary \ref{corollary:b_preserves_P_algebras} and any lax-oplax double functor preserves these by Proposition \ref{prop:lax_double_functor_preserves_monads}, we see that $\Ho(i)$ is monadic. Moreover, by Lemma \ref{lemma:oplax_structure_map_is_monadic}, both $\mu_{b,u}^h$ and $\alpha_{\mathcal Sp_{\gl}}^h$ are monadic, and thus also $(\alpha_{\mathcal Sp_{\gl}}^h)\inverse$. In total, $I$ is a monadic transformation as a composition of such.

Analogously, we define $J= \mu_{u,b}^h\circ \Ho(i) \circ (\alpha_{\mathcal Sp_{\st}}^h)\inverse$, this also is a monadic transformation by the same arguments. Moreover, $J$ is invertible, since $\mu_{u,b}^h$ is by Proposition \ref{prop:deriving_is_lax_oplax_partial_functor}, $(\alpha_{\mathcal Sp_{\st}}^h)\inverse$ by definition and $\Ho(i)$ is invertible in the stable homotopy category since $i$ is a stable equivalence. Thus also $J\inverse \colon UR \to id_{\SH}$ is a monadic transformation.

Hence, the two transformations $I\colon id_{\GH}\to RU$ and $J\inverse \colon UR\to id_{\SH}$ lift to the categories of algebras. Moreover, Lemma \ref{lemma:unit_of_forgetful_cofree_adjunction} shows that for any homotopy types $X\in \GH$ and $Y\in \SH$, the morphisms 
\[ \SH(UX, Y)\nameto{R} \GH(RUX, RY) \nameto{I^\ast} \GH(X, RY)\]
and 
\[ \GH(X, RY) \nameto{U} \SH(UX, URY) \nameto{J\inverse_\ast} \SH(UX, Y)\]
are inverse isomorphisms, and thus $I$ and $J\inverse$ are unit and counit of an adjunction. As the forgetful functor from the categories of algebras to the base category is faithful, this property lifts to prove that the lifts of $I$ and $J\inverse$ are unit and counit of an adjunction
\[ \adjunction{G_\infty\textup{-Rings}}{H_\infty\textup{-Rings}}{U}{\hat{R}}.\qedhere\]
\end{enumerate}
\end{proof}

As an application of this result, we use the right adjoint to give examples of $G_\infty$-ring spectra which are not obtained as the homotopy type of an ultra-commutative ring spectrum.

\begin{remark}\label{remark:G_infty_is_not_ucom}
Let $X\in \SH$ be an $H_\infty$-ring spectrum. Then we consider the induced $G_\infty$-ring structure on the global homotopy type $RX$. This induces an $H_\infty$-ring structure on $U(RX)$. The map $J_X\colon X\to U(RX)$, defined in the proof of Theorem \ref{thm:right_adjoint_for_G_infty}, is a stable equivalence and by monadicity of the transformation an isomorphism of $H_\infty$-ring spectra.

Assume now that there is an ultra-commutative ring spectrum $Y$ such that the homotopy type of $Y$ is $RX$, and such that the $G_\infty$-ring structure is induced from the structure map $\mathbb PY\to Y$ of $Y$. Then, the $H_\infty$-ring structure on $U(RX)$ is induced by the commutative multiplication on $UY$. But the $H_\infty$-ring spectrum $U(RX)$ is equivalent to $X$. Hence, if we find an ultra-commutative representative $Y$ for the $G_\infty$-ring spectrum $RX$, then $UY$ is a commutative ring spectrum which induces the $H_\infty$-ring structure on $X$.
\end{remark}
Thus, in order to provide examples of $G_\infty$-ring spectra that are not induced by ultra-commutative ring spectra, it is enough to consider this question non-equivariantly, where counterexamples are already exhibited in the papers \cite{Noel_2014} and \cite{Lawson_2015}.

\begin{thm}[{\cite[Theorem 1.2]{Noel_2014}}]\label{thm:Example_Noel}
Let $s_k\in H^{2k}(BU; \mathbb Z_{(2)})$ be a primitive generator. Define the space $KL_k$ as the homotopy fibre
\[KL_k\nameto{i_k} BU_{(2)} \nameto{4s_k} K(\mathbb Z_{(2)}, 2k)\]
and consider the suspension spectra
\[\Sigma^\infty_+ KL_k\nameto{\Sigma^\infty_+ i_k} \Sigma^\infty_+ BU_{(2)} \nameto{\Sigma^\infty_+ 4s_k} \Sigma^\infty_+ K(\mathbb Z_{(2)}, 2k)\]
Then, for any $k$, the spectrum $\Sigma^\infty_+ KL_k$ admits the structure of an $H_\infty$-ring spectrum, and $\Sigma^\infty_+ i_k$ is an $H_\infty$-ring map. Moreover, for $k= 15$, the $H_\infty$-ring structure of $\Sigma^\infty_+ KL_{15}$ is not induced by an $E_\infty$-ring spectrum.
\end{thm}

As the homotopy category of $E_\infty$-ring spectra is equivalent to the homotopy category of commutative ring spectra, for example by \cite[Chapters II.3 and 4]{EKMM_1997}, this indeed gives rise to an example of a $G_\infty$-ring spectrum whose structure is not induced from an ultra-commutative ring spectrum.

\begin{thm}[{\cite[Theorem 1]{Lawson_2015}}]\label{thm:Example_Lawson}
Let $R_k$ be a wedge of Eilenberg-Mac Lane spectra such that $\pi_\ast R_k$ is isomorphic to the graded ring $\mathbb F_2[x]/(x^3)$, where $|x|=-2^k$. Then $R_k$ has an $H_\infty$-ring structure, and for $k>3$, these structures are not induced from commutative ring spectra.
\end{thm}

For more details on these examples, the reader is referred to the cited articles.

\subsection{Homotopical analysis of the extended powers}\label{section:homotopical_properties_P}
In this section, we generalize the analysis of the symmetric powers $\mathbb P^mX$ classically provided by Lemma \ref{lemma:extended_versus_symmetric_powers_stable}, comparing them to the extended powers $(E\Sigma_m)_+ \wedge_{\Sigma_m} X^{\wedge m}$, to the global context. Thus, we connect our definition of $G_\infty$-ring spectra using the derived monad $\mathbb G$ to the original definition of $H_\infty$-ring spectra using the extended powers. As in the $G$-equivariant version of Lemma \ref{lemma:extended_versus_symmetric_powers_stable} given in \cite[Proposition B.117]{HHR_2016}, we need to replace the universal space $E\Sigma_m$ with an appropriate global object. The correct analogue is the global universal space $E_{\gl}\Sigma_m$ defined in Remark \ref{remark:global_classifying_spaces}.

Recall that the global universal space $E_{\gl}\Sigma_m$ is constructed as $\mathbf L_V$, where $V$ is a faithful $\Sigma_m$-representation. Then, its suspension spectrum can be described by the untwisting isomorphism
\begin{equation}\label{eq:suspension_spectrum_of_EG}
\Sigma^\infty_+ \mathbf L_V\to \mathbf O(V, \_) \wedge S^V = F_V S^V
\end{equation}
as defined in \eqref{eq:untwisting_iso}. This isomorphism $\Sigma^\infty_+ \mathbf L_V \to F_V S^V$ induces the isomorphism $\Sigma^\infty_+ \mathbf L_{G,V} \cong F_{G,V} S^V$ from \eqref{eq:suspension_spectrum_of_BG} on $G$-orbits.

We also consider the morphism
\begin{equation}\label{eq:lambda_projection}
\lambda_{G,V,W}\colon F_{G, V\oplus W} S^V\to F_{G,W}
\end{equation}
defined in \cite[4.1.28]{Schwede_2018} for any compact Lie group $G$ and $G$-representations $V$ and $W$ with $W$ faithful. At an inner product space $U$, the map $\lambda_{G,V,W}$ is represented as 
\begin{align*}
\mathbf O(V\oplus W, U) \wedge_G S^V & \to \mathbf O(W, U)/G \\
[(u, \varphi), t] & \mapsto [u + \varphi(t), \varphi\circ i_2],
\end{align*}
where $i_2\colon W\to V\oplus W$ is the inclusion as the second factor. The map $\lambda_{G,V,W}$ is a global equivalence by \cite[Theorem 4.1.29]{Schwede_2018}.

\begin{thm}\label{thm:comparison_symmetric_extended_powers}
Let $X$ be a positively cofibrant orthogonal spectrum, and $n\geq 1$. Then the map 
\[q=q^X_n\colon \Sigma^\infty_+ \mathbf L_{\mathbb R^n} \wedge_{\Sigma_n} X^{\wedge n} \to X^{\wedge n}/\Sigma_n = \mathbb P^n X\]
that collapses $\Sigma^\infty_+ \mathbf L_{\mathbb R^n}$ to $\mathbb S= \Sigma^\infty_+ \ast$ is a global equivalence. 
\end{thm}
\begin{proof}
For the proof, we use a $\Sigma_n$-equivariant decomposition 
\[\Sigma^\infty_+ \mathbf L_{\mathbb R^n} \cong F_{\mathbb R^n}S^n \xleftarrow{j} (F_{\mathbb R}S^1)^{\wedge n}.\]
The isomorphism $j$ arises from the homeomorphism $(S^1)^{\wedge n}\cong S^n$ and the isomorphism $(F_{\mathbb R})^{\wedge n} \cong F_{\mathbb R^n}$ from \cite[Remark C.11]{Schwede_2018}. This decomposition is $\Sigma_n$-equivariant, since both of the involved maps are symmetric. Explicitly, this isomorphism is given at inner product spaces $U_1,\ldots , U_n$ as
\begin{align*}
(F_{\mathbb R}S^1)(U_1) \wedge \ldots \wedge (F_{\mathbb R}S^1) (U_n) & \nameto{j_{U_1, \ldots, U_n}} (F_{\mathbb R^n}S^n) (U_1\oplus \ldots \oplus U_n)\\
[(u_1, \varphi_1), t_1] \wedge \ldots \wedge [(u_n, \varphi_n), t_n] & \mapsto [(u_1 \oplus \ldots \oplus u_n, \varphi_1\oplus \ldots \oplus \varphi_n), t_1\wedge \ldots \wedge t_n].
\end{align*}

Using this decomposition, we can rewrite the domain of the morphism $q$ as follows: 
\begin{align*}
\Sigma^\infty_+ \mathbf L_{\mathbb R^n} \wedge_{\Sigma_n} X^{\wedge n} \cong & F_{\mathbb R^n}S^n \wedge_{\Sigma_n} X^{\wedge n} \\
\cong & (F_{\mathbb R}S^1)^{\wedge n} \wedge_{\Sigma_n} X^{\wedge n}\\
\cong & (F_{\mathbb R}S^1 \wedge X)^{\wedge n}/\Sigma_n
\end{align*}
We claim that under this translation, the morphism $q$ corresponds to the morphism
\[ (\lambda_{\Sigma_1, \mathbb R, 0}\wedge X)^{\wedge n}/\Sigma_n \colon (F_{\Sigma_1, \mathbb R}S^1\wedge X)^{\wedge n}/\Sigma_n \to (F_{\Sigma_1, 0}\wedge X)^{\wedge n}/\Sigma_n\]
from \eqref{eq:lambda_projection}. Note that $\Sigma_1=e$ is the trivial group, and hence $0$ is a faithful $\Sigma_1$-representation. Moreover, $F_{\Sigma_1, 0}= \mathbf O(0, \_)/\Sigma_1 \cong \mathbb S$ is the sphere spectrum.
To prove this claim, we consider the diagram
\begin{equation}\label{diagram:comparison_lambda_q}
\begin{tikzcd}
(F_{\mathbb R}S^1)^{\wedge n} \wedge_{\Sigma_n} X^{\wedge n} \arrow[r] \arrow[d, "j\wedge_{\Sigma_n} X^{\wedge n}", swap] & (F_{\mathbb R}S^1 \wedge X)^{\wedge n}/\Sigma_n \arrow[dd, "(\lambda_{\Sigma_1, \mathbb R, 0}\wedge X)^{\wedge n}/\Sigma_n"]\\
F_{\mathbb R^n}S^n \wedge_{\Sigma_n} X^{\wedge n} \arrow[d, "\textrm{untwisting}", swap] \arrow[rd, "\lambda_{\Sigma_1, \mathbb R^n, 0}\wedge X^{\wedge n}" near start] \\
\Sigma^\infty_+ \mathbf L_{\mathbb R^n}\wedge_{\Sigma_n} X^{\wedge n} \arrow[r, "q", swap] & X^{\wedge n}/\Sigma_n.
\end{tikzcd}
\end{equation}
We first consider the upper diagram. Let $U_1, \ldots, U_n$ be inner product spaces, and consider the diagram
\[\begin{tikzcd}[column sep= 5.5em]
(F_{\mathbb R}S^1) (U_1) \wedge \ldots \wedge (F_{\mathbb R}S^1)(U_n) \arrow[r, "\lambda_{\mathbb R, 0}(U_1)\wedge \ldots \wedge \lambda_{\mathbb R, 0}(U_n)"] \arrow[d, "j_{U_1, \ldots, U_n}", swap]
& \mathbf O(0, U_1)\wedge \ldots \wedge \mathbf O(0, U_n) \arrow[d, "\vcong", "\oplus" swap] \\
(F_{\mathbb R^n}S^n) (U_1\oplus \ldots \oplus U_n) \arrow[r, "\lambda_{\mathbb R^n, 0}(U_1\oplus \ldots \oplus U_n)", swap] 
& \mathbf O(0, U_1\oplus \ldots \oplus U_n) \cong S^{U_1\oplus \ldots \oplus U_n}
\end{tikzcd}\]
Evaluating this on an element yields
\[\begin{tikzcd}[column sep = small]
{[(u_1, \varphi_1), t_1]\wedge \ldots \wedge [(u_n, \varphi_n), t_n]} \arrow[r, maps to] \arrow[d, maps to]
	& {[u_1+\varphi_1(t_1)] \wedge \ldots \wedge [u_n+\varphi_n(t_n)]} \arrow[d, maps to]\\
{[(u_1\oplus \ldots \oplus u_n, \varphi_1\oplus \ldots \oplus \varphi_n), t_1\wedge \ldots \wedge t_n]} \arrow[r, maps to]
	& {[(u_1\oplus \ldots \oplus u_n) + (\varphi_1(t_1)\oplus \ldots \oplus \varphi_n(t_n))]}
\end{tikzcd}\]
By applying $(\_)\wedge_{\Sigma_n} X^{\wedge n}$ to this diagram, we see that the upper half of \eqref{diagram:comparison_lambda_q} commutes. For the second half of the Diagram \eqref{diagram:comparison_lambda_q}, let $U$ be an inner product space. We need to consider the left diagram
\[\begin{tikzcd}
\mathbf O(\mathbb R^n, U) \wedge S^n \arrow[rd, "\lambda_{\Sigma_1, \mathbb R^n, 0}(U)"] \arrow[d, "\textrm{untwisting}", swap] \\
S^U\wedge \mathbf L(\mathbb R^n, U) \arrow[r, "q(U)", swap] & \mathbf O(0, U) = S^U
\end{tikzcd}
\hspace{1em}
\begin{tikzcd}
{[(u, \varphi), t]} \arrow[rd, maps to] \arrow[d, maps to]\\
{[u+\varphi(t), \varphi]} \arrow[r, maps to] & {[u+\varphi(t)].}
\end{tikzcd}\]
On elements, this takes the right form.

Thus, also this part of the Diagram \eqref{diagram:comparison_lambda_q} commutes. Hence, we have translated the statement of the theorem into the claim that the map
\[\mathbb P^n(\lambda_{\Sigma_1, \mathbb R, 0}\wedge X)= (\lambda_{\Sigma_1, \mathbb R, 0} \wedge X)^{\wedge n} /\Sigma_n \colon \mathbb P^n (F_{\Sigma_1, \mathbb R} S^1\wedge X) \to \mathbb P^nX\]
is a global equivalence. But by \cite[Theorem 4.1.29]{Schwede_2018}, the morphism $\lambda_{\Sigma_1, \mathbb R, 0}$ is a global equivalence. As the spectrum $X$ is positively cofibrant, smashing with $X$ preserves global equivalences by \cite[Theorem 4.3.27]{Schwede_2018}. Moreover, we know by Lemma \ref{lemma:P^m_is_derivable} that $\mathbb P^n$ sends global equivalences between positively cofibrant spectra to global equivalences. As both $X$ and $F_{\Sigma_1, \mathbb R}S^1= F_{\mathbb R}S^1$ are positively cofibrant, this proves that $\mathbb P^n(\lambda_{\Sigma_1, \mathbb R, 0}\wedge X)$ is a global equivalence, and hence also $q$ is.
\end{proof}

\begin{remark}
Here, we used a decomposition of the global classifying space $\Sigma^\infty_+ \mathbf L_{\mathbb R^n}$ and the fact that we already know that $\mathbb P^n$ preserves global equivalences by Lemma \ref{lemma:P^m_is_derivable} to give an easy proof of the theorem. Our proof thus relies on the fact that we already have a model structure on commutative ring spectra, following the approach of White in \cite{White_2017} via analysing the symmetric powers $\mathbb P^n$. More classically, in \cite{EKMM_1997}, \cite{MMSS_2001} and \cite{Stolz_2011}, the theorems analogous to Theorem \ref{thm:comparison_symmetric_extended_powers} are used to provide the model structure on commutative ring spectra. In these sources, the proof of the above theorem is done by a cellular induction, see for example the proof of \cite[Theorem III.5.1]{EKMM_1997}. A similar proof can also be done in our context, using the above calculations for the induction start and then using \cite[Theorem 22]{Gorchinskiy_Guletski_2016} for the induction over the cell attachments.
\end{remark}

%% file: Chapter_Moore_spectra.tex
\section{Power operations on Moore spectra and \texorpdfstring{$\beta$}{beta}-rings}\label{section:Moore_spectra_and_beta_rings}

In this part of the article, we study $G_\infty$-structures on global Moore spectra $\mathbb SB$ for torsion-free commutative rings $B$. We show that $G_\infty$-ring structures on $\mathbb SB$ provide $\beta$-ring structures on all equivariant homotopy groups. The reason that we study Moore spectra is first of all that they are (almost) completely determined by the underlying algebra of the ring $B$, so we can translate the topological structure of being a $G_\infty$-ring spectrum into an algebraic structure on $B$. Moreover, a Moore spectrum can be thought of as an extension of coefficients of the sphere spectrum and hence has relevance to talking about cohomology theories with coefficients in a ring $B$. Hence, providing power operations in the Moore spectrum $\mathbb SB$ is a first step to providing power operations on these extended cohomology theories.

It is known classically that torsion in the ring $B$ obstructs the existence of a highly structured multiplication on the Moore spectrum $\mathbb SB$. We can also show that these obstructions occur in the algebra of global power functors. Hence, we restrict our analysis to the class of torsion-free rings, where these phenomena are not visible.

In Section \ref{section:Moore_spectra}, we study the topological side of the situation and construct from a global power structure on the homotopy groups of a Moore spectrum a $G_\infty$-ring structure. In Theorem \ref{thm:equivalence_G_Moore_spectra_A_power_algebras}, we arrive at an equivalence between the homotopy category of Moore spectra for torsion-free rings equipped with a $G_\infty$-structure and the category of corresponding global power functors. Using one direction of this relationship, which does not need the torsion-freeness assumption, we also obtain easy arguments that the Moore spectra for $\mathbb Z/n$ and $\mathbb Z[i]$ cannot support $G_\infty$-ring structures. This is a shadow of the classical facts that these spectra cannot support an $A_\infty$- or $E_\infty$-structure respectively.

Then, in an algebraic Section \ref{section:beta_rings}, we study for which rings there can be global power operations on the homotopy groups of $\mathbb SB$. For this, we link the power operations on global functors of the form $\mathbb A\otimes B$ to $\beta$-ring structures on $\mathbb A(G)\otimes B$. In particular, we give a new perspective on these objects, using a well-structured theory of global power operations to obtain $\beta$-rings. This approach is also used to obtain the $\beta$-ring structure of stable cohomotopy.

A similar analysis has already been done by Julia Singer in her PhD-thesis \cite[Chapter 2.4]{Singer_2007} in the case of $H_\infty$-structures. There, the representation ring functor $R$ takes the role of the Burnside ring. Our treatment generalizes the results to a global context.

In this chapter, all rings $B$ are commutative and unital.

\subsection{\texorpdfstring{$G_\infty$}{G-infinity}-structure on Moore spectra for global power functors}\label{section:Moore_spectra}

In this section, we construct $G_\infty$-ring structures on Moore spectra from power operations on their homotopy groups. It is well known that the existence of a multiplication on Moore spectra is a subtle question, which traces back to the fact that the cone of a map is only well defined in the homotopy category up to non-canonical isomorphism. This distinguishes them notably from Eilenberg-MacLane spectra, which are unique in a more rigorous way and hence are better-behaved for algebraic manipulations. However, there are classes of rings which work better than others, in particular, torsion-free rings support multiplications on their Moore spectra. We thus restrict our attention to the subcategory of torsion-free rings in the following chapter. The main examples of Moore spectra for torsion rings that do not support commutative multiplications are $\mathbb S(\mathbb Z/2)$, which does not admit a unital multiplication, and $\mathbb S(\mathbb Z/3)$, where Massey products obstruct the associativity of the multiplication. A remark about the latter phenomenon can be found in \cite[Example 3.3]{Angeltveit_2008}, and details for the first can be found in \cite[Theorem 1.1]{Araki_1965}. In fact, no Moore spectrum $\mathbb S(\mathbb Z/n)$ can have an $A_n$-ring structure. That no Moore spectrum $\mathbb S(\mathbb Z/n)$ can be endowed with a $G_\infty$-structure can be seen in Corollary \ref{corollary:non_existence_of_power_ops_on_Z_mod_n}, where we show that $\mathbb A\otimes\mathbb Z/n$ does not admit the structure of a global power functor. In a similar way, we also observe that the Moore spectrum for the Gaussian integers $\mathbb S(\mathbb Z[i])$ does not support a $G_\infty$-ring structure. This is a shadow of the result that this spectrum also cannot be endowed with an $E_\infty$-multiplication. Note, however, that the non-existence results for $G_\infty$-ring structures for $\mathbb S(\mathbb Z/p)$ and $\mathbb S(\mathbb Z[i])$ are obtained by easy algebraic calculations rather than by higher obstructions. In this way, the global viewpoint simplifies the calculations.

For countable torsion-free rings $B$ however, these problems vanish, and we are able to describe $G_\infty$-ring structures on Moore spectra completely algebraically. Theorem \ref{thm:equivalence_G_Moore_spectra_A_power_algebras} provides an equivalence of $G_\infty$-ring structures on $\mathbb SB$ and global power functor structures on $\mathbb A\otimes B$.

In this chapter, we use the triangulated structure of the categories $\SH$ and $\GH$. As a reference for these triangulated structures we use \cite[Chapter 4.4]{Schwede_2018}.\\

Recall from \cite[Theorem 4.5.1]{Schwede_2018} that the forgetful functor $U\colon \GH\to \SH$ has both adjoints, and that the left adjoint is the left derived functor of the identity functor $\Id\colon \mathcal Sp \to \mathcal Sp$. As such, it can be represented by assigning to a non-equivariant spectrum $X$ a non-equivariant cofibrant replacement $QX$, considered as a global homotopy type. By construction of the model structures, in fact any non-equivariant stable equivalence between cofibrant spectra is a global equivalence. We call the objects in the image of the left adjoint {\itshape left induced homotopy types}, and a spectrum {\itshape left induced (from the trivial group)} if its homotopy type is. When we talk about Moore spectra for rings, we require that all their homology is not only concentrated in degree 0, but also determined by non-equivariant data. Hence we make the following definition:

\begin{definition}\label{def:Moore_spectra}
	Let $B$ be a ring. A (global) Moore spectrum for $B$ is a connective spectrum $X$, left induced from the trivial group, such that $H_0^e(X) \cong \pi_0^e(X)\cong B$ and $H_\ast^e(X)=0$ for all $\ast\neq 0$. We denote a Moore spectrum for the ring $B$ by $\mathbb SB$.
\end{definition}

Since we study power operations on the homotopy groups of a Moore spectrum, we need to calculate $\underline{\pi}_0(\mathbb SB)$ as a global functor. This can be done in terms of $B$, using that $\mathbb SB$ is connective and left induced.

\begin{prop}\label{prop:homotopy_groups_of_left_induced_spectrum}
	Let $X$ be a connective spectrum left induced from the trivial group. Then the exterior product
	\[ \boxtimes\colon \pi_0^G(\mathbb S) \otimes \pi_0^e(X) \to \pi_0^{G\times e} (\mathbb S\wedge X) \cong \pi_0^G(X)\]
	is an isomorphism of abelian groups. As the group $G$ varies, these assemble into an isomorphism 
	\[ \underline{\pi}_0(\mathbb S) \otimes \pi_0^e(X) \to \underline{\pi}_0(X)\]
	of global functors. If $X$ is a homotopy ring spectrum, then this is an isomorphism of global Green functors.
\end{prop}
\begin{proof}
	We first observe that the proposition is true in the case $X=\mathbb S$, since in this case $\pi_0^e(\mathbb S)\cong \mathbb Z$, and the exterior product is the multiplication map 
	\[ \pi_0^G(\mathbb S)\otimes \mathbb Z \to \pi_0^G(\mathbb S),\]
	which is an isomorphism. Moreover, we note that the class of spectra $X$ for which the exterior product map is an isomorphism is closed under coproducts, since $\boxtimes$ is additive in the spectrum $X$. It is also closed under cones as defined before \cite[Proposition 4.4.13]{Schwede_2018}, by the following argument:\\
	Suppose we have a distinguished triangle
	\[ X\to Y\to Z\to X[1]\]
	of connective spectra in the stable homotopy category, where for $X$ and $Y$, the map 
	\[\boxtimes\colon \pi_0^G(\mathbb S) \otimes \pi_0^e(\_) \to \pi_0^G(\_)\]
	is an isomorphism. Then the left-induced triangle is also distinguished, and we obtain a long exact sequence in homotopy groups
	\[ \pi_0^G(X) \to \pi_0^G(Y) \to \pi_0^G(Z) \to \pi_{-1}^G(X) = 0\]
	for all compact Lie groups $G$. As the exterior product is natural, we obtain the commutative diagram
	\[\begin{tikzcd}
	\pi_0^G(\mathbb S)\otimes \pi_0^e(X) \arrow[r] \arrow[d, "\boxtimes", swap] &\pi_0^G(\mathbb S)\otimes \pi_0^e(Y) \arrow[r] \arrow[d, "\boxtimes", swap] & \pi_0^G(\mathbb S)\otimes \pi_0^e(Z) \arrow[r] \arrow[d, "\boxtimes", swap] & 0\\
	\pi_0^G(X) \arrow[r] & \pi_0^G (Y) \arrow[r] & \pi_0^G(Z) \arrow[r] & 0
	\end{tikzcd}\]
	with exact rows, where the two left vertical maps are isomorphisms. Then by the 5-lemma, also for $Z$ the exterior product is an isomorphism.\\
	Then by \cite[Proposition 4.4.13]{Schwede_2018}, respectively its non-equivariant analogue, the map $\boxtimes$ is an isomorphism for all connective left-induced spectra, since the sphere spectrum is a compact weak generator of $\SH$.
	
	Now, by the properties of the external product \cite[Theorem 4.1.22]{Schwede_2018}, the map $\boxtimes$ is a morphism of global functors, and levelwise an isomorphism, hence it is an isomorphism of global functors. Moreover, if $X$ is a homotopy ring spectrum, then $\boxtimes$ is a map of global Green functors. This proves the proposition.
\end{proof}

\begin{remark}\label{remark:identification_genuine_fixed_points_of_left_induced_spectrum}
	The above statement can also be deduced from an identification of the genuine fixed points of a left induced spectrum. Let $X$ be a non-equivariant spectrum and denote by $LX$ the corresponding left induced spectrum. Then we obtain
	\[ (LX)^G \simeq \mathbb S^G\wedge X, \]
	where $(\_)^G$ denotes the genuine fixed point functor. This relation can easily be seen using an $\infty$-categorical approach: Here, the functor $(\_)\mapsto (L\_)^G$ preserves colimits, hence we are able to calculate $(LX)^G$ by evaluating the functor on the sphere spectrum and smashing with $X$. From this formula, we deduce the above proposition by observing that $\pi_0^G(X)= \pi_0(X^G)$ and that $\pi_0$ is strict monoidal on connective spectra.
\end{remark}

\begin{corollary}\label{corollary:G_structure_induces_power_algebra}
	Let $X$ be a left induced spectrum. If $X$ supports the structure of a $G_\infty$-ring spectrum, then it induces the structure of a global power functor on $\underline{\pi}_0(X)\cong \mathbb A\otimes \pi_0^e(X)$.
\end{corollary}

Thus, we take particular interest in global power functors of the form $\mathbb A\otimes B$ in the following. These can also be described as global power functors $R$ where the multiplication map $\mathbb A\otimes R(e) \to R$, arising from the $\mathbb A$-module structure of $R$, is an isomorphism of global Green functors. We call such global power functors {\itshape left-induced} and call the full subcategory of such global power functors $\mathcal GlPow_{\mathrm{left}}$. Any morphism between such left-induced global power functors is already determined by the morphism at the trivial group $e$.

Note that a priori, $\mathbb A\otimes B$ only has the structure of a global Green functor for a ring $B$. The existence of power operations is additional structure. In Section \ref{section:beta_rings}, we exhibit a relationship between global power functor structures on $\mathbb A\otimes B$ and $\beta$-ring structures on $\mathbb A(G)\otimes B$. Moreover, we can show that for certain rings, such power operations cannot exist:

\begin{corollary}\label{corollary:non_existence_of_power_ops_on_Z_mod_n}
	The Moore spectra $\mathbb S(\mathbb Z/n)$ for $n\in \mathbb N$ and $\mathbb S(\mathbb Z[i])$ do not support a $\mathbb G_\infty$-ring structure.
\end{corollary}
\begin{proof}
	We claim that the global functor $\mathbb A\otimes \mathbb Z/n$ does not support a global power functor structure. Suppose otherwise, then we in particular obtain power operations
	\[ P^m\colon \mathbb Z/n \to \mathbb A(\Sigma_m) \otimes \mathbb Z/n\]
	for all $m\geq 0$. Note that since $\mathbb Z/n$ is additively cyclic, if a power operation exists, it is induced by the power operations on $\mathbb A$ upon taking the quotient by $n$, using additivity. Hence, we can provide formulas for the power operations in terms of powers of finite sets. We claim that if $p$ is any prime factor of $n$, then the power operation $P^p\colon \mathbb Z\to \mathbb A(\Sigma_p)$ does not descend to $\mathbb Z/n\to \mathbb A(\Sigma_p)\otimes \mathbb Z/n$.
	
	To check this, we calculate $P^p(n) \in \mathbb A(\Sigma_p)$. We observe by the explicit description of the power operations on the Burnside ring of a finite group in \cite[Example 5.3.1]{Schwede_2018} that the element $P^p(n)\in \mathbb A(\Sigma_p)$ is represented by the $\Sigma_p$-set $[n]^p$ with the permutation action, where we denote by $[n]$ the set $\{1,\ldots n\}$ with no group actions. We decompose this $\Sigma_p$-set as a disjoint union of $\Sigma_p$-orbits.
	
	In particular, we consider the free orbits in $[n]^p$. Such orbits are in bijection with equivalence classes of points $(k_1, \ldots, k_p)$ with pairwise different $k_i$, up to reordering. There are $\left(\!\myatop{n}{p}\!\right)$ such classes, hence this is the coefficient of $\Sigma_p/e$ in $[n]^p$. Since we have
	\[ \left(\!\myatop{n}{p}\!\right)= \frac{n\cdot \ldots \cdot (n-p+1)}{p\cdot \ldots \cdot 1}\]
	and in the numerator, only $n$ is divisible by $p$, we see that $\left(\!\myatop{n}{p}\!\right)$ is divisible by $\frac{n}{p}$, but not by $n$. Thus the element $P^p(n) \in \mathbb A(\Sigma_p)$ is not divisible by $n$. Hence we do not have power operations on the global functors $\mathbb A\otimes \mathbb Z/n$ for any $n$.
	
	For the case $\mathbb Z[i]$, we use the relation $(P^m(i))^2= P^m(-1)$. Then, we observe that $P^2(-1) = t-1$, where we denote $t= \tr_e^{\Sigma_2}\in \mathbb A(\Sigma_2)$. This is not a square in $\mathbb A(\Sigma_2)\otimes \mathbb Z[i]$, as we can calculate 
	\[ (at+b)^2 = 2a (a+b) t+b^2 \]
	for $a,b\in \mathbb Z[i]$, and 2 is not invertible in $\mathbb Z[i]$.
\end{proof}
Note that $\mathbb A\otimes \mathbb Z/n$ does however support truncated power operations, namely maps $P^k$ for all $k<p$, where $p$ is the smallest prime divisor of $n$. This can be shown by similar computations as above for all stabilizers of points in $[n]^p$. Such truncated power operations are also an object of research, for example in \cite{Blumberg_2018}, and relate to the fact that $\mathbb S(\mathbb Z/p)$ has an $A_{p-1}$-multiplication. Moreover, we observe that $P^2(-1)= t-1$ becomes a square after inverting $2$ in $\mathbb A\otimes \mathbb Z[i]$. In fact, there is a strict commutative model for the Moore spectrum of $\mathbb Z[i, \frac{1}{2}]$.\\

We now study the reverse direction in the case that $X$ is a global Moore spectrum. We aim to define a $G_\infty$-ring structure on $\mathbb SB$ from power operations on its homotopy groups. As already mentioned, we restrict to torsion-free rings in the following.

We first show that the ring structure of $B$ induces the structure of a homotopy ring spectrum on $\mathbb SB$. For this, we show that we can test properties of morphisms between Moore spectra on the homotopy groups. This lemma should be well-known, but since the author is not aware of a proof in the literature, it is included here for completeness' sake.

\begin{lemma}\label{lemma:functorial_Moore_spectra_free_groups}
	The functor 
	\[\pi_0^e\colon \textit{Moore}^{\textup{torsion-free}}\to \Ab^{\textup{torsion-free}}\]
	between the homotopy category of Moore spectra of torsion-free abelian groups and  the category of torsion-free abelian groups is fully-faithful, and hence an equivalence of categories.
\end{lemma}
\begin{proof}
	Let $A$ and $B$ be torsion-free groups. We have to calculate the group of morphisms $[\mathbb SA, \mathbb SB]\coloneqq \GH(\mathbb SA, \mathbb SB)$. Note that in fact $\GH(\mathbb SA, \mathbb SB) \cong \SH(\mathbb SA, \mathbb SB)$, since Moore spectra are left induced. To calculate this group of morphisms, we consider a free resolution
	\begin{equation}\label{eq:free_resolution_of_ring}
	0\to \mathbb Z^{\oplus I}\to \mathbb Z^{\oplus J}\to A \to 0
	\end{equation}
	of $A$. Then we take as a model of the Moore spectrum for $A$ the cone
	\[ \bigvee_I \mathbb S \to \bigvee_J \mathbb S\to \mathbb SA.\]
	Using this distinguished triangle and mapping into $\mathbb SB$, we obtain an exact sequence
	\[ \left[ \bigvee_J \Sigma \mathbb S, \mathbb SB\right] \to \left[\bigvee_I \Sigma \mathbb S, \mathbb SB\right] \to [\mathbb SA, \mathbb SB] \to \left[\bigvee_J \mathbb S, \mathbb SB\right]\to \left[\bigvee_I \mathbb S, \mathbb SB\right].\]
	Since the homotopy classes are additive under wedges, we can write the above sequence as
	\[ \Hom (\mathbb Z^{\oplus J}, \pi_1(\mathbb SB)) \to \Hom(\mathbb Z^{\oplus I}, \pi_1(\mathbb SB)) \to [\mathbb SA, \mathbb SB] \to \Hom (\mathbb Z^{\oplus J}, B)\to \Hom(\mathbb Z^{\oplus I}, B). \]
	As $\Hom$ is left exact and the sequence \eqref{eq:free_resolution_of_ring} is exact, we see that the kernel of the rightmost map is $\Hom(A,B)$. Moreover, the cokernel of the leftmost map is isomorphic to $\Ext^1_{\mathbb Z} (A,\pi_1(\mathbb SB))$. We now calculate $\pi_1(\mathbb SB)$, using a free resolution
	\[ 0\to \mathbb Z^{\oplus I^\prime} \to \mathbb Z^{\oplus J^\prime} \to B\to 0,\]
	which gives a cofibre sequence in the global homotopy category. The associated long exact sequence in homotopy groups gives
	\[ \bigoplus_{I^\prime} \mathbb Z/2 \to \bigoplus_{J^\prime} \mathbb Z/2 \to \pi_1(\mathbb SB) \to \bigoplus_{I^\prime} \mathbb Z \to \bigoplus_{J^\prime} \mathbb Z, \]
	using $\pi_1(\mathbb S)\cong \mathbb Z/2$. As the rightmost map is injective and tensor product is right-exact, we obtain $\pi_1(\mathbb SB)\cong B\otimes \mathbb Z/2$.
	
	Since $\Ext_{\mathbb Z}^1$ is additive, we see that $\Ext_{\mathbb Z}^1(A, B/2)$ is 2-torsion. Moreover, the sequence $0\to A\nameto{2} A\to A/2 \to 0$ is exact, as $A$ is torsion-free, and the long exact sequence of $\Ext$-functors yields that
	\[ \Ext_{\mathbb Z}^1(A, B/2) \nameto{2} \Ext_{\mathbb Z}^1(A, B/2) \to \Ext_{\mathbb Z}^2(A/2, B/2) = 0\]
	is exact. Thus $\Ext_{\mathbb Z}^1(A, B/2)$ is also 2-divisible and hence vanishes.
\end{proof}

To consider ring structures on Moore spectra, we need to calculate whether $\mathbb SB\wedge \mathbb SB$ is again a Moore spectrum. This follows from the following:

\begin{prop}\label{prop:product_of_Moore_spectra}
	Let $A$ and $B$ be commutative rings whose underlying abelian group is torsion-free. Then the external product 
	\[ H_\ast(\mathbb SA,\mathbb Z) \otimes H_\ast(\mathbb SB,\mathbb Z) \to H_\ast (\mathbb SA \wedge\mathbb SB,\mathbb Z) \]
	in homology is an isomorphism. Hence $\mathbb SA\wedge \mathbb SB$ is a Moore spectrum for $A\otimes B$.
\end{prop}
\begin{proof}
	This follows from the Künneth-theorem, see for example \cite[KT1 and Note 12]{Adams_1969}. The homology groups $H_\ast (\mathbb SA)$ and $H_\ast(\mathbb SB)$ are flat over $\mathbb Z$, since they are torsion-free. This proves that $H_\ast (\mathbb SA\wedge \mathbb SB)$ is $A\otimes B$ concentrated in degree $0$.
\end{proof}

\begin{corollary}\label{corollary:functorial_Moore_ring_spectra}
	The equivalence from Lemma \ref{lemma:functorial_Moore_spectra_free_groups} induces an equivalence 
	\[ \pi_0^e \colon \textit{Moore}^{\textup{torsion-free}}_{\textup{Rings}}\to \textup{Rings}^{\textup{torsion-free}}\]
	between the homotopy category of commutative homotopy ring Moore spectra for torsion-free rings and the category of torsion-free commutative rings.
\end{corollary}

We now follow the same approach in order to put a $G_\infty$-ring structure on $\mathbb SB$. Note that $\mathbb G(\mathbb SB)$ is not a Moore spectrum, since it is not left induced. However, we still can calculate the group $[\mathbb G(\mathbb SB), \mathbb SB]$ in terms of the homotopy groups of $\mathbb SB$. The proof of the following lemma uses representability of the equivariant homotopy groups for the case that $B$ is finitely generated and free. In order to reduce to this case, we assume countability of the group $B$.

\begin{lemma}\label{lemma:G_of_Moore_spectra_is_algebraic}
	Let $B$ be a countable torsion-free abelian group, and let $Y$ be an orthogonal spectrum such that there is an isomorphism $\underline{\pi}_1(Y)\cong R\otimes A$ of global functors, where $A$ is a commutative ring and $R$ is a global functor such that for any finite group $G$, $R(G)$ is finite. Then for any of the spectra $X= \mathbb G(\mathbb SB)$ or $X=\mathbb G(\mathbb G(\mathbb SB))$, the morphism 
	\begin{equation}\label{eq:homotopy_groups_induced_map}
	\underline{\pi}_0\colon \GH(X,Y) \to \GF (\underline{\pi}_0(X), \underline{\pi}_0(Y))
	\end{equation}
	is an isomorphism.
\end{lemma}
\begin{proof}
	We start by considering $X=\mathbb G(\mathbb SB)$, and first consider the case that $B$ is free. Then we choose a basis $(x_i)_{i\in I}$ of $B$, such that $\mathbb SB\cong \bigvee_{i\in I}\mathbb S$. We calculate
	\begin{align}\label{eq:G_of_Moore_spectrum}
	\begin{split}
	\mathbb G^m (\mathbb SB) = & \mathbb G^m\left(\bigvee_{i\in I} \mathbb S\langle x_i\rangle \right)\\
	= & \bigvee_{(m)} \bigwedge_{\myatop{i\in I}{m_i\neq 0}} \mathbb G^{m_i} \mathbb S=\bigvee_{(m)} \bigwedge_{\myatop{i\in I}{m_i\neq 0}} \Sigma^\infty_+ B_{\gl} \Sigma_{m_i} \\
	= & \bigvee_{(m)} \Sigma^\infty_+ B_{\gl} \left(\bigtimes_{\myatop{i\in I}{m_i\neq 0}} \Sigma_{m_i}\right),
	\end{split}
	\end{align}
	where $(m)= (m_i)_{i\in I}$ runs through all partitions $m= \sum_{i\in I,\, m_i\neq 0} m_i$.\\
	By representability of the homotopy groups $\pi_0^G$ by $\Sigma^\infty_+ B_{\gl} G$ from \eqref{eq:representability_of_pi_0}, the map $\underline{\pi}_0$ \eqref{eq:homotopy_groups_induced_map} is an isomorphism for the spectra 
	\[X=\Sigma^\infty_+ B_{\gl} \left(\bigtimes_{\myatop{i\in I}{m_i\neq 0}} \Sigma_{m_i}\right).\]
	As both domain and codomain of the map \eqref{eq:homotopy_groups_induced_map} are additive under wedges, also for $\mathbb G(\mathbb SB)$ the morphism $\underline{\pi}_0$ is an isomorphism.
	
	Let now $B$ be a general torsion-free abelian group. By \cite{Lazard_1964}, any flat module $M$ is isomorphic to a directed colimit of finitely generated free modules, generated on finite sets of elements in $M$. Since over $\mathbb Z$, flat and torsion-free are equivalent, we can write $B\cong\colim_{i\in I} B_i$ as a directed colimit of finitely generated free $\mathbb Z$-modules. Moreover, since $B$ is countable, we find a cofinal sequential system in the directed indexing system, so that we can write $B$ as a sequential colimit of finitely generated free modules $B_n$ with $n\in \mathbb N$. To see that this is the case, we consider the directed system $(\varphi_A\colon F_A\to B)_{A\in \mathcal S(B)}$ for the directed set $\mathcal S(B)$ of finite subsets of $B$. Here, $F_A$ is the free module on $A$. Then, $\mathcal S(B)$ is countable. We choose a bijection $s\colon\mathbb N \to \mathcal S(B)$ and iteratively define $B_n$ as follows: We set $B_1= F_{s(1)}$. Once we have defined $B_n= F_{A_n}$, we take an element $A_{n+1}$ such that both $A_n$ and $s(n+1)$ map to $A_{n+1}$ in the directed set $\mathcal S(B)$ and define $B_{n+1}= F_{A_{n+1}}$. This clearly defines a cofinal sequential subset of $\mathcal S(B)$.
	
	We can then lift this sequential system to a cofibrant system $\mathbb SB_n$ of Moore spectra, such that the colimit $\colim_n \mathbb SB_n$ models the homotopy colimit. Since homology commutes with sequential homotopy colimits, we see that $\mathbb SB\cong \colim_n \mathbb SB_n$ is a model for the Moore spectrum of $B$. Since sequential colimits of ultra-commutative ring spectra are calculated on the underlying spectra, also the functor $\mathbb P\colon \mathcal Sp\to \mathcal Sp$ commutes with sequential colimits, and thus $\mathbb G(\mathbb SB) \cong \colim_{n} \mathbb G(\mathbb SB_n)$. Then, we obtain the Milnor exact sequence
	\[ 0 \to {\lim_n}^1 \GH(\mathbb G(\mathbb SB_n), \Omega Y) \to \GH (\mathbb G(\mathbb SB), Y) \to \lim_n \GH(\mathbb G(\mathbb SB_n), Y) \to 0.\]
	By the above arguments for free modules $B$, we see that the right hand object is isomorphic to 
	\[ \lim_n \GH(\mathbb G(\mathbb SB_n), Y) \cong \lim_n \Hom_{\GF} (\underline{\pi}_0(\mathbb G(\mathbb SB_n)), \underline{\pi}_0(Y)) \cong \Hom_{\GF} (\underline{\pi}_0(\mathbb G(\mathbb SB)), \underline{\pi}_0(Y)).\]
	Similarly, the left hand object is isomorphic to
	\[ {\lim_n}^1 \GH(\mathbb G(\mathbb SB_n),\Omega Y) \cong {\lim_n}^1 \Hom_{\GF} (\underline{\pi}_0(\mathbb G(\mathbb SB_n)), \underline{\pi}_1(Y)).\]
	By the calculations \eqref{eq:G_of_Moore_spectrum}, this last term is isomorphic to 
	\[ {\lim_n}^1 \prod_{(m^n)} \pi_1^{\Sigma_{(m^n)}}(Y), \]
	where the product is over all tupels of natural numbers $m_i^n\geq 0$ indexed on a basis $I_n$ of $B_n$, and we denote $\Sigma_{(m^n)} = \bigtimes_{i\in I_n} \Sigma_{m_i^n}$. Now this $\lim^1$-term decomposes as the product
	\[ \prod_{m\geq 0} \left[ \ {\lim_n}^1 \prod_{\sum_{I_n} m_i^n=m} \pi_1^{\Sigma_{(m^n)}} (Y)\ \right] \cong \prod_{m\geq 0} \left[ \ {\lim_n}^1 \left(A \otimes \prod_{\sum_{I_n} m_i^n=m} R(\Sigma_{(m^n)})\right) \ \right]. \]
	In each of the individual $\lim^1$-terms, the inverse system consists only of tensor products of finite groups with $A$, since $I_n$ is finite for every $n$ and $R(G)$ is finite for any finite group $G$. Thus, these systems satisfy the Mittag-Leffler condition and thus the $\lim^1$-term vanishes.\\
	This proves that the morphism $\underline{\pi}_0$ is an isomorphism.
	
	The arguments for $\mathbb G(\mathbb G(\mathbb SB))$ are completely analogous.
\end{proof}

We now check that the assumptions on $\underline{\pi}_1(Y)$ are satisfied in the case of a Moore spectrum $\mathbb SB$. 

\begin{lemma}
	Let $B$ be an abelian group and $\mathbb SB$ be a Moore spectrum for $B$. Then 
	\begin{enumerate}[\itshape i)]
		\item there is an isomorphism $\underline{\pi}_1 (\mathbb SB) \cong \underline{\pi}_1(\mathbb S)\otimes B$ of global functors, given by $\boxtimes$.
		\item for any finite group $G$, the homotopy group $\pi_1^G(\mathbb S)$ is finite.
	\end{enumerate}
\end{lemma}
\begin{proof}
	Let $0\to \mathbb Z^{\oplus I} \to \mathbb Z^{\oplus J} \to B\to 0$ be a free resolution of $B$. Then, we construct the Moore spectrum $\mathbb SB$ as the mapping cone in 
	\[ \bigvee_I \mathbb S\to \bigvee_J \mathbb S \to \mathbb SB. \]
	Thus, the long exact sequence of homotopy groups becomes
	\begin{align*}
	\ldots \to &\bigoplus_I\underline{\pi}_1(\mathbb S) \to \bigoplus_J \underline{\pi}_1(\mathbb S) \to \underline{\pi}_1 (\mathbb SB) \to \\
	\to & \bigoplus_I\underline{\pi}_0(\mathbb S) \to \bigoplus_J \underline{\pi}_0(\mathbb S) \to \underline{\pi}_0 (\mathbb SB).
	\end{align*}
	Now, by freeness of $\pi_0^G (\mathbb S)\cong \mathbb A(G)$, we know that the second row of this sequence is left exact, so the first row is right exact. Moreover, tensoring with $\underline{\pi}_1(\mathbb S)$ is right exact, hence applying this to the sequence $0\to \mathbb Z^{\oplus I} \to \mathbb Z^{\oplus J} \to B\to 0$ proves $\underline{\pi}_1(\mathbb SB) \cong \underline{\pi}_1(\mathbb S)\otimes B$.
	
	For the second statement, we use the tom Dieck splitting \cite[Satz 2]{tomDieck_1975} and the Adams isomorphism \cite[Theorem 5.4]{Adams_1984}, which decompose for any compact Lie group $G$ the homotopy groups as 
	\[\pi_1^G(\mathbb S) \cong \bigoplus_{(H)\subset G} \pi_1^{W_G H} (E W_GH_+\wedge \mathbb S^H)\cong \bigoplus_{(H)\subset G} \pi_1(\Sigma^{\infty}_+ B W_GH). \]
	Here, the sum runs over conjugacy classes of closed subgroups of $G$, and $W_GH$ denotes the Weyl group of $H$ in $G$. Now the based suspension spectrum of $BW_GH$ splits stably as $\mathbb S\vee \widetilde{\Sigma}^{\infty}_+ BW_GH$, a sum of the sphere spectrum and the reduced suspension spectrum of $BW_GH$. Thus, we find
	\[ \pi_1(\Sigma^{\infty}_+ B W_GH) \cong \pi_1(\mathbb S) \oplus \pi^{\st}_1(BW_GH)\cong \mathbb Z/2\oplus \pi_0(W_GH)^{\ab}.\]
	Hence, $\pi_1^G(\mathbb S)$ is finite for any finite group $G$.  
\end{proof}

\begin{thm}\label{thm:equivalence_G_Moore_spectra_A_power_algebras}
	The functor
	\[\underline{\pi}_0\colon G_\infty\textit{-Moore}^{\textup{torsion-free}} \to {\mathcal GlPow}_{\mathrm{left}}^{\textup{torsion-free}}\]
	is an equivalence of categories between the homotopy category of $G_\infty$-Moore spectra for countable torsion-free commutative rings and the category of countable torsion-free left-induced global power functors $R$.
\end{thm}
\begin{proof}
	We first prove that the functor $\underline{\pi}_0$ from the theorem is essentially surjective. Thus, let $B= R(e)$ for a torsion-free left-induced global power functor $R\cong \mathbb A\otimes B$. We have to define a $G_\infty$-multiplication $\zeta\colon \mathbb G(\mathbb SB)\to \mathbb SB$ on $\mathbb SB$.
	
	By Lemma \ref{lemma:G_of_Moore_spectra_is_algebraic}, we know that the map 
	\[ \underline{\pi}_0\colon \GH(\mathbb G(\mathbb SB), \mathbb SB)\to \GF(\underline{\pi}_0(\mathbb G(\mathbb SB)), \underline{\pi}_0(\mathbb SB)) \]
	is an isomorphism. We have calculated the homotopy groups global functor of $\mathbb SB$ in Proposition \ref{prop:homotopy_groups_of_left_induced_spectrum} to be $\mathbb A\otimes B$. Moreover, by \cite[Theorem 5.4.11]{Schwede_2018}, we have that $\underline{\pi}_0(\mathbb G(\mathbb SB))$ is the free global power functor on the global functor $\underline{\pi}_0(\mathbb SB) \cong \mathbb A\otimes B$. We denote this free global power functor by $F(\mathbb A\otimes B)$.
	
	This free global power functor is part of a diagram
	\[ \begin{tikzcd}
	{\mathcal GlPow} \arrow[r, "U", shift left] \arrow[rr, bend left = 30, "U", shift left] 
	& \Green \arrow[r, "U", shift left] \arrow[l, "C", shift left]
	& \GF \arrow[l, "\mathbb P", shift left] \arrow[ll, bend left = 30, "F", shift left] 
	\end{tikzcd}\]
	of adjunctions, where all functors labelled $U$ are forgetful functors and right adjoints, $\mathbb P$ is the symmetric algebra functor for the box product of global functors and $C$ is the free global power functor for a global Green functor constructed in \cite[Proposition 5.2.21]{Schwede_2018}. We now claim that the composite adjunction featuring $F$ is monadic. For this, we use Beck's monadicity theorem, see for example \cite[VI.7 Theorem 1]{MacLane_1971}.
	
	Let $R\rightrightarrows S$ be a pair of parallel arrows in ${\mathcal GlPow}$ that has a split coequalizer in $\GF$. Then by monadicity of the adjunction $\adjunction{\Green}{\GF}{U}{\mathbb P}$, the coequalizer in $\GF$ has the unique structure of a global Green functor such that it is a coequalizer of $R\rightrightarrows S$ in $\Green$. Since ${\mathcal GlPow}$ is also comonadic over $\Green$ by \cite[Theorem 5.2.13]{Schwede_2018}, colimits in ${\mathcal GlPow}$ are created by $U\colon {\mathcal GlPow}\to \Green$. Thus, Beck's monadicity theorem shows that the adjunction $\adjunction{\mathcal GlPow}{\GF}{U}{F}$ is monadic. We denote the associated monad $UF$ also by $F$.
	
	Hence, the power functor structure on $R \cong \mathbb A\otimes B$ is equivalent to a morphism $\tau\colon F(\mathbb A\otimes B) \to \mathbb A\otimes B$, satisfying the compatibility conditions with the monad structure on $F$. Since $\underline{\pi}_0$ \eqref{eq:homotopy_groups_induced_map} is an isomorphism for $\mathbb G(\mathbb SB)$ and $\mathbb SB$, this morphism $\tau\colon F(\mathbb A\otimes B)\to \mathbb A\otimes B$ is the image of a unique morphism $\zeta\colon \mathbb G(\mathbb SB)\to \mathbb SB$ under the functor $\underline{\pi}_0$. We claim that $\zeta$ endows $\mathbb SB$ with the structure of a $G_\infty$-ring spectrum.
	
	For this, we check that the functor $\underline{\pi}_0\colon \GH_{\geq 0}\to \GF$ sends the monad diagrams for $\mathbb G$ to those of $F$. Here, the subscript $\geq 0$ denotes the full subcategory on the connective spectra. For this, we consider the diagrams 
	\begin{equation}\label{diagram:homotopy_groups_are_monad_functor_for_power_operations}
	\begin{tikzcd}
	\GH_{\geq 0} \arrow[r, "L_{\gl}\mathbb P"] \arrow[d, "\underline{\pi}_0", swap] & \Ho(\ucom)_{\geq 0}\arrow[d, "\underline{\pi}_0"]\\
	\GF \arrow[r, "F", swap] & \mathcal GlPow
	\end{tikzcd}
	\textup{ and }
	\begin{tikzcd}
	\Ho(\ucom)_{\geq 0} \arrow[r, "U"] \arrow[d, "\underline{\pi}_0", swap] & \GH_{\geq 0}\arrow[d, "\underline{\pi}_0"]\\
	\mathcal GlPow \arrow[r, "U", swap] & \GF.
	\end{tikzcd}
	\end{equation}
	The right diagram commutes, thus we get a natural transformation $\rho\colon F\circ \underline{\pi}_0 \to \underline{\pi}_0 \circ L_{\gl}\mathbb P$ in the left diagram as the mate of the right isomorphism (see \cite[Proposition 2.1]{Kelly_Street_1974} for the definition of mates). This transformation is thus defined by freeness of $F$, and \cite[Theorem 5.4.11]{Schwede_2018} proves that $\rho$ is a natural isomorphism. The pasting of the two diagrams in \eqref{diagram:homotopy_groups_are_monad_functor_for_power_operations}, using the inverse of the transformation $\rho$, then exhibits $\underline{\pi}_0\colon \GH_{\geq 0} \to \GF$ as a monad functor between the monads $\mathbb G$ and $F$. For the compatibility with the monad structure, naturality of mates is used.
	
	Thus, we see that the $G_\infty$-diagrams
	\[ \begin{tikzcd}
	\mathbb G(\mathbb G(\mathbb SB)) \arrow[r, "\mathbb G(\zeta)"] \arrow[d, "\mu", swap] & \mathbb G(\mathbb SB) \arrow[d, "\zeta"]\\
	\mathbb G(\mathbb SB) \arrow[r, "\zeta", swap] & \mathbb SB
	\end{tikzcd}
	\textup{ and }
	\begin{tikzcd}
	\mathbb SB \arrow[r, "\eta"] \arrow[rd, equals] & \mathbb G(\mathbb SB) \arrow[d, "\zeta"] \\
	& \mathbb SB
	\end{tikzcd}\]
	are sent under $\underline{\pi}_0$ to the corresponding diagrams for the monad $F$. Since $\tau$ defines the structure of an $F$-algebra on $\mathbb A\otimes B$ by assumption, the monad diagrams for $\mathbb A\otimes B$ commute. Using Lemma \ref{lemma:G_of_Moore_spectra_is_algebraic} for morphisms out of $\mathbb G(\mathbb G(\mathbb SB))$ and Lemma \ref{lemma:functorial_Moore_spectra_free_groups} for $\mathbb SB$, also the $G_\infty$-diagrams for $\mathbb SB$ commute. This proves that 
	\[ \underline{\pi}_0\colon G_\infty\textit{-Moore}^{\textup{torsion-free}} \to {\mathcal GlPow}_{\mathrm{left}}^{\textup{torsion-free}}\]
	is essentially surjective.
	
	To check that $\underline{\pi_0}$ is also fully faithful, we only need to check that the unique induced map $\mathbb SB\to \mathbb SC$ from Lemma \ref{lemma:functorial_Moore_spectra_free_groups} for a map $f\colon \mathbb A\otimes B\to \mathbb A\otimes C$ of left induced global power functors is a morphism of $G_\infty$-ring spectra. But this again follows from Lemma \ref{lemma:G_of_Moore_spectra_is_algebraic} by looking at the diagrams
	\[\left( \begin{tikzcd}
	\mathbb G(\mathbb SB) \arrow[r, "\zeta_B"] \arrow[d] & \mathbb SB\arrow[d] \\
	\mathbb G(\mathbb SC) \arrow[r, "\zeta_C", swap] & \mathbb SC
	\end{tikzcd}\right)
	\mapsto
	\left( \begin{tikzcd}
	F(\mathbb A\otimes B) \arrow[r, "\tau_B"] \arrow[d, "Ff"] & \mathbb A\otimes B\arrow[d, "f"] \\
	F(\mathbb A\otimes C) \arrow[r, "\tau_C", swap] & \mathbb A\otimes C
	\end{tikzcd}\right). \]
	Here, the right diagram commutes by assumptions on $f$, so also the left diagram commutes. In total, the functor 
	\[\underline{\pi_0}\colon G_\infty\textit{-Moore}^{\textup{torsion-free}} \to {\mathcal GlPow}_{\mathrm{left}}^{\textup{torsion-free}}\]
	is an equivalence of categories.
\end{proof}

\subsection{The relation to \texorpdfstring{$\beta$}{beta}-rings}\label{section:beta_rings}

We now connect the theory of global power functors to the theory of $\beta$-rings. For this, we use the perspective of global power functors as coalgebras over the comonad $\exp$ on the category of global Green functors, see \cite[Chapter 5.2]{Schwede_2018}. Thus, a global power functor $R$ comes with power operations $R(G)\to \exp(R;G) \subset \prod_{m\geq 0}R(\Sigma_m\wr G)$. On the other hand, a $\beta$-ring $A$ has power operations indexed by the Burnside rings of symmetric groups, given as maps $A\otimes \mathbb A(\Sigma_m)\to A$. In order to obtain such structure from the power operations on a global power functor, we assume that the global power functor $R$ comes equipped with {\itshape deflation maps} $R(K\times G)\times \mathbb A(K)\to R(G)$. These allow to dualize the power operations to obtain $\beta$-operations on the values $R(G)$.

It is a classical observation that the Burnside rings in fact support such deflations. Thus, we can apply the theory presented in this section to the left induced global power functors $\mathbb A\otimes B$, which by Theorem \ref{thm:equivalence_G_Moore_spectra_A_power_algebras} completely parametrize $G_\infty$-ring structures on Moore spectra $\mathbb SB$. Such structures are hence tied to $\beta$-ring structures on $\mathbb A(G)\otimes B$, extending the classical $\beta$-ring $\mathbb A(G)$.\\
Another example possessing the necessary deflations is given by stable cohomotopy $\underline{\pi}^0(X)$ for any based space $X$. Here, the deflations can be constructed by the theory of polynomial operations, as described in \cite{Vallejo_1990b} and \cite{Guillot_2006}. Thus, our approach of providing $\beta$-operations via global power operations and deflations gives back the classical examples of $\beta$-rings.

Our approach is loosely based on the discussion of $\tau$-rings in \cite[Section 4.2]{Ganter_2013}, which itself goes back to \cite{Hoffman_1979}.\\
Note that we could also consider deflations indexed by arbitrary global power functors $T$ instead of $\mathbb A$. For the representation ring global power functor, this yields the theory of $\lambda$-rings. However, there are not many other examples of global power functors supporting deflations, so we focus on the Burnside ring in this work.

\begin{remark}\label{remark:history_of_beta_rings}
We give a short overview over the history of the notion of $\beta$-rings.\\
The notion of a $\beta$-ring was first introduced by Rymer in \cite{Rymer_1977}, based on the question posed by Boorman in \cite{Boorman_1975} whether there is a theory of $\beta$-rings which formalizes the $\beta$-operations on the Burnside ring defined in this work. Rymer, however, did not define his operator ring structure on $\mathbf B= \bigoplus_{m\geq 0} \mathbb A(\Sigma_m)$ properly, a fact explained and amended by Ochoa in \cite{Ochoa_1988}. An explicit construction of the $\beta$-operations, using the language of polynomial operations, is given by Vallejo in \cite{Vallejo_1990}, and he extended the definition of a $\beta$-ring by a unitality condition in \cite{Vallejo_1993}.\\
Lastly, the survey article \cite{Guillot_2006} of Guillot provides more details on the history of $\beta$-rings and their connection to $\lambda$-rings, as well as showing that stable cohomotopy is an example of a $\beta$-ring. Moreover, there an additivity condition is added to the notion of $\beta$-rings. 
\end{remark}

A $\beta$-ring encodes power operations indexed by Burnside rings. This is formalized by a certain operator ring:

\begin{definition}\label{def:universal_Tau_ring}
We denote
\[ \mathbf B=\bigoplus_{m\geq 0} \mathbb A(\Sigma_m).\]
We endow this abelian group with a commutative multiplication via 
\[ x\cdot y= \tr_{\Sigma_k\times \Sigma_l}^{\Sigma_{k+l}} (x\times y)\]
for $x\in \mathbb A(\Sigma_k)$ and $y\in \mathbb A(\Sigma_l)$. This defines a ring structure on $\mathbf B$. Moreover, we define an operation $\ast$ on $\mathbf B$ as follows: Let $x\in \mathbb A(\Sigma_k)$ and $y_i\in \mathbb A(\Sigma_{l_i})$ for $1\leq i\leq n$. Then we define
\begin{equation}\label{eq:tau_ring_plethysm}
x\ast (y_1+ \ldots + y_n)= \sum_{(k)} \tr_{\bigtimes \Sigma_{k_i} \wr \Sigma_{l_i}}^{\Sigma_{(k)\cdot (l)}} \left( \bigtimes_{i=1}^n P^{k_i}(y_i) \cdot ((\bigtimes \Sigma_{k_i}\wr p_{\Sigma_{l_i}})^\ast \Phi_{(k)}^\ast )(x)\right),
\end{equation}
where the sum runs over all partitions $(k)=(k_i)_{i=1,\ldots, n}$ of $k$, we denote $(k)\cdot (l)\coloneqq \sum_{i=1}^{n} k_il_i$ and the transfer is along the monomorphisms $\Psi_{k_i, l_i}$ and $\Phi_{(k_il_i)_i}$ \cite[2.2.5-6]{Schwede_2018}. In the map $\bigtimes \Sigma_{k_i}\wr p_{\Sigma_{l_i}}\colon \bigtimes \Sigma_{k_i}\wr \Sigma_{l_i}\to \bigtimes \Sigma_{k_i}$, the product runs over $i=1$ to $n$ and the map $p_{\Sigma_{l_i}}\colon\Sigma_{l_i}\to e$ is the unique map to the trivial group.\\
The operation $\ast$ is additive in the first component and can be extended linearly to give a map $\ast \colon \mathbf B\times \mathbf B \to \mathbf B$.
\end{definition}
The operation $\ast$ is sometimes called plethysm, for example in \cite{Ochoa_1988} and in \cite{Hoffman_1979} for the representation ring instead of the Burnside ring. It makes $\mathbf B$ into an operator ring by \cite[Theorem 1.11]{Vallejo_1993}. We then define a $\beta$-ring as an {\itshape operator module} over $\mathbf B$. The following definition is given in \cite[Definition 1.12]{Vallejo_1993}. Note that here and in the following, we denote by $\Map$ the collection of maps of \emph{sets}. In contrast, $\Hom$ in the following denotes ring homomorphisms.

\begin{definition}\label{def:tau_rings}
A $\beta$-ring is a commutative ring $A$ together with a map $\vartheta \colon \mathbf B\to \Map(A, A)$ such that, for all $a\in A$ and $x,y\in \mathbf B$, the following relations hold:
\begin{enumerate}[\itshape i)]
	\item $\vartheta(x+y)=\vartheta(x)+\vartheta(y)$
	\item $\vartheta(x\cdot y)=\vartheta(x)\cdot \vartheta(y)$
	\item $\vartheta(x\ast y)=\vartheta(x)\circ \vartheta(y)$
	\item $\vartheta(1)(a)=1$ for the multiplicative unit $1\in \mathbb A(\Sigma_0)\subset \mathbf B$ on the left and $1\in A$ on the right
	\item $\vartheta(e) = id_{A}$, where $e=1\in \mathbb A(\Sigma_1)\subset \mathbf B$ is a unit for the operation $\ast$.
\end{enumerate}
In the first two statements, we use the pointwise ring structure on $\Map(A, A)$.
\end{definition}

For the construction of a $\beta$-ring structure on $R(G)$ for a global power functor $R$, we need an additional structure in the form of a pairing with the Burnside ring.

\begin{definition}\label{def:global_power_functor_with_pairing}
Let $R$ be a global power functor. Maps 
\[ \langle \_, \_ \rangle_{K,G}\colon R(K\times G) \times \mathbb A(K)\to R(G), \]
defined for all compact Lie groups $K$ and $G$, are called $\mathbb A$-deflations if the maps $\langle\_,\_\rangle_{K,G}$ are biadditive and satisfy 
\begin{enumerate}[{\itshape i)}]
	\item $\langle (K\times \alpha)^\ast r,x\rangle_{K,G} = \alpha^\ast \langle r, x\rangle_{K, L}$ for any continuous group homomorphism $\alpha\colon G\to L$.
	\item $\langle \tr_{L\times G}^{K\times G} r, x \rangle_{K,G}= \langle r, \res^K_L x\rangle_{L, G}$ for any closed subgroup $L\subset K$, and the reversed relation also holds.
	\item $\langle (\alpha\times G)^\ast r, \alpha^\ast x\rangle = \langle r, x\rangle$ for $r\in R(K\times G)$, $x\in \mathbb A(K)$ and a surjective group homomorphism $\alpha\colon L\to K$.
	\item $\langle r, 1\rangle = r$ for all $r\in R(G)$ and $1\in \mathbb A(e)$.
	\item $\langle r\cdot s, x\cdot y\rangle_{K,G}= \langle r, x\rangle_{K,G}\cdot \langle s, y\rangle_{K,G}$ for all $r,s\in R(K\times G)$ and $x,y\in \mathbb A(K)$.
	\item $\langle r\cdot \pr_K^\ast y, x\rangle_{K, G} = \langle r, y\cdot x\rangle_{K, G}$ for all $r\in R(K\times G)$ and $x,y\in \mathbb A(K)$.
	\item $\langle (\delta_n^G)^\ast P^n\langle r, y\rangle, x\rangle = \langle (\delta_n^{\Sigma_k, G})^\ast (P^n(r)\cdot (\Sigma_n\wr \pr_{\Sigma_k})^\ast P^n(y)), (\Sigma_n\wr p_{\Sigma_k})^\ast x\rangle$ for all $r\in R(\Sigma_k\times G)$, $y\in \mathbb A(\Sigma_k)$ and $x\in \mathbb A(\Sigma_n)$. Here, we considered the diagonal inclusion
	\begin{equation}\label{eq:diagonal_inclusion_wreath_product}
	\delta_n^G\colon \Sigma_n\times G\to \Sigma_n\wr G, \, (\sigma, g)\mapsto (\sigma; g,\ldots, g)
	\end{equation} 
	and the relative version $\delta_n^{\Sigma_k, G}\colon (\Sigma_n\wr \Sigma_k)\times G\to \Sigma_n\wr(\Sigma_k\times G)$.
\end{enumerate}
We denote by $\mathcal GlPow_{\mathbb A\textup{-defl}}$ the category of global power functors with $\mathbb A$-deflations and morphisms of global power functors compatible with the pairing $\langle\_,\_\rangle$.
\end{definition}

\begin{remark}\label{remark:connection_deflation_pairing_actual_deflations}
	In practice, such a deflation pairing on a global power functor often arises from actual deflations, i.e. maps $\varphi_\ast \colon R(G)\to R(K)$ for surjective group homomorphisms $\varphi\colon G\to K$. These satisfy certain relations, as exhibited in \cite{Bouc_2010} for finite groups, and packaged in \cite[Chapter 8]{Webb_2000} in the notion of a globally defined Mackey functor. The compatibility of the deflations with the power operations also explains the compatibility condition of the above deflation pairing with the power operations. \\
	In the following, we only use the deflation pairing on a global power functor $R$.
\end{remark}

\begin{lemma}
For the left-induced global power functor $\mathbb A\otimes C$ for a ring $C$, the composition
\[\mathbb A(K\times G)\otimes C\otimes \mathbb A(K)\nameto{\times} \mathbb A(K\times G\times K)\otimes C \nameto{\Delta_K^\ast} \mathbb A(K \times G)\otimes C \nameto{(\pr_G)_\ast} \mathbb A(G)\otimes C \]
defines an $\mathbb A$-deflation. Here $\Delta_K\colon K\times G\to K\times G\times K$ is the diagonal of $K$, and $(\pr_G)_\ast$ denotes the deflation along $\pr_G\colon K\times G\to G$ present in the Burnside rings.
\end{lemma}
We omit the calculations of the properties of this pairing. We recall however the definition of the deflations, since they are the main part of this construction:

\begin{construction}\label{constr:deflations_for_Burnside_rings}
In the case of finite groups, the deflations in the Burnside ring global functor are described as follows: Let $f\colon G\to K$ be a group homomorphism, and $X$ be a finite $G$-set. Then we define $f_\ast (X) = K\times_G X=(K\times X)/G$, where we identify $(k\cdot f(g), x)$ with $(k,gx)$ for all $g\in G$, and consider $f_\ast(X)$ as a $K$-set by the multiplication on the $K$-factor. This defines a map $\mathbb A(G)\to \mathbb A(K)$. For the general case of a morphism $G\to K$ of compact Lie groups, this construction is generalized in \cite[Proposition 20]{tomDieck_1975} and \cite[Proposition IV.2.18]{tomDieck_1987}.
\end{construction}

\begin{remark}\label{remark:pairing_on_cohomotopy}
Another main example of a global functor with $\mathbb A$-deflations is stable cohomotopy $\underline{\pi}^0(X)$ for a based space $X$. By Remark \ref{remark:power_operations_on_cohomology}, we have restricted power operations on the equivariant stable cohomotopy defined by 
\[ \pi^0_G(X) \nameto{P^m} \pi^0_{\Sigma_m\wr G} (X^m) \nameto{(\delta_m^G)^\ast} \pi^0_{\Sigma_m\times G}(X^m) \nameto{\Delta^\ast} \pi^0_{\Sigma_m\times G} (X).\]
As explained below, these restricted power operations, using $\delta_m^G$, are also used to obtain the $\beta$-operations. Hence, the theory developed below is also applicable to equivariant stable cohomotopy.\\
Moreover, we obtain a pairing $\pi^0_{K\times G} (X)\times \mathbb A(K)\to \pi^0_G(X)$ by generalizing the definition for $\mathbb A$ to an arbitrary base. We again restrict to finite groups. We may define operations on stable cohomotopy by defining additive (or only polynomial) operations on the monoid $\Cov^+(X)$ of isomorphism classes of (equivariant) coverings over $X$. Non-equivariantly, this is \cite[Theorem 2.4]{Vallejo_1990b}, building on the observations by Segal in \cite{Segal_1974a}. The main observation for this is that stable cohomotopy can be seen as the group completion of the monoid of coverings over $X$, which is a consequence of the Barratt-Priddy-Quillen theorem \cite{Barratt_1971, Priddy_1971, Segal_1974}. The equivariant statement holds analogously, using the equivariant Barratt-Priddy-Quillen theorem established in \cite{Segal_1971, Guillou_May_2017, Barwick_2017}.

Hence, let $E\to X$ be a $(K\times G)$-equivariant covering of $X$ and $T$ be a finite $K$-set. Then we define $\langle E\to X, T\rangle$ to be $E\times_K T$, which is a $G$-equivariant covering. One can check that this is biadditive and induces a pairing $\pi^0_{K\times G} (X)\times \mathbb A(K)\to \pi^0_G(X)$ as required.
\end{remark}

Suppose now that $R$ is a global power functor with $\mathbb A$-deflations. We are ultimately interested in a pairing $R(\Sigma_m\wr G)\otimes \mathbb A(\Sigma_m) \to R(G)$, hence we use the morphism $\delta_m^G\colon \Sigma_m\times G\to \Sigma_m\wr G$ from \eqref{eq:diagonal_inclusion_wreath_product} to define
\[ R(\Sigma_m\wr G)\otimes \mathbb A(\Sigma_m)\nameto{(\delta_m^G)^\ast \otimes id} R(\Sigma_m\times G)\otimes \mathbb A(\Sigma_m) \nameto{\langle\_,\_\rangle_{\Sigma_m, G}} R(G). \]

Using these definitions, we can define for any global power functor $R$ with $\mathbb A$-deflations a morphism
\begin{align}
\begin{split}\label{eq:duality_morphism_tau_rings}
D_G\colon \exp(R; G) &\to \Map(\mathbf B, R(G))\\
x=(x_n)_{n\geq 0} & \mapsto \left( (y_n)\mapsto \sum_{n\geq 0} \langle (\delta_n^G)^\ast x_n, y_n\rangle \right).
\end{split}
\end{align}

\begin{prop}\label{prop:duality_morphism_lands_in_ring_morphisms}
Let $R$ be a global power functor with $\mathbb A$-deflations and $G$ be a compact Lie group. Then for any $x \in \exp(R;G)$, the morphism $D_G(x)$ is a ring homomorphism.
\end{prop}
\begin{proof}
Additivity of the morphism $D_G(x)\colon \mathbf B \to R(G)$ is clear from biadditivity of the pairing $\langle \_,\_\rangle$. Moreover, let $1\in \mathbb A(\Sigma_0)$ be the multiplicative unit of $\mathbf B$. Then we calculate 
\[ D_G(x) (1) = \langle (\delta_0^G)^\ast x_0, 1\rangle =p_G^\ast(x_0) =1,\]
using that exponential sequences have as zeroth term the unit of $R(e)$ and that $\delta_0^G$ is the unique map $p_G\colon G\to e$.

Now, we check that the map $D_G(x)$ is multiplicative. Recall that the product on $\mathbf B$ is the transfer product from Definition \ref{def:universal_Tau_ring}. For two elements $y, z \in \mathbf B$, we have 
\begin{align*}
D_G(x) (y\cdot z) = & \sum_{n\geq 0} \Big\langle (\delta_n^G)^\ast x_n,\sum_{k+l=n} \tr_{\Sigma_k\times \Sigma_l}^{\Sigma_{k+l}} (y_k\times z_l) \Big\rangle \\
= & \sum_{k,l\geq 0} \langle \res^{\Sigma_{k+l}\times G}_{\Sigma_k\times \Sigma_l\times G} (\delta_{k+l}^G)^\ast x_{k+l}, y_k\times z_l\rangle\\
= & \sum_{k,l\geq 0} \langle \Delta_G^\ast ((\delta_k^G)^\ast x_k\times (\delta_l^G)^\ast x_l), y_k\times z_l\rangle \\
= & \sum_{k,l\geq 0} \langle (\delta_k^G)^\ast x_k, y_k\rangle \cdot \langle (\delta_l^G)^\ast x_l, z_l\rangle\\
= & D_G(x)(y) \cdot D_G(x)(z).
\end{align*}
Here, we use that the morphisms $\delta_{k+l}^G \circ (\Phi_{k,l}\times G)$ and $\Phi_{k,l}^G \circ (\delta_k^G\times \delta_l^G) \circ \Delta_G\colon \Sigma_k\times \Sigma_l \times G\to \Sigma_{k+l}\wr G$ agree and that the sequence $x$ is exponential.
\end{proof}

Using this morphism, we can now study how any global power functor $R$ with $\mathbb A$-deflations induces the structure of a $\beta$-ring on $R(G)$:

\begin{construction}\label{constr:structure_map_of_tau_rings}
Let $R$ be a global power functor with $\mathbb A$-deflation and let $G$ be a compact Lie group. Then we define 
\[ \bar\vartheta_G \colon R(G)\nameto{P} \exp(R; G) \nameto{D_G} \Hom_{\textup{Rings}} (\mathbf B, R(G)), \]
where $P$ denotes the power operation on $R$.\\
This map is adjoint to a map
\[ \vartheta \colon \mathbf B\to \Map(R(G),R(G)).\]
\end{construction}

\begin{prop}\label{prop:pre_tau_rings_from_powered_algebras}
The map $\vartheta$ makes $R(G)$ into a $\beta$-ring.
\end{prop}
\begin{proof}
This proof is a lengthy calculation, using the properties of the global power operations and the pairing $\langle\_, \_\rangle$. An essentially similar calculation in the case $R=\mathbb A$ has been carried out in \cite[Theorem 2]{Rymer_1977}, where the relations are only checked for some additive generators of $\mathbf B$, and in \cite[Corollary 1.16]{Vallejo_1993}, where these calculations are extended to the entirety of $\mathbf B$ using the theory of polynomial operations. Also, a similar calculation can be found in \cite{Ochoa_1988}.
\end{proof}

\begin{corollary}\label{corollary:beta_ring_structure_on_Burnside_ring_and_cohomotopy}
The rings $\mathbb A(G)$ and $\pi^0_G(X)$ for any based space $X$ carry the structure of $\beta$-rings.
\end{corollary}

\begin{corollary}
Let $C$ be a commutative ring. If $\mathbb A\otimes C$ supports the structure of a global power functor, then $\mathbb A(G)\otimes C$ inherits the structure of a $\beta$-ring.
\end{corollary}

Note that the $\beta$-ring structures on Burnside rings and stable cohomotopy rings are already known by the classical literature. However, the above definition of the structure map $\vartheta$ illustrates how this structure can be obtained from a more naturally arising structure, namely from a global power functor structure on $R$ together with deflations. We hope that this allows for a more insightful picture of $\beta$-rings.

The definition of a $\beta$-ring only incorporates relations between the different operations $\vartheta(x)$ on $R(G)$, indexed by $x\in \mathbf B$. It does not provide any compatibility of these operations with the ring structure on $R(G)$. We now add a condition of ``external'' additivity, due to \cite{Guillot_2006}:

\begin{definition}\label{def:iterated_universal_tau_ring}
We define 
\[ \mathbf B^2=\bigoplus_{p,q\geq 0} \mathbb A(\Sigma_p\times \Sigma_q). \]
This has a ring structure analogous to the one on $\mathbf B$, given by $x\cdot y= \tr_{\Sigma_p\times \Sigma_q\times\Sigma_r\times \Sigma_s}^{\Sigma_{p+r} \times \Sigma_{q+s}}(x\times y)$ for $x\in \mathbb A(\Sigma_p \times \Sigma_q)$ and $y\in \mathbb A(\Sigma_r \times \Sigma_s)$. Moreover, we have maps
\[ \Phi\colon \mathbf B\to \mathbf B^2, \, x\mapsto \sum_{p+q=m} \Phi^\ast_{p,q} x \textup{ for } x\in \mathbb A(\Sigma_m)\]
and
\[ \times\colon \mathbf B\otimes \mathbf B\to \mathbf B^2, \, x\otimes y\mapsto x\times y \textup{ for } x\in \mathbb A(\Sigma_p), y\in \mathbb A(\Sigma_q). \]
\end{definition}

Additivity is then expressed by a morphism $\vartheta^2\colon \mathbf B^2\to \Map(A\times A, A)$ analogous to $\vartheta$, using the map $\Phi$.

\begin{definition}\label{def:additive_tau_ring}
An additive $\beta$-ring is a commutative ring $A$ with maps $\vartheta \colon \mathbf B\to \Map(A, A)$ and $\vartheta^2 \colon \mathbf B^2\to \Map(A\times A, A)$ such that $(A, \vartheta)$ is a $\beta$-ring and the following properties hold:
\begin{enumerate}[\itshape i)]
	\item $\vartheta^2(x\times y)(c,d) = \vartheta(x)(c)\cdot \vartheta(y)(d)$ for all $x,y\in \mathbf B$ and all $c,d\in A$.
	\item $\vartheta(x) (c+d) = \vartheta^2(\Phi x)(c,d)$ for all $x\in \mathbf B$ and $c,d\in A$.
\end{enumerate}
\end{definition}

We construct such a map $\vartheta^2$ for $R(G)$ when $R$ a global power functor with $\mathbb A$-deflations.

\begin{construction}\label{constr:iterated_duality_and_structure_morphism}
Let $R$ be a global power functor with $\mathbb A$-deflations. Then we define
\begin{align*}
D^2\colon \exp(R; G)\times  \exp(R; G) &\to \Hom(\mathbf B^2, R(G))\\
(x,y)&\mapsto \left(z\mapsto \sum_{p,q\geq 0} \langle (\delta_{p,q}^G)^\ast (x_p\times y_q), z_{p, q}\rangle\right),
\end{align*}
where $\delta_{p,q}^G=(\delta_p^G\times \delta_q^G)\circ \Delta_G\colon \Sigma_p\times \Sigma_q\times G\to \Sigma_p\wr G\times \Sigma_q\wr G$ is the diagonal on $G$. This in fact takes values in ring homomorphisms by a similar argument to Proposition \ref{prop:duality_morphism_lands_in_ring_morphisms}.\\
Moreover, we define
\[\bar\vartheta^2 \colon R(G)\times R(G)\nameto{P\times P} \exp(R; G)\times  \exp(R; G) \nameto{D^2} \Hom(\mathbf B^2, R(G)),\]
and denote the morphism adjoint to $\bar\vartheta^2$ as  
\[ \vartheta^2\colon \mathbf B^2\to \Map(R(G)\times R(G), R(G)).\]
\end{construction}

\begin{prop}\label{prop:additivity_for_tau_rings}
The morphisms $\vartheta$ and $\vartheta^2$ make $R(G)$ into an additive $\beta$-ring.
\end{prop}
\begin{proof}
We only prove part $ii)$ from Definition \ref{def:additive_tau_ring}, since the first assertion is an easy calculation, using the description $\delta_{p,q}^G= (\delta_p^G\times \delta_q^G)\circ \Delta_G$.\\
We thus calculate
\begin{align*}
\vartheta(x)(c+d) = & \sum_{m\geq 0} \langle (\delta_{m}^G)^\ast P^m(c+d), x_m\rangle \\
= & \sum_{m\geq 0} \Big\langle (\delta_m^G)^\ast \sum_{p+q=m} \tr_{\Sigma_p\wr G\times \Sigma_q\wr G}^{\Sigma_{p+q}\wr G} (P^p(c)\times P^q(d)), x_m\Big\rangle \\
= & \sum_{p,q\geq 0} \langle \tr_{\Sigma_p\times \Sigma_q\times G}^{\Sigma_{p+q}\times G} (\delta_{p,q}^G)^\ast (P^p(c)\times P^q(d)), x_{p+q} \rangle\\
= & \sum_{p,q\geq 0} \langle (\delta_{p,q}^G)^\ast (P^p(c)\times P^q(d)),\Phi_{p,q}^\ast x_{p+q} \rangle\\
= & D^2(P(c)\times P(d))(\Phi x) = \vartheta^2(\Phi x)(c,d).
\end{align*}
Here, in the third line, we use the observation that there is only one double coset in \[\Sigma_{p+q}\times G\backslash \Sigma_{p+q}\wr G / \Sigma_p\wr G\times \Sigma_q\wr G, \]
and hence the double coset formula for $(\delta_{p+q}^G)^\ast \tr_{\Sigma_p\wr G\times \Sigma_q\wr G}^{\Sigma_{p+q}\wr G}$ consists of a single summand.
\end{proof}

Finally, we consider in which sense this construction is functorial. We can study functoriality both in the global power functor $R$ and in the compact Lie group $G$. 

\begin{definition}\label{def:morphisms_of_tau_rings}
Let $A$ and $A^\prime$ be additive $\beta$-rings with structure morphisms $\vartheta, \vartheta^2$ and $\vartheta^\prime,\vartheta^{\prime2}$. Then a morphism $f\colon A\to A^\prime$ of $\beta$-rings is a ring homomorphism $f$ such that the relations 
\[ f(\vartheta(x)(a)) = \vartheta^\prime (x)(f(a)) \textup{ and } f(\vartheta^2(y)(a_1, a_2)) = \vartheta^{\prime2} (y)(f(a_1), f(a_2)) \]
hold for all $x\in \mathbf B, y\in \mathbf B^2$ and $a, a_1,a_2\in A$.
\end{definition}

\begin{prop}\label{prop:restrictions_induce_tau_ring_morphism}
\begin{enumerate}[i)]
\item Let $G$ be a compact Lie group and $f\colon R\to S$ be a morphism of global power functors with $\mathbb A$-deflations. Then $f(G)$ is a morphism between the $\beta$-rings $R(G)$ and $S(G)$.
\item Let $\varphi\colon K\to G$ be a homomorphism of compact Lie groups and $R$ be a global power functor with $\mathbb A$-deflations. Then $\varphi^\ast$ is a morphism between the $\beta$-rings $R(G)$ and $R(K)$.
\end{enumerate}
\end{prop}
\begin{proof}
We only check the compatibility with $\vartheta$, the calculations for $\vartheta^2$ are similar.\\
For the first assertion, we calculate for $x\in \mathbf B$ and $b\in R(G)$:
\begin{align*}
\vartheta_S(x)(f(G)(b)) = & D_G(P_S(f(G)(b)))(x) = \sum_{n\geq 0} \langle (\delta_n^G)^\ast P^n_S(f(G)(b)),x_n\rangle\\
= & \sum_{n\geq 0}  \langle f(\Sigma_n \times G) (\delta_n^G)^\ast P^n_R(b),x_n\rangle\\
= & \sum_{n\geq 0} f(G) \langle (\delta_n^G)^\ast P^n_R(b),x_n\rangle = f(G)(\vartheta_R(x)(b)).
\end{align*}
For the second assertion, we calculate for $x\in \mathbf B$ and $c\in R(G)$:
\begin{align*}
\vartheta_K(x)(\varphi^\ast(c)) = & \sum_{n\geq 0} \langle (\delta_n^K)^\ast P^n(\varphi^\ast (c)), x_n\rangle \\
= & \sum_{n\geq 0} \langle (\delta_n^K)^\ast (\Sigma_n\wr \varphi)^\ast P^n(c), x_n\rangle \\
= & \sum_{n\geq 0} \langle (\Sigma_n\times \varphi)^\ast (\delta_n^G)^\ast P^n(c), x_n\rangle \\
= & \sum_{n\geq 0} \varphi^\ast\langle (\delta_n^G)^\ast P^n(c), x_n\rangle = \varphi^\ast \vartheta_G(x)(c).
\end{align*}
The corresponding calculations with $\vartheta^2$ work analogously.
\end{proof}

Thus, we have proven the following result, where we denote by $\textup{Rep}$ the category of compact Lie groups and conjugacy classes of continuous group homomorphisms and by $\beta\textup{-Rings}$ the category of additive $\beta$-rings. Moreover, we denote by $\mathcal GlPow_{\mathbb A\textup{-defl}}$ the category of global power functors with $\mathbb A$-deflations.

\begin{thm}\label{thm:powered_algebras_yield_tau_rings}
The assignment $(G, R)\mapsto R(G)$ extends to a functor
\[ \ev\colon \textup{Rep}^{\operatorname{op}} \times \mathcal GlPow_{\mathbb A\textup{-defl}} \to \beta\textup{-Rings}, \]
which sends a conjugacy class of a morphism of compact Lie groups to the corresponding restriction.
\end{thm}

In this theorem, we only treat restrictions. In fact, transfers do not induce morphisms of $\beta$-rings. The reason is that transfers do not commute with the morphism $(\delta_n^G)^\ast$.

To illustrate the theory of $\beta$-rings, we calculate the $\beta$-operations in one example.

\begin{example}\label{example:beta_rings_as_tau_rings}
We apply our theory to the global power functor $\mathbb A$, where $C=\mathbb Z=\mathbb A(e)$. Then we obtain $\beta$-ring structures on the Burnside rings $\mathbb A(G)$ for all compact Lie groups $G$. The operations here are given as follows:\\
For the element $x= \Sigma_n/H \in \mathbf B =\bigoplus_{n\geq 0} \mathbb A(\Sigma_n)$, we obtain for finite $G$ and a finite $G$-set $X$ the formula
\[ \vartheta_H(X) \coloneqq \vartheta_{\Sigma_n/H}(X) = \langle P^n(X), \Sigma_n/H \rangle = \Sigma_n/H \times_{\Sigma_n} X^n = X^n/H,\]
where we consider the resulting set as a $G$-set. This formula agrees with the one from \cite{Vallejo_1990} and generalizes to compact Lie groups as shown in \cite{Rymer_1977}. Thus, in this case, we obtain the classical $\beta$-ring structure on $\mathbb A(G)$, using our abstract definition. Also the iterated operations $\vartheta^2$ used for additivity agree with those defined in \cite[Example 3.2]{Guillot_2006}. In fact, we have
\[ \vartheta_H^2(X,Y)= (X^p\times Y^q)/H\]
for a finite group $G$, finite $G$-sets $X$ and $Y$ and $H\subset \Sigma_p\times \Sigma_q$.
\end{example}

The above construction of $\beta$-ring structures on $R(G)$ highlights the importance of a global point of view. Theorem \ref{thm:powered_algebras_yield_tau_rings} shows that the notion of a global power functor with $\mathbb A$-deflations encodes compatible $\beta$-ring structures for all compact Lie groups at once. In this way, we may approach the still rather mysterious theory of $\beta$-rings from the direction of the well-structured global power functors.\\
The comparison \ref{thm:powered_algebras_yield_tau_rings} is not perfect, however. It remains open to what extent we can represent all $\beta$-rings by global power functors, for example. In general, the condition of having global power operations on a ring is stronger than admitting a $\beta$-ring structure. Also, we require no multiplicative behaviour of the operations $\vartheta$, whereas the power operations of a global power functor are multiplicative. In face of the complications posed in the analysis of $\beta$-rings, starting with finding a feasible definition, it seems sensible to propose that the notion of global power functors is the more fundamental one.

%% file: Appendix_Transferring_Monads.tex
\section{Transferring monads under lax functors}\label{section:monads_under_lax_functors}

In Section \ref{section:adjunctions_between_infinity_ring_spectra}, we study two lifting theorems for functors between algebras over the derived symmetric algebra monad in the global and stable homotopy categories. To separate the homotopy theoretic properties needed to provide the liftings from the formal background in monad theory, it is convenient to use the language of 2-categories. We also treat aspects of double categories, which we use when we encounter both left and right derived functors. For the theory of 2-categories, we refer to \cite{Kelly_Street_1974}, \cite{Street_1972b} and \cite[Chapter 7]{Borceux_1994_1}, for the theory of double categories, we refer to \cite{Kelly_Street_1974} and \cite{Shulman_2011}.

\begin{definition}\label{def:2_and_double_cats}
A 2-category is a category enriched in the category $\operatorname{Cat}$ of categories, and a double category is a category object in $\operatorname{Cat}$.
\end{definition}
Thus, explicitly, a 2-category consists of classes of objects, morphisms and transformations, where we have a horizontal composition $\star$ and a vertical composition $\circ$ of transformations, and compositions of morphisms is strictly unital and associative. For horizontal and vertical composition, we use the conventions
\vspace*{-1em}
\[ (\eta\colon g\to g^\prime)\star (\vartheta\colon f\to f^\prime) =
\begin{tikzcd}[cramped]
X \arrow[r, bend left, "f", ""{name=start theta, below}] \arrow[r, bend right, "f^\prime", swap, ""{name=end theta, above}] \arrow[Rightarrow, "\,\vartheta", from = start theta, to= end theta]
& Y \arrow[r, bend left, "g", ""{name=start eta, below}] \arrow[r, bend right, "g^\prime", swap, ""{name=end eta, above}] \arrow[Rightarrow, "\,\eta", from = start eta, to= end eta]
& Z
\end{tikzcd}
\textup{ and }
(\eta\colon g\to h)\circ (\vartheta\colon f\to g) = 
\begin{tikzcd}
X 	\arrow[r, bend left = 60, "f", ""{name=start theta, below, below=-2pt}]
	\arrow[r, "g" near start, ""{name=end theta}, ""{name=start eta, below=-2pt}]
	\arrow[r, bend right = 60, "h" swap, ""{name=end eta}]
& Y.
\arrow[Rightarrow, "\,\vartheta", from= start theta, to= end theta]
\arrow[Rightarrow, "\,\eta", from= start eta, to= end eta]
\end{tikzcd}
\]
A double category consists of a class of objects, classes of horizontal and vertical morphisms each being part of a category with common objects, and a class of transformations 
\[\begin{tikzcd}
X\arrow[r]\arrow[d] & Y \arrow[d] \arrow[ld, phantom, "\Swarrow"]\\
Z\arrow[r] & W,
\end{tikzcd}\]
also called squares or 2-cells. Transformations can be composed both horizontally and vertically, and all possible orders of composition agree. We denote horizontal composition by $\horizontal$ and vertical composition by $\vertical$. Note that for any double category $\mathbf C$, we obtain two 2-categories $\mathcal V(\mathbf C)$ and $\mathcal H(\mathbf C)$ by considering only vertical morphisms and 2-cells with identities as horizontal morphisms, or considering horizontal morphisms respectively.

For these notions of higher categories, there exist various versions of functors and natural transformations between them. We need the following:

\begin{definition}\label{def:lax_functors_2_cat}
Let $\mathcal C$ and $\mathcal D$ be 2-categories. A lax 2-functor $F\colon \mathcal C\to \mathcal D$ consists of the following data:
\begin{enumerate}[\itshape i)]
\item assignments $X\mapsto F(X)$, $(f\colon X\to Y)\mapsto (F(f)\colon F(X)\to F(Y))$ and $(\eta\colon f\to g)\mapsto (F(\eta)\colon F(f)\to F(g))$ of objects, morphisms and transformations,
\item and transformations $\alpha_X\colon id_{F(X)}\to F(id_X)$ and $\mu_{g,f}\colon F(g)\circ F(f) \to F(gf)$ for any object $X$ and any pair $(g,f)$ of composable morphisms in $\mathcal C$.
\end{enumerate}
These have to satisfy the compatibility conditions given in \cite[Definition 7.5.1]{Borceux_1994_1}.
\end{definition}

\begin{definition}\label{def:lax_transformation_2_cat}
Let $F,G\colon \mathcal C\to \mathcal D$ be two lax functors between 2-categories. A lax natural transformation $\eta\colon F\to G$ between $F$ and $G$ consists of assignments 
\[X\mapsto (\eta_X\colon F(X) \to G(X)) \textup{ and } (f\colon X\to Y)\mapsto \left(
\begin{tikzcd}
F(X) \arrow[r, "\eta_X"] \arrow[d, "Ff", swap] & G(X) \arrow[d, "Gf"] \arrow[ld, phantom, "\Swarrow_{\eta_f}"]\\
F(Y) \arrow[r, "\eta_Y", swap] & G(Y)
\end{tikzcd}\right),\]
such that the compatibility conditions given in \cite[Definition 7.5.2]{Borceux_1994_1} are satisfied.
\end{definition}
For two lax transformations $\eta\colon F\to G$ and $\theta\colon G\to H$ between lax $2$-functors, the composite is given by sending $X$ to the morphism $\theta_X\circ \eta_X$ and a morphism $f\colon X\to Y$ to the transformation 
\[\begin{tikzcd} 
F(X) \arrow[r, "\eta_X"] \arrow[d, "Ff", swap] & G(X) \arrow[d, "Gf"] \arrow[r, "\theta_X"] \arrow[ld, phantom, "\Swarrow_{\eta_f}"] & H(X) \arrow[d, "Hf"] \arrow[ld, phantom, "\Swarrow_{\theta_f}"]\\
F(Y) \arrow[r, "\eta_Y", swap] & G(Y) \arrow[r, "\theta_Y", swap] & H(Y).
\end{tikzcd} \]

We now relate these notions to the theory of monads. First note that the definition of a monad can be given in any 2-category, generalizing an endofunctor $T\colon \mathcal C\to \mathcal C$ to an endomorphism $T\colon X\to X$ of an object $X$ and the multiplication and unitality natural transformations $\mu\colon TT\to T$ and $\eta\colon \Id\to T$ to corresponding transformations. Moreover, we can consider morphisms between such monads, compare \cite[§1]{Street_1972a}.

\begin{definition}\label{def:lax_monad_functor}\label{def:monadic_natural_transformation}
	Let $\mathcal C$ be a 2-category and $(P, \mu, \eta)$ and $(Q, \nu, \eps)$ be monads on objects $X$ and $Y$ of $\mathcal C$ respectively. A (lax) monad morphism is a pair $(F, \rho)$, consisting of a morphism $F\colon X\to Y$ and a transformation $\rho\colon QF\to FP$, such that the diagrams
	\[\begin{tikzcd}
	F  \arrow[r, "\eps F"] \arrow[dr, "F\eta", swap] & QF \arrow[d, "\rho"]\\
	& FP
	\end{tikzcd}
	\hspace{1em}\textup{and}\hspace{1em}
	\begin{tikzcd}
	QQF \arrow[r, "Q\rho"] \arrow[dd, "\nu F", swap] & QFP \arrow[d, "\rho P"]\\
	& FPP \arrow[d, "F\mu"]\\
	QF \arrow[r, "\rho", swap] & FP
	\end{tikzcd}\]
	commute.\\
	Let $(F, \rho), (G, \sigma)\colon X\to Y$ be two monad functors between $P$ and $Q$. Then a monadic transformation between $F$ and $G$ is a transformation $\theta\colon F\to G$ such that $\theta P \circ \rho = \sigma \circ Q\theta$ as transformations $QF\to GP$.
\end{definition}

Note that if a 2-category $\mathcal C$ admits a construction of algebras, i.e. a right adjoint to the inclusion of $\mathcal C$ into the category of monads in $\mathcal C$, then a monad morphism $(F, \rho)\colon P\to Q$ induces a morphism between the corresponding objects of algebras over $P$ and $Q$. In particular, a monad functor in $\mathrm{Cat}$ induces a functor between the categories of algebras, and a monadic transformation a transformation between the induced functors.

The following observation goes back to \cite[5.4.1]{Benabou_1967}:

\begin{lemma}\label{lemma:monad_is_lax_functor}
Let $\mathcal C$ be a $2$-category, and let $1$ be the terminal $2$-category with a single object $\ast$, its identity morphism and the identity natural transformation. Then, the category of lax $2$-functors $1\to \mathcal C$ and lax natural transformations and the category of monads and lax monad morphisms in $\mathcal C$ are isomorphic via the functor
\begin{align*}
(T\colon 1\to \mathcal C) \mapsto &\; (T(id_\ast)\colon T(\ast)\to T(\ast),\: \mu\colon T(id_\ast)\circ T(id_\ast)\to T(id_\ast),\: \eta\colon id_{T(\ast)} \to T(id_\ast))\\
(\rho\colon S\to T) \mapsto & \left(\rho_\ast\colon S(\ast)\to T( \ast),\:\rho(id_\ast): 
\begin{tikzcd}[ampersand replacement=\&]
T(\ast) \arrow[r, "\rho_\ast"] \arrow[d, "S(id_\ast)", swap] \& T(\ast) \arrow[d, "T(id_\ast)"] \arrow[ld, phantom, "\Swarrow"]\\
S(\ast) \arrow[r, "\rho_\ast", swap] \& T(\ast)
\end{tikzcd}
\right).
\end{align*}
\end{lemma}
The proof is an easy translation of the corresponding properties.

Using this description, we see that monads are preserved under any lax 2-functor,  and lax monad morphisms are preserved if we moreover assume that some of the structure maps of a lax functor are invertible.

\begin{corollary}\label{corollary:lax_functor_preserves_monads}
Let $\mathcal C$ and $\mathcal D$ be $2$-categories and let $F\colon \mathcal C\to \mathcal D$ be a lax $2$-functor. Let $(T\colon X\to X, \nu\colon T\circ T\to T, \eps\colon id_X\to T)$ be a monad in $\mathcal C$. Then
\[ (F(T)\colon F(X)\to F(X),\: F(\nu)\circ \mu_{T,T}\colon FT\circ FT\to FT,\: F(\eps)\circ \eta_X\colon id_{FX}\to FT)\]
is a monad in $\mathcal D$.\\
Moreover, let $S,\, T$ be two monads in $\mathcal C$ on objects $X$ and $Y$ respectively, let $(f\colon X\to Y, \rho\colon Tf\to fS)$ be a lax monad morphism, and assume that the transformation $\mu_{f,S}$ is invertible. Then $(Ff\colon FX\to FY, \mu_{f,S}\inverse\circ F(\rho)\circ \mu_{T,f}\colon FT\circ Ff\to Ff\circ FS)$ is a lax monad morphism between $FS$ and $FT$.
\end{corollary}
\begin{proof}
The first part of this corollary follows directly from the above lemma: We can consider the monad $T$ in $\mathcal C$ as a lax 2-functor $T\colon 1\to \mathcal C$. Then, the composition $F\circ T\colon 1\to \mathcal D$ is a lax 2-functor, with coherence morphisms the composites of the coherence morphisms of $F$ and $T$. Thus, $FT$ is a monad in $\mathcal D$, and the structure of the monad is exactly given by the described morphism.\\
For the second part, we note that the above lemma shows that a lax monad morphism from $S$ to $T$ is the same as a lax natural transformation between the corresponding lax 2-functors $1\to \mathcal C$. It is an easy argument that a lax 2-functor with invertible transformation $\mu_{f,S}$ preserves such a transformation.
\end{proof}

We now consider the double categorical context. This is used in Section \ref{section:adjunctions_between_infinity_ring_spectra} in order to handle the occurrence of both right and left derived functors. These different types of functors can conveniently be handled by assigning them as vertical and horizontal morphisms of a double category, respectively.\\
In a double category $\mathbf C$, we use a similar formalism as for 2-categories to consider whether a corresponding notion of weak double functor preserves monads and morphisms between them. We thus first define the appropriate notion of a weak double functor.

\begin{definition}\label{def:lax_oplax_double_functor}
Let $\mathbf C$ and $\mathbf D$ be double categories. A lax-oplax double functor $F\colon \mathbf C\to \mathbf D$ consists of assignments of objects, horizontal 1-cells, vertical 1-cells and 2-cells of $\mathbf D$ to those of $\mathbf C$, and the following coherence data:
\begin{enumerate}[\itshape i)]
\item Invertible unitality 2-cells
\[\begin{tikzcd}
FX \arrow[r, "F(id_X)"] \arrow[d, equals] & FX \arrow[d, equals]\arrow[ld, phantom, "\Swarrow_{\alpha_X^h}"]\\
FX \arrow[r, equals, "id_{F(X)}", swap] & FX
\end{tikzcd}
\hspace{1em} \textup{and} \hspace{1em}
\begin{tikzcd}
FX \arrow[r, equals] \arrow[d, "F(id_X)", swap] & FX \arrow[d, equals, "id_{F(X)}"]\arrow[ld, phantom, "\Swarrow_{\alpha_X^v}"]\\
FX \arrow[r, equals] & FX
\end{tikzcd}\]
for any object $X$ of $\mathbf C$.
\item Composition 2-cell
\[\begin{tikzcd}
FX \arrow[rr, "F(gf)"] \arrow[d, equals] && FZ \arrow[d, equals] \arrow[lld, phantom, "\Swarrow_{\mu_{g,f}^h}"]\\
FX \arrow[r, "Ff", swap] & FY \arrow[r, "Fg", swap] & FZ
\end{tikzcd}
\hspace{1em} \textup{and} \hspace{1em}
\begin{tikzcd}
	FX \arrow[r, equals] \arrow[dd, "F(gf)", swap] & FX \arrow[d, "Ff"] \arrow[ldd, phantom, "\Swarrow_{\mu_{g,f}^v}"]\\
	& FY \arrow[d, "Fg"]\\
	FZ \arrow[r, equals] & FZ
\end{tikzcd}\]
for composable pairs $X\nameto{f} Y\nameto{g}Z$ of horizontal and vertical morphisms, respectively.
\end{enumerate}
These coherence cells need to satisfy the unitality, associativity and naturality relations as written down in \cite[Definition 6.1 {\itshape v)} and {\itshape vi)}]{Shulman_2011}.
\end{definition}
\begin{remark}\label{remark:double_functor_definition}
Note that in \cite{Shulman_2011}, the direction of the vertical structure 2-cells is reversed. In that work, all of the above are assumed to be isomorphisms, so the direction of the cells is irrelevant. In our application, the 2-cells $\mu_{g,f}^v$ are not invertible in general, so we have to take care of the orientation. We choose the given convention since deriving (vertical) left derivable functors comes endowed with an oplax structure, and deriving (horizontal) right derivable functors with a lax structure.\\
On the other hand, the unitality 2-cells $\alpha_X^h$ and $\alpha_X^v$ are assumed to be invertible. This allows us to obtain from a lax-oplax double functor a lax 2-functor $\mathcal V(F)\colon \mathcal V(\mathbf C)\to \mathcal V(\mathbf D)$ by applying $F$ to vertical morphisms, and by defining
\[ \mathcal V(F) \left( \begin{tikzcd}
X \arrow[r, equals] \arrow[d, "f", swap] & X \arrow[d, "g"] \arrow[ld, phantom, "\Swarrow_{\eta}"]\\
Y \arrow[r, equals] & Y
\end{tikzcd}\right)
= \begin{tikzcd}[row sep = large]
FX \arrow[r, equals] \arrow[d, equals] & FX \arrow[d, equals]\arrow[ld, phantom, "\Swarrow_{(\alpha_X^h)\inverse}"] \\
FX \arrow[r, "F(id_X)"] \arrow[d, "Ff", swap] & FX \arrow[d, "Fg"] \arrow[ld, phantom, "\Swarrow_{F\eta}"]\\
FY \arrow[r, "F(id_Y)", swap] \arrow[d, equals] & FY \arrow[d, equals]\arrow[ld, phantom, "\Swarrow_{\alpha_X^h}"]\\
FY\arrow[r, equals] & FY.
\end{tikzcd}\]
In the same way, we obtain an oplax 2-functor $\mathcal H(F)\colon \mathcal H(\mathbf C)\to \mathcal H(\mathbf D)$.
\end{remark}

In fact, the constraint that all $\alpha_X^h$ are invertible can be used to strictify $F$ into a lax-oplax functor where $\alpha_X^h=id_X$ holds. The main result is the following:

\begin{lemma}\label{lemma:shrinking_of_unitors_for_double_functor}
Let $\mathbf C$ and $\mathbf D$ be double categories and $F\colon \mathbf C\to \mathbf D$ be a lax-oplax double functor. Let 
\[\begin{tikzcd}
X\arrow[r, equals] \arrow[d, "f", swap] & X\arrow[d, "g"] \arrow[ld, phantom, "\Swarrow_{\theta}"]\\
Y \arrow[r, equals] & Y
\end{tikzcd} 
\hspace{1em} \textrm{ and } \hspace{1em}
\begin{tikzcd}
X \arrow[r, "h"] \arrow[d, "g", swap] & X^\prime \arrow[d, "g^\prime"] \arrow[ld, phantom, "\Swarrow_{\eta}"]\\
Y \arrow[r, "k", swap] & Y^\prime
\end{tikzcd}\]
be $2$-cells in $\mathbf C$. Then, the $2$-cells
\[\begin{tikzcd}
FX \arrow[r, "Fh"] \arrow[d, "Ff", swap] & FX^\prime \arrow[d, "Fg^\prime"] \arrow[ld, phantom, "\Swarrow_{F(\theta\horizontal \eta)}"]\\
FY \arrow[r, "Fk", swap] & FY^\prime
\end{tikzcd} 
\hspace{1em} \textrm{ and } \hspace{1em}
\begin{tikzcd}
FX \arrow[r, equals] \arrow[d, equals]
	& FX \arrow[r, "Fh"] \arrow[d, equals] \arrow[ld, phantom, "\Swarrow_{(\alpha_X^h)\inverse}"]
	& FX^\prime \arrow[ddd, "Fg^\prime"] \arrow[lddd, phantom, "\Swarrow_{F(\eta)}"]\\
FX \arrow[r, "F(id_X)"] \arrow[d, "Ff", swap] 
	& FX \arrow[d, "Fg"] \arrow[ld, phantom, "\Swarrow_{F(\theta)}"]\\
FY \arrow[r, "F(id_Y)", swap] \arrow[d, equals] 
	& FY \arrow[d, equals] \arrow[ld, phantom, "\Swarrow_{\alpha_Y^h}"]\\
FY \arrow[r, equals]
	& FY\arrow[r, "Fk", swap]
	& FY^\prime 
\end{tikzcd}\]
in $\mathbf D$ agree. The analogous statement holds for $F(\eta\horizontal \theta)$.
\end{lemma}
\begin{proof}
We use the following chain of pasting diagrams:
\begin{align*}
\begin{tikzcd}[ampersand replacement=\&]
FX \arrow[r, equals] \arrow[d, equals]
	\& FX \arrow[r, "Fh"] \arrow[d, equals] \arrow[ld, phantom, "\Swarrow_{(\alpha_X^h)\inverse}"]
	\& FX^\prime \arrow[ddd, "Fg^\prime"] \arrow[lddd, phantom, "\Swarrow_{F(\eta)}"]\\
FX \arrow[r, "F(id_X)"] \arrow[d, "Ff", swap] 
	\& FX \arrow[d, "Fg"] \arrow[ld, phantom, "\Swarrow_{F(\theta)}"]\\
FY \arrow[r, "F(id_Y)", swap] \arrow[d, equals] 
	\& FY \arrow[d, equals] \arrow[ld, phantom, "\Swarrow_{\alpha_Y^h}"]\\
FY \arrow[r, equals]
	\& FY\arrow[r, "Fk", swap]
	\& FY^\prime 
\end{tikzcd} =
	& \begin{tikzcd}[ampersand replacement=\&]
	FX \arrow[r, equals] \arrow[d, equals]
		\& FX \arrow[r, "Fh"] \arrow[d, equals] \arrow[ld, phantom, "\Swarrow_{(\alpha_X^h)\inverse}"]
		\& FX^\prime \arrow[d, equals] \arrow[ld, phantom, "\Swarrow_{id_{Fh}}"]\\
	FX \arrow[r, "F(id_X)"] \arrow[d, "Ff", swap] 
		\& FX \arrow[d, "Fg"] \arrow[r, "Fh"] \arrow[ld, phantom, "\Swarrow_{F(\theta)}"]
		\& FX^\prime \arrow[d, "Fg^\prime"] \arrow[ld, phantom, "\Swarrow_{F(\eta)}"]\\
	FY \arrow[r, "F(id_Y)", swap] \arrow[d, equals] 
		\& FY \arrow[d, equals] \arrow[r, "Fk", swap] \arrow[ld, phantom, "\Swarrow_{\alpha_Y^h}"]
		\& FY^\prime \arrow[d, equals] \arrow[ld, phantom, "\Swarrow_{id_{Fk}}"]\\
	FY \arrow[r, equals]
		\& FY\arrow[r, "Fk", swap]
		\& FY^\prime 
	\end{tikzcd}\\
= \begin{tikzcd}[ampersand replacement=\&]
	FX \arrow[rr, "Fh"] \arrow[d, equals]
		\& \& FX^\prime \arrow[d, equals] \arrow[lld, phantom, "\Swarrow_{\mu_{h, id}^h}"]\\
	FX \arrow[r, "F(id_X)"] \arrow[d, "Ff", swap] 
		\& FX \arrow[d, "Fg"] \arrow[r, "Fh"] \arrow[ld, phantom, "\Swarrow_{F(\theta)}"]
		\& FX^\prime \arrow[d, "Fg^\prime"] \arrow[ld, phantom, "\Swarrow_{F(\eta)}"]\\
	FY \arrow[r, "F(id_Y)", swap] \arrow[d, equals] 
		\& FY \arrow[r, "Fk", swap]
		\& FY^\prime \arrow[d, equals] \arrow[lld, phantom, "\Swarrow_{(\mu_{k, id}^h)\inverse}"]\\
	FY \arrow[rr, "Fk", swap]
		\&\& FY^\prime 
	\end{tikzcd}
=& \begin{tikzcd}[ampersand replacement = \&]
FX \arrow[r, "Fh"] \arrow[d, "Ff", swap] \& FX^\prime \arrow[d, "Fg^\prime"] \arrow[ld, phantom, "\Swarrow_{F(\theta\horizontal \eta)}"]\\
FY \arrow[r, "Fk", swap] \& FY^\prime.
\end{tikzcd}
\end{align*}
Here, we used the unitality and naturality conditions on a lax-oplax double functor in the second and third step respectively. Also note that the unitality condition guarantees that the transformation $\mu_{k, id}^h$ is indeed invertible.
\end{proof}

Now, we define the relevant notions of monads and morphisms between them in a double category. In our application, we have a left derivable monad and a right derivable monad morphism, and this motivates the following definition. Moreover, we also define monadic transformations between monad morphisms in this context.
\begin{definition}\label{def:monad_in_double_category}
Let $\mathbf C$ be a double category. A vertical monad $T$ in $\mathbf C$ is a monad in the vertical 2-category $\mathcal V(\mathbf C)$. A horizontal monad morphism between two vertical monads $S$ and $T$ on objects $X$ and $Y$ respectively is a horizontal morphism $F\colon X\to Y$ together with a 2-cell 
\[\begin{tikzcd}
X\arrow[r, "F"] \arrow[d, "S", swap] & Y \arrow[d, "T"] \arrow[ld, phantom, "\Swarrow_{\rho}"]\\
X\arrow[r, "F", swap] & Y,
\end{tikzcd}\]
satisfying the unitality and mulitplicativity conditions
\[ \begin{tikzcd}
X\arrow[r, "F"] \arrow[d, "S", swap]
	& Y \arrow[d, "T"] \arrow[r, equals] \arrow[ld, phantom, "\Swarrow_{\rho}"]
	& Y \arrow[d, equals] \arrow[ld, phantom, "\Swarrow_{\eta_T}"]\\
X\arrow[r, "F", swap]
	& Y \arrow[r, equals] 
	& Y
\end{tikzcd}
=  
\begin{tikzcd}
X\arrow[d, "S", swap] \arrow[r, equals] 
	& X\arrow[d, equals] \arrow[r, "F"] \arrow[ld, phantom, "\Swarrow_{\eta_S}"]
	& Y \arrow[d, equals] \arrow[ld, phantom, "\Swarrow_{id_{F}}"]\\
X \arrow[r, equals]
	& X\arrow[r, "F", swap] 
	& Y
\end{tikzcd}\]
and
\[\begin{tikzcd}
X \arrow[r, equals] \arrow[dd, "S", swap] 
	& X \arrow[d, "S", swap] \arrow[r, "F"] \arrow[ldd, phantom, "\Swarrow_{\mu_S}"]
	& Y \arrow[d, "T"] \arrow[ld, phantom, "\Swarrow_\rho"]\\
& X \arrow[r, "F"] \arrow[d, "S", swap] 
	& Y\arrow[d, "T"] \arrow[ld, phantom, "\Swarrow_\rho"]\\
X \arrow[r, equals] 
	& X \arrow[r, "F", swap] 
	& Y
\end{tikzcd}
=  
\begin{tikzcd}
X \arrow[r, "F"] \arrow[dd, "S", swap]
	& Y\arrow[dd, "T"] \arrow[r, equals] \arrow[ldd, phantom, "\Swarrow_\rho"] 
	& Y \arrow[d, "T"] \arrow[ldd, phantom, "\Swarrow_{\mu_{T}}"]\\
&& Y \arrow[d, "T"]\\
X \arrow[r, "F", swap] 
	& Y \arrow[r, equals]
	& Y.
\end{tikzcd}\]
For two horizontal monad morphisms $(F, \rho)$ and $(G, \sigma)$ between $S$ and $T$, a monadic transformation is a 2-cell
\[ \begin{tikzcd}
X \arrow[r, "F"] \arrow[d, equals] & Y \arrow[d, equals] \arrow[ld, phantom, "\Swarrow_{\eta}"]\\
X \arrow[r, "G", swap] & Y
\end{tikzcd}
\textup{ such that }
\begin{tikzcd}
X \arrow[r, "F"] \arrow[d, "S", swap] & Y \arrow[d, "T"] \arrow[ld, phantom, "\Swarrow_{\rho}"]\\
X \arrow[r, "F"] \arrow[d, equals] & Y \arrow[d, equals] \arrow[ld, phantom, "\Swarrow_{\eta}"]\\
X \arrow[r, "G", swap] & Y
\end{tikzcd}
= 
\begin{tikzcd}
X \arrow[r, "F"] \arrow[d, equals] & Y \arrow[d, equals] \arrow[ld, phantom, "\Swarrow_{\eta}"]\\
X \arrow[r, "G"] \arrow[d, "S", swap] & Y \arrow[d, "T"] \arrow[ld, phantom, "\Swarrow_{\sigma}"]\\
X \arrow[r, "G", swap] & Y
\end{tikzcd}\]
holds.
\end{definition}

\begin{prop}\label{prop:lax_double_functor_preserves_monads}
Let $\mathbf C$ and $\mathbf D$ be two double categories and let $L\colon \mathbf C\to \mathbf D$ be a lax-oplax double functor. Let $(S, \mu, \eta)$ be a vertical monad in $\mathbf C$. Then
\[(\mathcal V(L)(S), \mathcal V(L)(\mu)\circ \mu_{S,S},\mathcal V(L)(\eta) \circ \alpha_X^v)\]
is a vertical monad in $\mathbf D$.\\
Moreover, let $S$ and $T$ be vertical monads in $\mathbf C$ and let $(F,\rho)$ be a horizontal monad morphism between them. Then $(LF, L\rho)$ is a horizontal monad morphism between $\mathcal V(L)(S) $ and $\mathcal V(L)(T)$.\\
Furthermore, for any monadic transformation $\eta\colon F\to G$ between two monad morphisms, the natural transformation $\mathcal H(L)(\eta)$ is a monadic transformation between $LF$ and $LG$.
\end{prop}
\begin{proof}
The first part is a direct consequence of \eqref{corollary:lax_functor_preserves_monads}, since a vertical monad is a monad in the vertical $2$-category $\mathcal V(\mathbf C)$ and $L$ induces a lax 2-functor $\mathcal V(L)\colon \mathcal V(\mathbf C)\to \mathcal V(\mathbf D)$.

We now prove the second part. We check the unitality condition on $(LF, L\rho)$, and thus consider
\begin{align*}
\begin{tikzcd}[ampersand replacement = \&]
LX \arrow[r, "LF"] \arrow[ddd, "LS", swap] 
	\& LY \arrow[r, equals] \arrow[d, equals] \arrow[lddd, phantom, "\Swarrow_{L\rho}"]
	\& LY \arrow[r, equals] \arrow[d, equals] \arrow[ld, phantom, "\Swarrow_{\alpha_Y\inverse}"]
	\& LY \arrow[ddd, equals] \arrow[lddd, phantom, "\Swarrow_{\alpha_Y}" below]\\
 \& LY \arrow[r, "L(id_Y)"] \arrow[d, "LT", swap]
	 \& LY \arrow[d, "L(id_Y)"] \arrow[ld, phantom, "\Swarrow_{L(\eta_T)}"]\\
 \& LY \arrow[r, "L(id_Y)", swap] \arrow[d, equals]
	 \& LY \arrow[d, equals] \arrow[ld, phantom, "\Swarrow_{\alpha_Y}"]\\
LX\arrow[r, "LF", swap]
	\& LY \arrow[r, equals]
	\& LY \arrow[r, equals]
	\& LY
\end{tikzcd}
& =\\
= \begin{tikzcd}[ampersand replacement = \&, column sep = large]
LX \arrow[r, "LF"] \arrow[d, "LS", swap] 
	\& LY \arrow[r, equals] \arrow[d, "L(id_Y)"] \arrow[ld, phantom, "\Swarrow_{L(\rho\horizontal \eta_T)}"]
	\& LY \arrow[d, equals] \arrow[ld, phantom, "\Swarrow_{\alpha_Y}"right]\\
LX\arrow[r, "LF", swap]
	\& LY \arrow[r, equals]
	\& LY
\end{tikzcd} 
&= \begin{tikzcd}[ampersand replacement = \&, column sep = large]
LX \arrow[r, "LF"] \arrow[d, "LS", swap] 
	\& LY \arrow[r, equals] \arrow[d, "L(id_Y)"] \arrow[ld, phantom, "\Swarrow_{L(\eta_S\horizontal id_F)}"]
	\& LY \arrow[d, equals] \arrow[ld, phantom, "\Swarrow_{\alpha_Y}" right]\\
LX\arrow[r, "LF", swap]
	\& LY \arrow[r, equals]
	\& LY
\end{tikzcd}\\
= \begin{tikzcd}[ampersand replacement = \&]
LX\arrow[r, equals] \arrow[d, equals]
	\& LX \arrow[r, "LF"] \arrow[d, equals] \arrow[ld, phantom, "\Swarrow_{\alpha_X\inverse}"]
	\& LY \arrow[r, equals] \arrow[ddd, "L(id_Y)"] \arrow[lddd, phantom, "\Swarrow_{L(id_F)}" below]
	\& LY \arrow[ddd, equals] \arrow[lddd, phantom, "\Swarrow_{\alpha_Y}" below]\\
LX \arrow[r, "LF"] \arrow[d, "LS", swap]
	\& LX \arrow[d, "L(id_X)"] \arrow[ld, phantom, "\Swarrow_{L(\eta_S)}"]\\
LX \arrow[r, "LF", swap] \arrow[d, equals]
	\& LX  \arrow[d, equals] \arrow[ld, phantom, "\Swarrow_{\alpha_X}"]\\
LX\arrow[r, equals]
	\& LX \arrow[r, "LF", swap]
	\& LY \arrow[r, equals]
	\& LY
\end{tikzcd}
& = \begin{tikzcd}[ampersand replacement = \&]
LX\arrow[r, equals] \arrow[d, equals]
	\& LX \arrow[r, equals] \arrow[d, equals] \arrow[ld, phantom, "\Swarrow_{\alpha_X\inverse}"]
	\& LX \arrow[r, "LF"] \arrow[ddd, "L(id_Y)"] \arrow[lddd, phantom, "\Swarrow_{\alpha_X}" below]
	\& LY \arrow[ddd, equals] \arrow[lddd, phantom, "\Swarrow_{id_{LF}}" below]\\
LX \arrow[r, "LF"] \arrow[d, "LS", swap]
	\& LX \arrow[d, "L(id_X)"] \arrow[ld, phantom, "\Swarrow_{L(\eta_S)}"]\\
LX \arrow[r, "LF", swap] \arrow[d, equals]
	\& LX  \arrow[d, equals] \arrow[ld, phantom, "\Swarrow_{\alpha_X}"]\\
LX\arrow[r, equals]
	\& LX \arrow[r, equals]
	\& LX \arrow[r, "LF", swap]
	\& LY.
\end{tikzcd}
\end{align*}
Here, we use Lemma \ref{lemma:shrinking_of_unitors_for_double_functor} in the first and third step.\\
The multiplicativity condition is obtained by similar pasting diagrams. Thus $(LF, L( \rho))$ is a horizontal monad morphism.\\
The fact that $\mathcal H(L)(\eta)$ is a monadic transformation between $LF$ and $LG$ is proven in the same way, using a horizontal version of \eqref{lemma:shrinking_of_unitors_for_double_functor} for the exchange relation.
\end{proof}

Using this proposition, we also consider how the structure transformations of a lax-oplax double functor behave for a composite of monad functors.

\begin{lemma}\label{lemma:oplax_structure_map_is_monadic}
Let $\mathbf C$ and $\mathbf D$ be double categories, $R\colon X\to X,\, S\colon Y\to Y$ and $T\colon Z\to Z$ be vertical monads in $\mathbf C$ and let $(F,\rho)$ be a horizontal monad morphism from $R$ to $S$ and $(G, \sigma)$ be a horizontal monad morphism from $S$ to $T$. Let moreover $L\colon \mathbf C\to \mathbf D$ be a lax-oplax double functor. Then the structure maps $\mu_{G,F}^h \colon L(GF)\to LG \circ LF$ and $\alpha_X^h\colon L(id_X)\to id_{LX}$ are monadic transformations in $\mathbf D$.
\end{lemma}
The proof follows easily using the naturality constrains of a lax-oplax double functor, see \cite[6.2]{Shulman_2011}.